\documentclass[12pt,english]{article}
\usepackage[T1]{fontenc}
\usepackage[latin9]{inputenc}
\usepackage{geometry}
\usepackage{fancyhdr}
\usepackage{verbatim}
\usepackage{color}
\usepackage{mathrsfs}
\usepackage{amsmath}
\usepackage{amsthm}
\usepackage{amssymb}
\usepackage[misc]{ifsym}
\usepackage{hyperref}
\usepackage{stmaryrd}
\makeatletter
\theoremstyle{plain}
\newtheorem{theorem}{\protect\theoremname}[section] 
\theoremstyle{definition}
\newtheorem{definition}[theorem]{\protect\definitionname}
\theoremstyle{plain}
\newtheorem{lemma}[theorem]{\protect\lemmaname}
\theoremstyle{remark}
\newtheorem{remark}[theorem]{\protect\remarkname}
\theoremstyle{plain}
\newtheorem{corollary}[theorem]{\protect\corollaryname}
\theoremstyle{plain}
\newtheorem{proposition}[theorem]{\protect\propositionname}
\theoremstyle{plain}
\newtheorem{example}[theorem]{\protect\examplename}
\theoremstyle{plain}

\@ifundefined{date}{}{\date{}}
\usepackage{fancyhdr}

\usepackage{babel}
\providecommand{\definitionname}{Definition}
\providecommand{\lemmaname}{Lemma}
\providecommand{\theoremname}{Theorem}
\providecommand{\corollaryname}{Corollary}
\providecommand{\remarkname}{Remark}
\providecommand{\propositionname}{Proposition}
\providecommand{\examplename}{Example}
\providecommand{\problemname}{Problem}

\numberwithin{equation}{section}  

\makeatother
\def\R{{\mathbb R}}
\def\rn{{{\R}^n}}
\def\D{\mathcal{D}}
\def\d{\mathrm{d}}
\def\M{\mathcal{M}}
\def\Mit{\mathcal{M}^{\operatorname{it}}}
\def\H{\mathcal{H}}
\def\F{\mathcal{F}}
\def\sgn{\operatorname{sgn}}

\begin{document}
	\title
	{\bf\Large
		On the mixed Bourgain-Morrey spaces 
		\footnotetext{Jingshi Xu is supported by the National Natural Science Foundation of China (Grant No. 12161022) and the Science and Technology Project of Guangxi (Guike AD23023002).
		Pengfei Guo is supported by the National Natural Science Foundation of China (12561002).
	}
		}
	
	\date{}
	
	\author{Tengfei Bai\textsuperscript{a}, Pengfei Guo\textsuperscript{a},  Jingshi Xu\textsuperscript{b,c,d}\footnote{Corresponding author, E-mail: jingshixu@126.com}  \\
		{\scriptsize \textsuperscript{a} \scriptsize College of Mathematics and Statistics, Hainan Normal University, Haikou, Hainan 571158,
			China}\\
		{\scriptsize  \textsuperscript{b} School of Mathematics and Computing Science, Guilin University of Electronic Technology, Guilin 541004, China} \\
		{\scriptsize  \textsuperscript{c} Center for Applied Mathematics of Guangxi (GUET), Guilin 541004, China}\\
		{\scriptsize  \textsuperscript{d} Guangxi Colleges and Universities Key Laboratory of Data Analysis and Computation, Guilin 541004, China}
	}
	
	\pagestyle{myheadings}\markboth{\footnotesize\rm\sc Tengfei Bai, Pengfei Guo,  Jingshi Xu}
	{\footnotesize\rm\sc  }
	
\maketitle
\begin{abstract}
We introduce  the mixed Bourgain-Morrey spaces and 
obtain their preduals.
  The boundedness of Hardy-Littlewood maximal operator, iterated maximal operator, fractional integral operator, singular integral operator on these spaces is proved. The Littlewood-Paley theory for mixed Bourgain-Morrey spaces and their preduals are established. As  applications, we consider wavelet  characterizations for  mixed Bourgain-Morrey spaces and  a fractional chain rule in  mixed Bourgain-Morrey Triebel-Lizorkin spaces. In addition, we give a description of the dual of mixed Bourgain-Morrey spaces and  conclude the reflexivity of these spaces.
\end{abstract}

\tableofcontents

\section{Introduction}
In 1961, Benedek and Panzone \cite{BP61} introduced Lebesgue spaces with mixed norm. Bagby \cite{B75}
extended the Fefferman-Stein inequality to the mixed Lebesgue spaces.
In \cite{Ig86}, Igari researched the  interpolation theory for  linear operators in  mixed Lebesgue spaces and apply them to Fourier analysis.
Torres  and Ward considered Leibniz's rule, sampling and wavelets of mixed Lebesgue spaces  in \cite{TW15}.
In \cite{AIV19}, Antoni{\'c},  Ivec  and Vojnovi{\'c}  obtained a general framework for dealing with continuity of linear operators on mixed-norm Lebesgue spaces and showed the boundedness of a large class of pseudodifferential operators on these spaces.
The Hardy-Littlewood-Sobolev inequality on mixed-norm Lebesgue spaces was studied by Chen and Sun in \cite{CS22}.

Then  numerous other function spaces with mixed norms
 are developed. 
  Besov spaces with mixed norms  were  studied by in   \cite{BIN78, BIN79}. Triebel-Lizorkin spaces with  mixed norms were considered by  Besov et al. in \cite{BIN96}.  
In 1977,  the Lorentz spaces with mixed norms were first introduced by Fernandez  \cite{Fe77}.
An interpolation result on these spaces was got by Milman  \cite{Mi81}. 
 Banach function spaces with mixed norms were studied by  Blozinski   in \cite{Bl81}. 
Anisotropic mixed-norm {Hardy} spaces were  explored by Cleanthous et al. in  \cite{CGN17}.
Mixed-norm $\alpha$-modulation spaces were introduced by Cleanthous and Georgiadis in \cite{CG20}. 
The mixed Lebesgue spaces with variable exponents were studied by Ho in \cite{Ho18}. He obtained the boundedness of operators, such as the  Calder\'on-Zygmund operators on product domains, the Littlewood-Paley operators associated with family of disjoint rectangles,  the nontangential maximal function.  We refer the reader to the survey \cite{HY21} for the development of mixed norm spaces.

Recently,  
 there are new development of mixed norm spaces. For example, 
in \cite{WYY22}, Wu, Yang and Yuan studied interpolation in mixed-norm function spaces, including mixed-norm Lebesgue spaces, mixed-norm Lorentz spaces and mixed-norm Morrey spaces.
In \cite{Ho16},  Ho obtained the boundedness of the strong maximal operator on mixed-norm spaces. 
In \cite{ZX20}, Zhang and Xue introduced the  classes of multiple weights and of multiple fractional weights. They obtained the boundedness of multilinear strong maximal operators and  multilinear fractional strong maximal operators on weighted mixed norm spaces. 
In \cite{Ho21}, Ho established an extrapolation theory to mixed norm spaces.

In \cite{HLYY20}, Huang et al. introduced the anisotropic mixed-norm Hardy spaces  associated with  a general expansive matrix on  $\rn$ and established their radial or non-tangential maximal function characterizations. 
They obtained  characterizations of  anisotropic mixed-norm Hardy spaces by means of atoms, finite atoms, Lusin area functions, Littlewood-Paley  $g-$functions. The duality between anisotropic mixed-norm Hardy space and anisotropic mixed-norm Campanato space is also obtained. 
In \cite{HCY21}, Huang, Chang and Yang  researched the Fourier transform on anisotropic mixed-norm Hardy space. In \cite{HYY21}, Huang, Yang and Yuan  introduced  anisotropic mixed-norm Campanato-type space associated with a general expansive matrix on $\rn$. They proved that the Campanato-type space is the dual space of the anisotropic mixed-norm Hardy space.

More results can be founded in 
 \cite{BKPS06, Bl81, Fe77} for mixed Lorentz spaces, \cite{CGN19,GN16, JS08} for mixed Besov spaces and Triebel-Lizorkin spaces, \cite{DGZ24, DZ23, Hu23, Li22} for mixed Hardy spaces.
 
 Now we turn to the Morrey spaces.
In \cite{Mo38}, Morrey introduced the origin Morrey spaces  to research the  boundedness of the elliptic differential operators. We refer the reader to the  monographs \cite{SDH20, SDH202} for the theory of Morrey spaces. The mixed Morrey spaces were first introduced by  Nogayama in \cite{N19}.
He obtained the  boundedness of iterated
maximal operator,  fractional integral operator and singular integral operator on mixed Morrey spaces. 
 Nogayama \cite{N192} studied the predual spaces of mixed Morrey spaces and obtained the necessary and sufficient conditions for the boundedness of commutators of fractional integral operators on mixed Morrey spaces.  
 In \cite{IST15}, Izumi, Sawano and Tanaka researched the  Littlewood-Paley theory in Morrey spaces. Motivated by \cite{IST15}, 
  Nogayama \cite{N24} considered the Littlewood-Paley characterization  for mixed Morrey
spaces and their predual spaces. As an application, he showed the wavelet
characterization for mixed Morrey spaces.

Next we recall the Bourgain-Morrey spaces.
In 1991, Bourgain \cite{Bou91} introduced a special case of Bourgain-Morrey spaces to study the Stein-Tomas (Strichartz) estimate.
After then, many authors began to pay attention to the Bourgain-Morrey spaces. For example,
in \cite{M16}, Masaki showed the block spaces which are the  preduals of Bourgain-Morrey spaces. 
In \cite{HNSH23}, Hatano et al. researched the Bourgain-Morrey spaces from the viewpoints of harmonic analysis and functional analysis.    After then, some function spaces extending Bourgain-Morrey spaces were established.

In \cite{ZSTYY23}, Zhao et al.  introduced Besov-Bourgain-Morrey spaces which connect Bourgain-Morrey spaces with amalgam-type spaces.   They established an equivalent norm with an integral expression  and  showed the boundedness on these spaces of the Hardy-Littlewood maximal operator, the fractional integral, and the Calder\'on-Zygmund operator. The preduals, dual spaces and complex interpolations  of these spaces were also obtained.

Immediately after \cite{ZSTYY23},
Hu, et al.  introduced  Triebel-Lizorkin-Bourgain-Morrey spaces which connect Bourgain-Morrey spaces and global Morrey spaces in \cite{HLY23}. 
The embedding relations between Triebel-Lizorkin-Bourgain-Morrey spaces and Besov-Bourgain-Morrey spaces are proved.  
Various fundamental real-variable properties of these spaces are obtained. The sharp boundedness of  the Hardy-Littlewood maximal operator, the Calder\'on-Zygmund operator, and the fractional integral on these spaces was also proved.

Inspired by the generalized grand Morrey spaces and Besov-Bourgain-Morrey spaces, Zhang et al. introduced generalized grand Besov-Bourgain-Morrey spaces in \cite{ZYZ24}. The preduals  and the Gagliardo-Peetre  interpolation theorem,  extrapolation theorem were obtained.  The boundedness of Hardy-Littlewood maximal operator, the fractional integral and the Calder\'on-Zygmund operator on these spaces was also proved. 

The first author and the third author of this paper  introduced the weighted Bourgain-Morrey-Besov-Triebel-Lizorkin spaces associated with operators over a  space of homogeneous type in \cite{BX25} and  obtained two sufficient conditions for precompact sets in matrix weighted Bourgain-Morrey spaces in \cite{BX252}.

Motivated by the above literature, we will study the mixed Bourgain-Morrey spaces and their preduals from the viewpoints of harmonic analysis and functional analysis. The paper is organized as follows.
In Section \ref{preliminaries}, we give the definition of  mixed Bourgain-Morrey spaces and recall some lemmas about mixed Lebesgue spaces.
In Section \ref{sec property}, we research the properties of mixed Bourgain-Morrey spaces, such as embedding, dilation, translation, nontriviality, approximation, density.
In Section \ref{sec predual}, we introduce the mixed norm block spaces and show they are the preduals of mixed Bourgain-Morrey spaces.
In Section \ref{property block}, we consider the properties of the block spaces, such as completeness, density, Fatou property, lattice property. We also show the  duality of  mixed Bourgain-Morrey.
In Section \ref{sec operator}, we consider the boundedness of operators on (sequenced valued) mixed Bourgain-Morrey spaces and the (sequenced valued) block spaces. There operators include Hardy-Littlewood maximal operator, iterated maximal operator, fractional integral operator, singular integral operators.
In Section \ref{sec LP}, we study the Littlewood-Paley theory for mixed Bourgain-Morrey space and its predual. The characterizations by the heat semigroup and wavelets of mixed Bourgain-Morrey spaces are also obtained in Sections \ref{heat char} and \ref{wavelet char}.
In the last Section \ref{sec chain rules}, 
as an application, using the work of Douglas \cite{D25}, we prove a fractional chain rule in the mixed Bourgain-Morrey Triebel-Lizorkin spaces.

Throughout this paper,  let $c, C$ denote constants that are independent of the main parameters involved but whose value may differ from line to line.
For $A,B>0$, by $A\lesssim B$, we mean that $A\leq CB$ with some positive constant $C$ independent of appropriate quantities. By $ A \approx B$, we mean that $A\lesssim B$ and $B\lesssim A$.

\section{Preliminaries} \label{preliminaries}
First we recall some notations. Let $\mathbb N  =\{1,2,\ldots\} $ and let $\mathbb Z$ be all the integers. 
Let $\mathbb{N}_{0}:=\mathbb{N\cup}\{0\}$.
Let $\chi_{E}$ be the characteristic function of the set $E\subset\mathbb{R}^{n}$. 
Set $E^c : = \mathbb R^n \backslash E$ for a set $E \subset \rn$.
For $ 1 \le  p \le \infty$, let $p'$  be the conjugate exponent of $p$, that is $1/p +1/p' =1 $.  The function $\sgn $ is defined by 
\begin{equation*}
	\sgn (f) = \begin{cases}
		1, & \operatorname{if} f>0 \\
		0 & \operatorname{if} f = 0, \\
		-1, & \operatorname{if} f <0.
	\end{cases}
\end{equation*}
For $j\in\mathbb{Z}$, $m\in\mathbb{Z}^{n}$, let $Q_{j,m}:=\prod_{i=1}^{n}[2^{-j}m_{i},2^{-j}(m_{i}+1))$.
For a cube $Q$, $\ell(Q)$ stands for the length of cube $Q$. We
denote by $\mathcal{D} : = \{Q_{j,m},  j\in\mathbb{Z}$, $m\in\mathbb{Z}^{n} \}$ the family of all dyadic cubes in $\mathbb{R}^{n}$,
while $\mathcal{D}_{j}$ is the set of all dyadic cubes with $\ell(Q)=2^{-j},j\in\mathbb{Z}$. Let $a Q , a >0$ be the  cube  concentric with $Q$, having  the side length of $a \ell (Q)$.
For any $R>0$, Let $ B(x,R) := \{y\in \rn : |x-y| <R \} $ be an open ball in $\rn $. Let $B_R : = B (0,R)$ for $R>0$.

Let $\mathscr S := \mathscr{S}(\mathbb{R}^{n})$
denote the Schwartz space on $\rn$, and let $ \mathscr S ' :=  \mathscr{S}'(\mathbb{R}^{n})$
be its dual. Let $\mathcal{P} : = \mathcal{P}(\mathbb{R}^{n})$ be the class of the
polynomials on $\mathbb{R}^{n}$. Denote by $\mathscr{S}'/\mathcal{P} :=\mathscr{S}'/\mathcal{P}(\mathbb{R}^{n})$
the space of tempered distributions modulo polynomials. Let $\mathscr{S}_{0} : = \mathscr{S}_{0}(\mathbb{R}^{n}):=\{f\in\mathscr{S}:\partial^{\alpha}\mathcal{F}(g)(0)=0$
for all multi-indices $\alpha$\}. Recall that $\mathscr{S}'/\mathcal{P}$
is the dual of $\mathscr{S}_{0}$. Let $ L^0 := L^0 (\rn) $  be the set of  all measurable functions on $\rn$.
Let $\mathbb M ^+ : = \mathbb M ^+  (\rn)$ be the cone of all non-negative Lebesgue measurable functions.
We denote by $\operatorname{Sim} (\rn)$ the class of all simple functions 
\begin{equation*}
	f =\sum_{i=1}^N a_i  \chi_{E_i}, 
\end{equation*}
where $N \in \mathbb N$, $\{ a_i\} \subset \mathbb C$ and $\{ E_i \} \subset \rn  $ are pairwise disjoint.

The Fourier transform of $f$ is defined by $\F (f) (\xi) := \int_\rn f(x) e^{- 2\pi ix \cdot \xi} \d x$. Denote by $\F^{-1} (f) (x) := \int_\rn f(\xi) e^{ 2\pi i x \cdot \xi} \d \xi$ the inverse Fourier transform
of $f$ .

 The letters $\vec p  = (p_1, \ldots,  p_n)$ will denote $n$-tuples of the numbers in $[0,\infty]$ ($n\in \mathbb N$).
The notation $0< \vec p <\infty$ means that $0 < p_i < \infty$ for each $i \in \{1,\ldots,n\}$.	For $a\in \mathbb R$, let \begin{equation*}
	\frac{1}{\vec p }  = \left( \frac{1}{p_1}, \ldots , \frac{1}{p_n} \right),\quad  a \vec p = (ap_1, \ldots, a p_n), \quad  \vec{p}^{\,\prime} = ( p_1 ', \ldots, p_n ' ).
\end{equation*}

Then we  recall the  mixed Lebesgue spaces, which are  introduced by Benedek and Panzone in \cite{BP61}.
\begin{definition}
	Let $\vec p = (p_1, \ldots, p_n)  \in (0,\infty]^n$. Then the Mixed Lebesgue spaces $L^{\vec p} : = L^{\vec p} (\rn)$ is the set of all measurable functions $f :\rn \to \mathbb C$ such that
	\begin{equation*}
		\| f \|_{L^{\vec p} } : = \left(\int_\R \cdots  \left( \int_\R  |f(x_1, \ldots, x_n) |^{p_1 }\d x_1  \right)^{  \frac{p_2}{p_1} }   \cdots \d x_n  \right) ^{\frac{1 }{p_n}} <\infty .
	\end{equation*}
 We  make appropriate modifications when  $p_j =\infty$.
We also define $\| f\|_{L^{\vec p} (E) } := \| f \chi_E \|_{L^{\vec p}} $ for a set $E \subset \rn$.

\end{definition}

Note that the order in which the norms are taken is fundamental because in general (\cite[page 302]{BP61})
\begin{equation*}
	\left(\int_\R \left( \int_\R  |f(x_1, x_2) |^{p_1 }\d x_1  \right)^{  \frac{p_2}{p_1} }    \d x_2  \right) ^{\frac{1 }{p_2}} \neq \left(\int_\R \left( \int_\R  |f(x_1, x_2) |^{p_2 }\d x_1  \right)^{  \frac{p_1}{p_2} }    \d x_2  \right) ^{\frac{1 }{p_1}} .
\end{equation*}

In \cite{N19}, Nogayama introduced  Mixed Morrey spaces.

\begin{definition}
	Let $\vec p = (p_1, \ldots, p_n)  \in (0,\infty]^n$ and $t\ \in (0,\infty] $ satisfy 
	\begin{equation*}
		\sum_{i=1 } ^n \frac{1}{p_i}  \ge \frac{n}{t}.
	\end{equation*}
The mixed Morrey norm $\| \cdot \|_{ M_{\vec p} ^t   } $  is defined by 
\begin{equation*}
	\|f\|_{  M_{\vec p} ^t   } : = \sup_{\operatorname{cube} \; Q \subset \rn } |Q|^{ \frac{1}{t}  - \frac{1}{n} \sum_{i=1 } ^n \frac{1}{p_i}  }  \| f\chi_Q \|_{L^{\vec p}  } ,
\end{equation*}
where $f \in L^0 $. The mixed Morrey space $ M_{\vec p} ^t  $ is the set of all measurable functions $f$ with finite norm $ \|f\|_{  M_{\vec p} ^t  }$.
\end{definition}

Inspired by the Bourgain-Morrey space and the mixed Morrey space,
 we introduce the mixed Bourgain-Morrey space. 
\begin{definition}
	Let $\vec p = (p_1, \ldots, p_n)  \in (0,\infty]^n$ and $t\ \in (0,\infty] $ satisfy
	\begin{equation*}
		\sum_{i=1 } ^n \frac{1}{p_i}  \ge \frac{n}{t}.
	\end{equation*}
Let $ t \le r \le \infty$. Then the mixed Bourgain-Morrey space $ M_{\vec p} ^{t,r}:=  M_{\vec p} ^{t,r}  (\rn)  $ is the set of all measurable functions $f$ such that
\begin{equation*}
	\|f\|_{  M_{\vec p} ^{t,r}  } : = \left(\sum_{Q \in \D} |Q|^{ \frac{r}{t}  - \frac{r}{n} \sum_{i=1 } ^n \frac{1}{p_i}  }  \| f\chi_Q \|_{L^{\vec p}  } ^r  \right)^{1/r}<\infty.
\end{equation*}
\end{definition}

\begin{remark}
	(i) If $ \vec p = (p,p,\ldots, p) $, then $  M_{\vec p} ^{t,r}   $  is the Bourgain-Morrey spaces, which can be found in \cite{HNSH23, M16}.
	
	(ii) If $r =\infty$,  then $  M_{\vec p} ^{t,r}   $   becomes the mixed Morrey space.

\end{remark}

We list some properties for $L^{\vec p} $, which will be used frequently. 

The first property  is the Fatou's property for $L^{\vec p}$; for example, see \cite[Proposition 2.2]{N19}.
\begin{lemma} \label{lem fatou}
	Let $ 0 < \vec p \le \infty$. Let $\{ f_k\}_{k=1}^\infty$ be a 	sequence of non-negative measurable functions on $\rn$. Then 
	\begin{equation*}
		\left\|  \liminf _{k\to \infty} f_k \right\|_{L^{\vec p}}  \le \liminf _{k\to \infty}
		\left\|   f_k \right\|_{L^{\vec p}} .
	\end{equation*}
\end{lemma}

\begin{remark}
	Let $f,g \in L^0 $.
	It is easy to verify that the mixed Bourgain-Morrey space have the lattice property,  that is, $|g| \le |f| $ almost everywhere implies that $\| g\|_{ M_{\vec p} ^{t,r}  }  \le\| f\|_{ M_{\vec p} ^{t,r}   }  $. By the dominated converge theorem, $0 \le f_k \uparrow f$ almost everywhere as $k \to \infty$ implies that $ \|f_k\|_{  M_{\vec p} ^{t,r}  }  \uparrow \| f\|_{  M_{\vec p} ^{t,r}   }$ as $k \to \infty$ for $ n / ( \sum_{i=1}^n  1/p_{i})  \le  t \le r < \infty$.
	From \cite[Theorem 2]{CFMN21}, we obtain $M_{\vec p} ^{t,r}  $ is complete for $ n / ( \sum_{i=1}^n  1/p_{i})  \le  t \le r < \infty$.
	
	When $r=\infty$,  $  M_{\vec p} ^{t,\infty}   $  are the mixed Morrey spaces and  complete.
\end{remark}

The second property is the H\"older inequality, coming from,  for example, \cite[Section 2.1]{N19}.
\begin{lemma} \label{Holder mixed}
	Let $1 <\vec p ,\vec q <\infty$ and define $\vec r$ by $ 1/ \vec p + 1 / \vec q = 1/ \vec r$. If $f \in L^{\vec p} $  and  $g \in L^{\vec q} $, then $fg \in L^{\vec r} $  and
	\begin{equation*}
		\|fg\|_{ L^{\vec r} }  \le 	\|f\|_{ L^{\vec p}  }  	\|g\|_{ L^{\vec q}  } .
	\end{equation*}
\end{lemma}

The third property is duality of $L^{\vec p}$.

\begin{lemma}[Theorem 1, \cite{BP61}] \label{dual mixed Lp}
	
	{\rm (i)}
	Let $1 \le \vec p < \infty$, $1/\vec p + 1 / \vec p \, ^\prime =1 $. $J (f)$ is a continuous linear functional on the normed space $L^{\vec p}$  if and only if it can be represented by 
	\begin{equation*}
		J (f) = \int_\rn h(x) f(x) \d x
	\end{equation*}
where $h$ is a uniquely determined function of $ L^{\vec p \, ^\prime}$, and $ \| J\| = \|h\|_{ L^{\vec p \, ^\prime}}$.

	{\rm (ii)} Let $1 \le \vec p < \infty$. The normed space $L^{\vec p}$ is a Banach space and every sequence converging
	in $L^{\vec p}$ contains a subsequence convergent almost everywhere to the limit 	function.
\end{lemma}

Finally, we recall the Muckenhoupt's class $A_p$ and it will be used in Section \ref{sec operator}.

Denote by $\M$  the Hardy-Littlewood maximal function of $f$:
\begin{equation*} 
	\M f (x) = \sup_{B \ni x } \frac{1}{ |B| } \int_B | f (y) | \d y
\end{equation*}
where the sup is taken over all balls $B$ containing $x \in\rn$. For $\eta \in (0,\infty)$, set $\M _\eta (f) = \M (|f|^\eta)^{1/\eta}$.

\begin{definition}
	Let $1<p<\infty $. We say that a weight $\omega$ belongs to class $A_p := A_p (\rn)$ if
	\begin{equation*}
		[\omega]_{A_p} := \sup_{B\;  \mathrm{ balls \; in  } \;\rn} \Big(   \frac{1}{|B|} \int_B \omega(x) \mathrm {d} x    \Big) ^{1/p} \Big(   \frac{1}{|B|} \int_B \omega(x) ^ { -1/(p-1) } \mathrm {d} x  \Big) ^{ (p-1 ) /p} <\infty.
	\end{equation*}
	We call a weight $\omega$   an $A_1$ weight if
	\begin{equation*}
		\mathcal M (w) (x) \le C w(x)
	\end{equation*}
	for almost all $x \in X$ and some $C>0$.
	Set $A_\infty = \cup_{p\ge 1} A_p$.
\end{definition}

%

\section{Properties of mixed Bourgain-Morrey spaces} \label{sec property}
In this section, we discuss some properties and examples of mixed Bourgain-Morrey spaces.

\subsection{Fundamental properties}
In what follows, the symbol $\hookrightarrow$ stands for continuous embedding. That is, for quasi-Banach spaces $X$ and $Y$, the notation $X\hookrightarrow Y$ means that $X \subset Y$ and that the canonical injection from $X$ to $Y$ is continuous.
\begin{lemma} \label{embed ell r}
		Let $ 0 < \vec p \le \infty $.  Let $ 0 < n / ( \sum_{i=1}^n  1/p_{i}) < t \le r_1 \le  r_2 \le  \infty $.   Then
	\begin{equation*}
		M_{\vec p} ^{t,r_1}    \hookrightarrow  M_{\vec p} ^{t,r_2}   .
	\end{equation*}
\end{lemma}
\begin{proof}
	The inclusion follows immediately from $ \ell^{r_1}  \hookrightarrow  \ell^{r_2}$.
\end{proof}

\begin{proposition} \label{pro embedding}
	Let $ 0 < \vec p  \le \vec s \le \infty $.  Let $ 0< t \le r \le  \infty $ such that $ n/t \le  \sum_{ i=1 }^n 1/ s_i $.   Then
	\begin{equation*}
		 M_{\vec s} ^{t,r}    \hookrightarrow  M_{\vec p} ^{t,r}  .
	\end{equation*}
\end{proposition}
\begin{proof}
	Then case $r =\infty$  was proved in  \cite[Proposition 3.2]{N19}. Now let $0<  t \le r <\infty$.
	It suffices to show 
	\begin{equation} \label{s p embed}
		\sum_{Q \in \D} |Q|^{ \frac{r}{t}  - \frac{r}{n} \sum_{i=1 } ^n \frac{1}{p_i}  }  \| f\chi_Q \|_{L^{\vec p}  } ^r   \le 	\sum_{Q \in \D} |Q|^{ \frac{r}{t}  - \frac{r}{n} \sum_{i=1 } ^n \frac{1}{s_i}   }  \| f\chi_Q \|_{L^{\vec s}  } ^r  .
	\end{equation}
From \cite[Proposition 3.2]{N19}, we have
\begin{equation*}
	\| f\chi_Q \|_{L^{\vec p}  }  \le \ell (Q) ^{ \left(\sum_{i=1}^n \frac{1}{ p_i}   \right)  -  \left(\sum_{i=1}^n \frac{1}{ s_i}   \right) } \| f\chi_Q \|_{L^{\vec s}  } .
\end{equation*}
Put this estimate into the left side of (\ref{s p embed}) and we obtain  the desired result (\ref{s p embed}).
\end{proof}

We also have the dilation property for $ M_{\vec p} ^{t,r}  $. 
\begin{lemma}
	Let $0 < \vec p \le \infty$. Let $0 < n / ( \sum_{i=1}^n  1/p_{i})    \le t < r \le \infty $. Then for $f \in M_{\vec p} ^{t,r}  $,  and $ a >0 $, 
	\begin{equation*}
		\| f(a \cdot )  \|_{ M_{\vec p} ^{t,r}   }  \approx a^{- n/t} \| f  \|_{ M_{\vec p} ^{t,r}   } .
	\end{equation*}
\end{lemma}

\begin{proof}
	The proof is simple and we omit it here. We refer the reader to \cite[(6)]{N19} for the proof of $ M_{\vec p} ^{t,\infty}  $  and to \cite[Lemma 2.4]{HNSH23} for the proof of $ M_{ p} ^{t,r}  $.
\end{proof}

Next, consider the translation invariance of  $ M_{\vec p} ^{t,r}  $.
\begin{lemma}
	Let  $0< \vec p \le \infty  $. Let $0 < n / ( \sum_{i=1}^n  1/p_{i})    \le t < r \le \infty $. Then there exists $C_{n, \vec p, r} >0 $  such that for all $  y \in \rn $  and $f \in  M_{\vec p} ^{t,r}  $,  we have
	\begin{equation*}
		\| f( \cdot -y) \|_{ M_{\vec p} ^{t,r}   }  \le  C_{n, \vec p, r}   \| f \|_{ M_{\vec p} ^{t,r}   } .
	\end{equation*}
\end{lemma}

\begin{proof}
	The proof is similar to \cite[Lemma 2.5]{HNSH23} and we omit it here.
\end{proof}

Using this property and the triangle inequality for $M_{\vec p} ^{t,r}  $, we obtain the following convolution inequality.
\begin{corollary} \label{convolution}
	Let $1 \le \vec p \le \infty $. Let $0 < n / ( \sum_{i=1}^n  1/p_{i})     \le t < r \le \infty $. Then  there exists $C_{n, \vec p, r} >0 $  such that
	\begin{equation*}
		\| g*f \|_{ M_{\vec p} ^{t,r}   }  \le C_{n, \vec p, r}  \|g\|_{L^1}	\| f \|_{ M_{\vec p} ^{t,r}   }  
	\end{equation*}
for all  $f \in M_{\vec p} ^{t,r}    $  and $g \in L^1$.
\end{corollary}

Before proving Young's inequality for mixed Bourgain-Morrey spaces, we first recall Young's inequality for mixed Lebesgue spaces.

\begin{lemma} [p 319, Theorem 1, \cite{BP61}]  \label{young mixed Lp}
	Let $1 \le \vec p, \vec q, \vec s \le\infty $ satisfy the relation $1/ \vec p + 1/ \vec q = 1 + 1/\vec s $. If $f \in L^{\vec p}, g \in L^{\vec q}$, then  $ f* g \in L^{ \vec s} $  and 
	\begin{equation*}
		\| f*g \|_{ L^{ \vec s}  }  \le \| f \|_{ L^{\vec p} }   \|g\|_{ L^{\vec q} }.
	\end{equation*}
\end{lemma}

\begin{theorem}[Young's inequality]  
	Let parameters $ \vec p , \vec q, \vec s ,  t,t_0 ,t_1 , r , r_0 ,r_1 $ satisfy the following relations:
	
	{\rm (i)} $1 \le \vec p, \vec q, \vec s \le\infty $ and  $1 + 1/\vec p  =  1/ \vec s + 1/ \vec q  $;
	
	{\rm  (ii)}  $ n / ( \sum_{i=1}^n  1/p_{i})    \le t < r \le \infty $,    $ n / (\sum_{i=1}^n 1/  s_i    )  \le t_0 < r_0 < \infty $,    $n / (\sum_{i=1}^n 1/  q_i    )    \le t_1 < r_1 < \infty $;   
	
	{\rm  (iii)}  $1/t + 1 =1/t_0  +1/ t_1  $ and $1/r + 1 =1/ r_0 + 1/ r_1 $.
	
	Then for all $f \in  M_{\vec s} ^{t_0,r_0} $  and $g \in M_{\vec q} ^{t_1,r_1}   $,
	\begin{equation*}
		\| f*g\|_{  M_{\vec p} ^{t,r}  }  \lesssim \|f\|_{ M_{\vec s} ^{t_0,r_0}    } \|g\|_{ M_{\vec q} ^{t_1,r_1}    } .
	\end{equation*}
\end{theorem}
\begin{proof}
	We use the idea from \cite[Theorem 2.7]{HNSH23}. Since the spaces $ M_{\vec s} ^{t_0,r_0}  $ and  $ M_{\vec q} ^{t_1,r_1}   $ are closed when taking the absolute value, we assume $f$ and $g$ are non-negative. This justifies the definition of $f*g$. Fix $v\in \mathbb Z$ and $m \in \mathbb Z^n$.
	By Minkowski's inequality and Young's inequality for the $L^{ \vec p }$-norm (Lemma \ref{young mixed Lp}), we have
	\begin{align*}
		\| (f *g) \chi_{Q_{v , m} } \|_{L^{\vec p}}  & = \left\| \left( \sum_{\tilde m \in \mathbb Z^n} \int_{Q_{v , \tilde m} } f (y) g(\cdot -y) \d y  \right)  \chi_{Q_{v, m} }  \right\|_{L^{\vec p} }  \\
		& \le  \sum_{\tilde m \in \mathbb Z^n} \|  (f \chi_{ Q_{v, \tilde m} })  * (g \chi_{ Q_{v , m} - Q_{v , \tilde m}  })   \|_{ L^{\vec p} }  \\
		& \le  \sum_{\tilde m \in \mathbb Z^n} \|  f \chi_{ Q_{v ,  \tilde m} } \|_{L^{\vec s } } \| g \chi_{ Q_{v , m} - Q_{v  , \tilde m}  }  \|_{ L^{\vec q} }  ,
	\end{align*}
where $Q_{v, m} - Q_{v  , \tilde m}  := \{  x - \tilde x :x \in Q_{v , m} ,  \tilde x \in Q_{v,  \tilde m}  \} $. Since $ Q_{v, m} - Q_{v ,  \tilde m}  \subset 3 Q_{ v, m-\tilde m} $, by the definition of the parameters $ \vec p , \vec q, \vec s ,  t,t_0 ,t_1 , r , r_0 ,r_1 $, we  obtain
\begin{align*}
	& \left( \sum_{ m \in \mathbb Z^n}  | Q_{v, m}|^{ \frac{r}{t} - \frac{r}{n} \sum_{i=1}^n \frac{1}{p_i} } \| (f *g) \chi_{Q_{v, m} } \|_{L^{\vec p}} ^r	\right) ^{1/r}\\
	& \le \left( \sum_{ m \in \mathbb Z^n}  | Q_{v, m}|^{ \frac{r}{t} - \frac{r}{n} \sum_{i=1}^n \frac{1}{p_i} } \left( \sum_{\tilde m \in \mathbb Z^n} \|  f \chi_{ Q_{v,  \tilde m} } \|_{L^{\vec s } } \| g \chi_{  3 Q_{ v, m-\tilde m}  }  \|_{ L^{\vec q} }  \right) ^r	\right) ^{1/r}\\
	& \le \left( \sum_{ m \in \mathbb Z^n}  | Q_{v, m}|^{ \frac{r_0}{t_0} - \frac{r_0 }{n} \sum_{i=1}^n \frac{1}{s_i} }  \|  f \chi_{ Q_{v,  m} } \|_{L^{\vec s } } ^{r_0} \right) ^{1/r_0}  
	\left(  \sum_{ m \in \mathbb Z^n}  | Q_{v, m}|^{ \frac{r_1}{t_1} - \frac{r_1 }{n} \sum_{i=1}^n \frac{1}{q_i} }     \| g \chi_{3 Q_{v, m}  }  \|_{ L^{\vec q} }   ^{r_1}	\right) ^{1/r_1} \\
	& \lesssim  \left( \sum_{ m \in \mathbb Z^n}  | Q_{v, m}|^{ \frac{r_0}{t_0} - \frac{r_0 }{n} \sum_{i=1}^n \frac{1}{s_i} }  \|  f \chi_{ Q_{v , m} } \|_{L^{\vec s } } ^{r_0} \right) ^{1/r_0}  
	\left(  \sum_{ m \in \mathbb Z^n}  | Q_{v, m}|^{ \frac{r_1}{t_1} - \frac{r_1 }{n} \sum_{i=1}^n \frac{1}{q_i} }     \| g \chi_{ Q_{v, m}  }  \|_{ L^{\vec q} }   ^{r_1}	\right) ^{1/r_1} 
\end{align*}
where we used Young's inequality for the $\ell^r$-norm in the second inequality.  
Note that for each $j =1,2 $, $1 <r_j <r <\infty$ and hence $ \ell^{r_j} (\mathbb Z) \hookrightarrow \ell^{r} (\mathbb Z)  \hookrightarrow \ell^{\infty} (\mathbb Z)$. 
Consequently,
\begin{align*}
	\| f*g\|_{  M_{\vec p} ^{t,r}  } 
	& \lesssim \left\|\left( \sum_{ m \in \mathbb Z^n}  | Q_{v, m}|^{ \frac{r_0}{t_0} - \frac{r_0 }{n} \sum_{i=1}^n \frac{1}{s_i} }  \|  f \chi_{ Q_{v,  m} } \|_{L^{\vec s } } ^{r_0} \right) ^{1/r_0}  \right\|_{\ell_v^{r} }  \\
	& \quad \times
	 \left\|	\left(  \sum_{ m \in \mathbb Z^n}  | Q_{v, m}|^{ \frac{r_1}{t_1} - \frac{r_1 }{n} \sum_{i=1}^n \frac{1}{q_i} }     \| g \chi_{ Q_{v, m}  }  \|_{ L^{\vec q} }   ^{r_1}	\right) ^{1/r_1}  \right\|_{\ell_v^{\infty} } \\
	 & \le  \left\|\left( \sum_{ m \in \mathbb Z^n}  | Q_{v, m}|^{ \frac{r_0}{t_0} - \frac{r_0 }{n} \sum_{i=1}^n \frac{1}{s_i} }  \|  f \chi_{ Q_{v , m} } \|_{L^{\vec s } } ^{r_0} \right) ^{1/r_0}  \right\|_{\ell_v^{r_0 } } \\
	 & \quad \times
	 \left\|	\left(  \sum_{ m \in \mathbb Z^n}  | Q_{v, m}|^{ \frac{r_1}{t_1} - \frac{r_1 }{n} \sum_{i=1}^n \frac{1}{q_i} }     \| g \chi_{ Q_{v, m}  }  \|_{ L^{\vec q} }   ^{r_1}	\right) ^{1/r_1}  \right\|_{\ell_v^{r_1} } \\
	 & = \|f\|_{ M_{\vec s} ^{t_0,r_0}    } \|g\|_{ M_{\vec q} ^{t_1,r_1}   },
\end{align*}
as desired.
\end{proof}

\subsection{Nontriviality}
Next we consider the nontriviality of mixed Bourgain-Morrey spaces.
\begin{lemma}[(2),  \cite{N19}] \label{chi Q mixed Lp}
	Let $Q$  be a cube. Then for $0  < \vec p \le \infty$, 
	\begin{equation*}
		\| \chi_Q \|_{ L^ {\vec p} }  =|Q|^{\frac{1}{n}  \sum_{i=1}^n   \frac{1}{p_i}} .
	\end{equation*}
\end{lemma}

\begin{example} \label{exa Q0 iff}
	Let $ 0 < \vec p \le \infty $. Let $ 0 <t <\infty $ and $0 < r \le \infty$. Let $Q_0 = [0,1)^n$.
	$\chi_{Q_0 }  \in M_{\vec p} ^{t,r} $  if and only if  $   0 < n / ( \sum_{i=1}^n  1/p_{i})    < t <r <\infty $ or $ 0< n / ( \sum_{i=1}^n  1/p_{i})   \le t < r=\infty  $. 
\end{example}
\begin{proof}

	Case $r<\infty$. We first calculate the norm $\| \chi_{Q_0 } \|_{M_{\vec p} ^{t,r}  }$. By Lemma \ref{chi Q mixed Lp},	we obtain
	\begin{align*}
		& \sum_{v\ge 0} \sum_{ m \in \mathbb Z^n} \left(|Q_{v, m}|^{ 1/t- \frac{1}{n} \sum_{i=1}^n   \frac{1}{p_i}}   \| \chi_{Q_0 }  \chi_{Q_{v, m} }\|_{L^{\vec p}}    \right) ^r \\
		& = \sum_{v\ge 0} \sum_{ m \in \mathbb Z^n, Q_{v, m} \subset Q_0 } \left(|Q_{v, m}|^{ 1/t- \frac{1}{n} \sum_{i=1}^n \frac{1}{ p_i}  }   \|  \chi_{Q_{v, m} }\|_{L^{\vec p}}    \right) ^r \\
		& =  \sum_{v\ge 0} 2^{vn}  2^{-vn  r/t } ,
	\end{align*}
and
\begin{align*}
		& \sum_{v< 0} \sum_{ m \in \mathbb Z^n} \left(|Q_{v, m}|^{ 1/t- \frac{1}{n} \sum_{i=1}^n   \frac{1}{p_i}}   \| \chi_{Q_0 }  \chi_{Q_{v, m} }\|_{L^{\vec p}}    \right) ^r \\
	& = \sum_{v< 0} \left(|Q_{v, m}|^{ 1/t- \frac{1}{n} \sum_{i=1}^n \frac{1}{ p_i}  }   \|  \chi_{Q_{0} }\|_{L^{\vec p}}    \right) ^r \\
	& =  \sum_{v< 0}  \left(   2^{-vn (  1/t- \frac{1}{n} \sum_{i=1}^n \frac{1}{ p_i}   ) }    \right)^r .
\end{align*}
Hence 
\begin{equation*}
	\| \chi_{Q_0 } \|_{M_{\vec p} ^{t,r}  }  = \left( \sum_{v\ge 0} 2^{vn}  2^{-vn  (r/t) }  +   \sum_{v< 0}  \left(   2^{-vn (  1/t- \frac{1}{n} \sum_{i=1}^n \frac{1}{ p_i}   ) }    \right)^r   \right)^{1/r}.
\end{equation*}
The right-hand side is finite if and only if $   0 < n / ( \sum_{i=1}^n  1/p_{i})    < t <r <\infty $.

Case $r  = \infty$. If $v \ge 0$  and $Q_0 \cap Q_{v, m}  \neq \emptyset$, then $  Q_{v, m} \subset Q_0 $ and
\begin{align*}
	|Q_{v, m}|^{ 1/t- \frac{1}{n} \sum_{i=1}^n   \frac{1}{p_i}}   \| \chi_{Q_0 }  \chi_{Q_{v, m} }\|_{L^{\vec p}} 
= |Q_{v, m}|^{ 1/t} = 2^{ -vn /t}.
\end{align*}
If $v < 0$  and $Q_0 \cap Q_{v, m}  \neq \emptyset$, then $  Q_0  \subset Q_{v, m}$ and
\begin{align*}
	|Q_{v, m}|^{ 1/t- \frac{1}{n} \sum_{i=1}^n   \frac{1}{p_i}}   \| \chi_{Q_0 }  \chi_{Q_{v, m} }\|_{L^{\vec p}} 
 = 	|Q_{v, m}|^{ 1/t- \frac{1}{n} \sum_{i=1}^n   \frac{1}{p_i}}  = 2^{ -vn  ( 1/t- \frac{1}{n} \sum_{i=1}^n  \frac{1}{ p_i}  )}.
\end{align*}
Hence
\begin{equation*}
		\| \chi_{Q_0 } \|_{M_{\vec p} ^{t,\infty}  }  = \max \left\{  \sup_{v \ge 0, v \in \mathbb Z}  2^{ -vn /t},\sup_{v< 0, v\in \mathbb Z} 2^{ -vn  ( 1/t- \frac{1}{n} \sum_{i=1}^n  \frac{1}{ p_i}  )}  \right \}. 
\end{equation*}
The right-hand side is finite if and only if $ 0< n / ( \sum_{i=1}^n  1/p_{i})    \le t < r=\infty  $.
\end{proof}

Now we are ready to obtain the nontriviality of mixed Bourgain-Morrey spaces.

\begin{proposition}
	Let $ 0 < \vec p \le \infty $. Let $ 0 <t <\infty $.
	Then    $ M_{\vec p}^{t,r} \neq \{ 0\}$ if and only if $   0 < n / ( \sum_{i=1}^n  1/p_{i})    < t <r <\infty $ or $ 0< n / ( \sum_{i=1}^n  1/p_{i})    \le t < r=\infty  $.
\end{proposition}

\begin{proof}
	We use the idea from \cite[Theorem 2.10]{HNSH23}.
From Example \ref{exa Q0 iff}, the ``if'' part is clear. Thus let us prove the ``only if'' part. Since 
\begin{equation*}
	\|  |f|^u \|_{   M_{\vec p}^{t,r}  }  =  \|  f \|_{   M_{ u \vec p}^{ ut,ur}  } ^u
\end{equation*}
for all $f \in L^0$, we can assume $\min \vec p >1$. In this case $ M_{\vec p}^{t,r}$ is a Banach space. Let $f \neq 0 \in  M_{\vec p}^{t,r} $. We may suppose $ f \ge 0 $ by replacing $f$ with $|f|$. By  replacing $f$ with $\min(1, |f|) $, we assume $f \in L^\infty$. Thanks to Corollary \ref{convolution}, $\chi_{[0,1]^n  } * f \in  M_{\vec p}^{t,r} \backslash \{0\}$. Since $\chi_{[0,1]^n  } * f$ is a non-negative non-zero continuous function, there exist $x_0 \in \rn, \epsilon >0 $  and $R>0$  such that $\chi_{[0,1]^n  } * f \ge \epsilon \chi_{B(x_0, R)} $. Thus $ \chi_{B(x_0, R) }  \in M_{\vec p}^{t,r} $. Since $ M_{\vec p}^{t,r}$  is invariant under translation and dilation, $ \chi_{[0,1]^n } \in M_{\vec p}^{t,r} $. Thus from Example \ref{exa Q0 iff}, we obtain $   0 < n / ( \sum_{i=1}^n  1/p_{i})     < t <r <\infty $ or $ 0< n / ( \sum_{i=1}^n  1/p_{i})     \le t < r=\infty  $.
\end{proof}

\subsection{Approximation, density and separability}
Here we investigate the approximation property of $ M_{\vec p}^{t,r} $ when $r<\infty$. 

We first recall the definition of filtrations, adaptedness and martingales; for example, see \cite[Definition 3.1.1]{HNVW16}.
\begin{definition}
	Let $(S, \mathcal A, \mu )$ be a measure space and $(I, \le) $ an ordered set.
	
	(i) A family of sub-$\sigma$-algebras $\mathscr F$ of $\mathcal  A$ is called a filtration in $(S, \mathcal A, \mu )$ 
	if $\mathscr F_m \subset \mathscr F_n$ whenever $m, n \in I$ and $m \le n$. The filtration is called $\sigma$-finite if $\mu$ is $\sigma$-finite on each $ \mathscr F_n$.
	
	(ii) A family of functions $(\mathscr F_n )_{n \in I} $ in $L^0(S)$ (the measurable functions on $S$) is adapted to the filtration $(\mathscr F_n )_{n \in I} $ if $f_n \in L^0(S)$ for all $n \in I$.

(iii) A family of functions $(\mathscr F_n )_{n \in I} $ in $L^0(S)$  is called a martingale with
respect to a $\sigma$-finite filtration $(\mathscr F_n )_{n \in I} $ if it is adapted to $(\mathscr F_n )_{n \in I} $, and
for all indices $m\le n$ the function $f_n$ is $\sigma$-integrable over $\mathscr F_m$ and satisfies
\begin{equation*}
	\mathbb E(f_n | \mathscr F_m) = f_m.
\end{equation*}
\end{definition}
The following example comes from \cite[Example 2.6.13, Example 3.1.3]{HNVW16}.
\begin{example}
	Let $(S, \mathcal A, \mu ) = (\rn, \mathcal B (\rn), \d x )$.
	Let $ \mathcal D_j := \{ 2^{-j} ( [0,1)^n +k ) : k\in \mathbb Z^n \}  , j\in \mathbb Z$ be the standard dyadic cubes of side-length $2^{-j} $ as in Section \ref{preliminaries}. Let $\mathscr F_j := \sigma (  \mathcal D_j  )$ its 	generated $\sigma$-algebra. 
 Then $( \mathscr F_j)_{j \in \mathbb Z}$ is a $\sigma$-finite filtration. Every function $f\in L^1_{\operatorname{loc}}$ is $\sigma$-integrable over every $\mathscr F_j$
 and therefore generates a martingale
$(\mathbb E (f| \mathscr F_j)) _{j\in \mathbb Z}$. In this case, a conditional expectation of $f$ with respect to $\mathscr F_j $
 exists and is given by
 \begin{equation*}
 	\sum_{Q \in \D_k}  \frac{1}{|Q|} \int_Q f (y) \d y   \chi_Q (x), \quad x \in \rn.
 \end{equation*}
\end{example}

For each $k \in \mathbb Z$, for a   measurable function $f$ and a cube $Q$, define the  operator $	\mathbb E_k  $  by

\begin{equation} \label{def E_k}
	\mathbb E_k   (f) (x) = \sum_{Q \in \D_k}  \frac{1}{|Q|} \int_Q f (y) \d y   \chi_Q (x), \quad x \in \rn.
\end{equation}

By virtue of the Lebesgue differentiation theorem, for almost every $x\in \rn $, 	$\mathbb E_k   (f)(x)$ converges to $f(x)$ as $k \to \infty $. 
We define 
\begin{equation*}
	\mathbb M f (x)= \sup _{ k \in\mathbb Z} \mathbb E_k   (f) (x) .
\end{equation*}
The following result is the Doob's inequality in mixed Lebesgue spaces.
\begin{lemma}
	[Theorem 2, \cite{SW21}]  \label{doob}
	If $1<\vec p <\infty$ or $\vec p = (\infty, \ldots, \infty ,p_{k+1}, \ldots, p_n)$ with $ 1<,p_{k+1}, \ldots, p_n <\infty $ for some $k \in \{1, \ldots, n \} $. Then the maximal operator $\mathbb M$ is bounded on $L^{\vec p}$, that is, for all $f\in L^{\vec p}$
	\begin{equation*}
	\| \mathbb M f \|_{ L^{\vec p}} \lesssim 	\| f \|_{ L^{\vec p}} .
	\end{equation*}
\end{lemma}

For a real-valued measurable function $f$, we set $f_+ := \max (f,0)$ and $f_{-} := \max (-f, 0)$. Then $ \mathbb E_k f = \mathbb E_k f_+  - \mathbb E_k f_-$.

\begin{theorem} \label{E_k^q f to f}
	Let $1 < \vec p <\infty $. 
	Let $1 \le  n / ( \sum_{i=1}^n  1/p_{i})     < t <r <\infty $.
	  Then for a real-valued functions  $f \in M_{\vec p}^{t,r} $, the sequence $  \{\mathbb E_k (f_+) -  \mathbb E_k (f_-)  \} $ converges to $f$ in $ M_{\vec p}^{t,r} $.
\end{theorem}
Note that our proof below shows $\mathbb E_k (f_+) - \mathbb E_k (f_-)  $ in $ M_{\vec p}^{t,r} $.
\begin{proof}
	We use the idea from \cite[Theorem 2.20]{HNSH23}.
	Let $f\in M_{\vec p}^{t,r} $ and $\epsilon>0 $ be fixed. 
	Since 
	\begin{equation*}
		\| f - ( \mathbb E_k  (f_+) -  \mathbb E_k (f_-) ) \|_{M_{\vec p}^{t,r}  }  \lesssim \| f_+ -  \mathbb E_k (f_+) \|_{M_{\vec p}^{t,r}  }  +\| f_{-} -   \mathbb E_k (f_-)  \|_{M_{\vec p}^{t,r}  } , 
	\end{equation*}
we need to show that $ \mathbb E_k (f_+)$ converges to $f_+$ and that $ \mathbb E_k (f_-)$ converges to $f_-$ in  $ M_{\vec p}^{t,r} $. Thus, we may assume that $f \in   M_{\vec p}^{t,r} $ is non-negative and we will prove that $ \{ \mathbb E_k (f_+) \}_{k =1}^\infty $  converges back to $f$ in $M_{\vec p}^{t,r} $.
	
	Since $f \in M_{\vec p}^{t,r}$, there exists $J \in \mathbb N$ large enough such that
	\begin{align} \label{J epsilon}
		\nonumber
		&	\left(\sum_{j\in \mathbb Z}  \sum_{m \in \mathbb Z^n}  ( \chi_{\mathbb Z \backslash [-J,J]} (j)  + \chi_{\mathbb Z^n \backslash B_J  }  (m) )  |Q_{j,m}|^{r/t-\frac{r}{n} \sum_{i=1}^n \frac{1}{ p_i}  } 
		\| \chi_{Q_{j,m}  } f \|_{ L^{\vec p}} ^ r
		 \right)^{1/r} \\
		& <\epsilon.
	\end{align}
	Fix $k \in \mathbb N \cap (J,\infty)$. Set
	\begin{align*}
		I & := \left\| \left\{       |Q_{j,m}|^{1/t-\frac{1}{n} \sum_{i=1}^n \frac{1}{ p_i} } 
		\| ( f -\mathbb E_k (f) ) \chi_{Q_{j,m} } \|_{L^{\vec p}  }
	  \right\}_{j \in [-J,J], m\in \mathbb Z^n}  \right\|_{\ell^r} , \\
		II & := \left\| \left\{       |Q_{j,m}|^{1/t-\frac{1}{n} \sum_{i=1}^n \frac{1}{ p_i} } 	\| f \chi_{Q_{j,m} } \|_{L^{\vec p}  } \right\}_{j \in \mathbb Z \backslash [-J,J], m\in \mathbb Z^n}  \right\|_{\ell^r} ,\\
		III & := \left\| \left\{       |Q_{j,m}|^{1/t-\frac{1}{n} \sum_{i=1}^n \frac{1}{ p_i} } \| \mathbb E_k (f)  \chi_{Q_{j,m} } \|_{L^{\vec p}  } \right\}_{j \in ( (-\infty, -J) \cup (J,k]  )\cap \mathbb Z, m\in \mathbb Z^n}  \right\|_{\ell^r}, \\
		IV & := \left\| \left\{       |Q_{j,m}|^{1/t-\frac{1}{n} \sum_{i=1}^n \frac{1}{ p_i} } \| \mathbb E_k (f)  \chi_{Q_{j,m} } \|_{L^{\vec p}  }   \right\}_{j \in  \mathbb Z \backslash (-\infty,k], m\in \mathbb Z^n}  \right\|_{\ell^r}.
	\end{align*}
	Then we decompose
	\begin{equation*}
		\| f-   \mathbb E_k (f) \|_{M_{\vec p}^{t,r}  }  \le I +II+ III+ IV.
	\end{equation*}
	For $I$, using the  Minkowski inequality and (\ref{J epsilon}), we see
	\begin{align*}
		I  \le \left\| \left\{       |Q_{j,m}|^{1/t-\frac{1}{n} \sum_{i=1}^n \frac{1}{ p_i}  } 	\| ( f -\mathbb E_k (f) ) \chi_{Q_{j,m} } \|_{L^{\vec p}  }   \right\}_{j \in [-J,J], m\in \mathbb Z^n \cap B_J}  \right\|_{\ell^r} +\epsilon.
	\end{align*}
Since $	\mathbb E_k   (f) $ converges back to $f$ in $L^{p}_{\operatorname{loc}}$ for $ p \in [1,\infty)$ (for example, sees \cite[Theorem 3.3.2]{HNVW16}) and $ L_{\operatorname{loc}} ^{\max \vec p}  \subset  L_{\operatorname{loc}} ^{\vec p}  \subset L_{\operatorname{loc}} ^{1} $, we have  $	\mathbb E_k   (f) $ converges back to $f$ in $L_{\operatorname{loc}} ^{\vec p}$.
Hence

	\begin{equation*}
		I  < 2\epsilon
	\end{equation*} 
	as long as $k$ large enough.
	
	We move on $II$.  Simply use (\ref{J epsilon}) to conclude $II< \epsilon$.
	
	Next, we estimate $III$. By Lemma \ref{doob},	
	\begin{align*}
		& \| \mathbb E_k (f)  \chi_{Q_{j,m} } \|_{L^{\vec p}  } = \| \mathbb E_k (f \chi_{Q_{j,m} } )  \chi_{Q_{j,m} } \|_{L^{\vec p}  }  \le  \| \mathbb M (f \chi_{Q_{j,m} } )  \chi_{Q_{j,m} } \|_{L^{\vec p}  } \lesssim  \| f \chi_{Q_{j,m} }  \|_{L^{\vec p}  } .
	\end{align*}
	Taking the $\ell^r$-norm over $j \in ( (-\infty, -J) \cup (J,k]  )\cap \mathbb Z, m\in \mathbb Z^n $, we obtain
	\begin{align*}
		III \le \left(   \sum_{ j \in ( (-\infty, -J) \cup (J,k]  )\cap \mathbb Z}   \sum_{ m\in \mathbb Z^n }  |Q_{j,m}|^{r/t-\frac{r}{n} \sum_{i=1}^n \frac{1}{ p_i} } \| f  \chi_{Q_{j,m} } \|_{L^{\vec p}  } ^r \right)^{1/r} <\epsilon.
	\end{align*}
	Finally, we estimate $IV$. Let $j \in \mathbb Z \backslash (-\infty,k]$. Since $Q_{j,m} \subset Q$ for $Q\in \D_k$  as long as $Q_{j,m} \cap Q \neq\emptyset$, using H\"older's inequality with $\vec p / q >  1$, we have
	\begin{align*}
		| IV |^r & \le  \sum_{j=k+1}^\infty \sum_{m \in \mathbb Z^n}  |Q_{j,m}|^{r/t-\frac{r}{n} \sum_{i=1}^n \frac{1}{ p_i} }  \| \mathbb E_k (f)  \chi_{Q_{j,m} } \|_{L^{\vec p}  } ^r  \\
		& = \sum_{j=k+1}^\infty  \sum_{ Q \in \D _k} \sum_{m \in \mathbb Z^n,   Q_{j,m} \subset Q}  |Q_{j,m}|^{r/t} \left( \frac{1}{|Q|} \int_Q f  (z)  \d z \right)  ^{r}  \\
		& \le \sum_{j=k+1}^\infty   \sum_{ Q \in \D _k} \sum_{m \in \mathbb Z^n,   Q_{j,m} \subset Q}  |Q_{j,m}|^{r/t   } |Q|^{ - \frac{r}{n} \sum_{i=1}^n \frac{1}{ p_i}   } \| f\chi_Q \|_{ L^{\vec p} } ^r .
	\end{align*}

	Note that $   |Q_{j,m}| = 2^{ -(j-k)n } |Q|$. Since $t<r<\infty$, 
	we obtain
	\begin{align*}
		| IV |^r & \le   \sum_{j=k+1}^\infty \sum_{ Q \in \D _k}  \sum_{m \in \mathbb Z^n,   Q_{j,m} \subset Q}  2^{ -(j-k)n r/t }  |Q|^{r/t - \frac{r}{n} \sum_{i=1}^n \frac{1}{ p_i}  } \| f\chi_Q \|_{ L^{\vec p} } ^r  \\
			& \le \sum_{j=k+1}^\infty  2^{ -(j-k)n (1- r/t) }\sum_{  Q\in \D_k}|Q|^{r/t - \frac{r}{n} \sum_{i=1}^n \frac{1}{ p_i}  } \| f\chi_Q \|_{ L^{\vec p} } ^r \\
		& \lesssim \sum_{  Q\in \D_k}|Q|^{r/t - \frac{r}{n} \sum_{i=1}^n \frac{1}{ p_i}  } \| f\chi_Q \|_{ L^{\vec p} } ^r.
	\end{align*}
	Since $J \le k$, from (\ref{J epsilon}), we conclude $IV \lesssim \epsilon$.
	
	Combining the estimates $I -IV$, we get 
	\begin{equation*}
		\limsup_{k\to \infty} \| f- \mathbb E_k (f) \|_{M_{\vec p}^{t,r}}   \lesssim \epsilon.
	\end{equation*} 
	Since $\epsilon>0$ is arbitrary, we  obtain the desired result.
\end{proof}

As a further corollary, the set $L_c^\infty$ of all compactly supported  bounded functions  is dense in $M_{\vec p}^{t,r} $.
\begin{corollary} \label{dense L_c mixed BM}
	Let $1< \vec p <\infty $. 
	Let $ 1 <   n / ( \sum_{i=1}^n  1/p_{i})      < t <r <\infty $. Then $L_c^\infty $  is dense in $M_{\vec p}^{t,r} $.
\end{corollary}
\begin{proof}

	We need to approximate $f \in M_{\vec p}^{t,r}$ with functions in $L_c^\infty$.
	We may assume that $f$ is non-negative from the decomposition $f =  f_ + - f_- $. Fix $\epsilon>0$. 
	 By Theorem \ref{E_k^q f to f}, there exists $k$ large enough such that 
	\begin{equation} \label{f -Ek f < eps}
		\| f - \mathbb E_k (f) \|_{M_{\vec p}^{t,r} } < \epsilon.
	\end{equation}
	Note that $\mathbb E_k (f)\chi_{B_R}  \in L_c^\infty $  for all $R>0$.
	Fix a large number $k\in \mathbb N$ such that (\ref{f -Ek f < eps}) holds. By the disjointness of the family $\{ Q_{k,m} \}_{m\in \mathbb Z^n}$, 
	\begin{align*}
		&	\| \mathbb E_k (f) - \mathbb E_k  (f)\chi_{B_R} \|_{ M_{\vec p}^{t,r}  }  \\
		& = \left( \sum_{j\in \mathbb Z, m\in \mathbb Z^n} |Q_{j,m}|^{r/t-\frac{r}{n} \sum_{i=1}^n \frac{1}{ p_i}  }   \|\mathbb E_k (f) - \mathbb E_k  (f)\chi_{B_R} \|_{L^{\vec p} } ^r   \right)^{1/r}  \\
		& \le  \left( \sum_{j\in \mathbb Z}  \sum_{ m\in \mathbb Z^n, Q_{j,m} \cap B_R^c \neq \emptyset} |Q_{j,m}|^{r/t-\frac{r}{n} \sum_{i=1}^n \frac{1}{ p_i}  }  \|\mathbb E_k (f) - \mathbb E_k  (f)\chi_{B_R} \|_{L^{\vec p} } ^r    \right)^{1/r} .
	\end{align*}
	Note that $\mathbb E_k (f) \in M_{\vec p}^{t,r}$ according to Theorem \ref{E_k^q f to f}. By the Lebesgue convergence theorem, we obtain
	\begin{equation*}
		\| \mathbb E_k (f) - \mathbb E_k  (f)\chi_{B_R} \|_{ M_{\vec p}^{t,r}  }    <\epsilon
	\end{equation*}
	as long as $R$ large enough. 
	Then by the Minkowski inequality, we have
	\begin{equation*}
		\| f - E_k (f)\chi_{B_R} \|_{ M_{\vec p}^{t,r} } \lesssim \epsilon.
	\end{equation*}
	Since $\epsilon>0$ is arbitrary, the proof is complete.
\end{proof}
\begin{remark}
	A consequence of Corollary \ref{dense L_c mixed BM} and Lemma \ref{embed ell r} is that $ M_{\vec p}^{t,r} \subset  \widetilde{ M_{\vec p}^{t,\infty} }  $, where $\widetilde{ M_{\vec p}^{t,\infty} }$ stands for the closure of $L_c^\infty $ in $M_{\vec p}^{t,\infty} $.
\end{remark}

\begin{lemma}[Lusin's property, Theorem 1.14-4 (c), \cite{C13}]
	Let $A$ be a subset of $\rn$. Let $f:\rn \to \mathbb R$ be a measurable function with $|A| <\infty$, where $A = \{ x \in \rn: f(x) \neq 0\}$. Then given any $\epsilon >0$, there exists a function $f_\epsilon \in C(\rn)$, the continuous functions on $\rn$, whose support is a compact subset of $A$ and such that 
	$
		\sup_{x\in \rn } |f_\epsilon (x) | \le \sup_{x\in \rn } |f (x) |
	$
and $ | \{ x\in \rn :  f(x) \neq f _\epsilon (x) \}| \le \epsilon$.
\end{lemma}

\begin{corollary} \label{compact continuous dense mixed BM}
	Let $1< \vec p <\infty $. 
	Let $ 1 <   n / ( \sum_{i=1}^n  1/p_{i})     < t <r <\infty $. Then $C_c := C_c (\rn) $, the  set of all compactly supported  continuous functions,  is dense in $M_{\vec p}^{t,r} $.
\end{corollary}
\begin{proof}
	Fix $\epsilon>0$ and choose $m_0 \in \mathbb N$ such that $ 1/m_0 \le \epsilon $.
	By Corollary \ref{dense L_c mixed BM}, for $f \in M_{\vec p}^{t,r}$, there exists $g \in L_c^\infty$ such that 
	\begin{equation*}
		\| f-g \|_{M_{\vec p}^{t,r} } \le 1/m_0 \le \epsilon.
	\end{equation*}
Let $A = \operatorname{supp} (g)$. By Lusin's property, there exists continuous function $g_m $ such that supp$(g_m) \subset A$ and $ | \{ x\in \rn :  g(x) \neq g _m (x) \}| \le1/m$. Hence the sequence $\{ g_m\}_{m \in \mathbb N}$ converges to $g$ in measure. By the Resiz theorem, there exists a subsequence $\{ g_m\}_{m \in \mathbb N}$ (we still use the notation $\{ g_m\}_{m \in \mathbb N}$) converges to $g $ a.e. Then by the dominated theorem,
 \begin{equation*}
\lim_{ m\to \infty }	\| g -g_m \|_{M_{\vec p}^{t,r} } =0.
\end{equation*}
For the above $\epsilon$, choose $m' $ large enough such that $ 	\| g -g_{m'} \|_{M_{\vec p}^{t,r} } \le \epsilon $.
By the Minkowski inequality, we obtain
\begin{equation*}
	\| f-g_{m'}  \|_{M_{\vec p}^{t,r} } \le 	\| f-g \|_{M_{\vec p}^{t,r} } +	\| g-g_{m'}  \|_{M_{\vec p}^{t,r} } \le 2\epsilon.
\end{equation*}
Thus we finish the proof.
\end{proof}

We say that a metric space $E$ is separable if there exists a subset $ D \subset E$
that is countable and dense.

\begin{theorem} \label{separable mixed BM}
	Let $1< \vec p <\infty $. 
	Let $ 1 <   n / ( \sum_{i=1}^n  1/p_{i})     < t <r <\infty $. Then  $M_{\vec p}^{t,r} $  is separable.
\end{theorem}
\begin{proof}
	Let $A := \{ I = \prod_{i=1}^n (a_i, b_i): a_i ,b_i \in \mathbb Q \}$. Then $A$ is countable. Let 
$\mathcal A  $ be the space spanned by functions $\{  \chi_{Q} : Q \in A  \}$ on $\mathbb Q$. That is, $f\in \mathcal A $ has the form $f = \sum_{j \in J} \alpha_j \chi_{Q_j} $ where the set $J$ of indices is finite, and $\alpha _j \in \mathbb Q$ and $Q_j \in A$ for all $j \in J$.
	
	For any $f \in M_{\vec p}^{t,r} $, by Corollary \ref{compact continuous dense mixed BM}, there exists $g\in C_c$ such that $\| f-g\|_{M_{\vec p}^{t,r} } \le \epsilon$.   Choose $I \in A$ such that supp $g \subset I $. Since $g \in C_c$, for any $\epsilon >0$, choose $h \in \mathcal A$  such that $\| h-g\|_{L^\infty} <\epsilon  \| \chi_{I} \|_{M_{\vec p}^{t,r}  } ^{-1}$ and supp $h \subset I $. Then
	\begin{equation*}
		\| g -h \|_{M_{\vec p}^{t,r} }  \le \epsilon.
	\end{equation*}
By the Minkowski inequality, we obtain $ \| f -h\|_{M_{\vec p}^{t,r} }  \le 2\epsilon$.
\end{proof}

To  show the duality of mixed Bourgain-Morrey spaces, we need a smaller dense space $\operatorname{Sim} (\rn) $.
The following result will be used in Section \ref{sec associate and dual}.

\begin{lemma} \label{dense simple}
	Let $1< \vec p <\infty $. 
	Let $ 1 <   n / ( \sum_{i=1}^n  1/p_{i})     < t <r <\infty $. Then  $\operatorname{Sim} (\rn) $ is dense in $M_{\vec p}^{t,r}$.
\end{lemma}
\begin{proof}
	Without	loss of generality, we may assume that $f \in M_{\vec p}^{t,r}$ is nonnegative. For any $f \in M_{\vec p}^{t,r}$, there exists $g \in L_c^\infty$  such that \begin{equation*}
		\| f-g\|_{ \in M_{\vec p}^{t,r} } \le \epsilon .
	\end{equation*} 
Using Theorem \ref{E_k^q f to f}, there exists $K \in \mathbb N$ such that 
\begin{equation*}
	\| g - \mathbb E_K  (g) \|_{M_{\vec p}^{t,r} } \le \epsilon .
\end{equation*}
Since $g \in L_c^\infty$, we deduce $\mathbb E_K  (g)  \in \operatorname{Sim} (\rn)$. Then 
\begin{equation*}
	\| f -   \mathbb E_K  (g) \|_{ M_{\vec p}^{t,r}}  \le 2 \epsilon.
\end{equation*}
Thus $\operatorname{Sim} (\rn) $ is dense in $M_{\vec p}^{t,r}$.
\end{proof}

\section{Preduals of mixed Bourgain-Morrey spaces} \label{sec predual}
In this section, we mainly prove  preduals of of mixed Bourgain-Morrey spaces. 
\begin{definition}
	Let $ 1 \le \vec p <\infty$. Let   $n / ( \sum_{i=1}^n  1/p_{i})  \le t \le \infty $. A measurable function $b$ is said to be a
	$(\vec{p}^{\,\prime}, t')$-block if there exists a cube $Q$ that supports $b$ such that 
	\begin{equation*}
		\| b \|_{L^{\vec{p}^{\,\prime} } }  \le |Q| ^{  1/t - \frac{1}{n}  ( \sum_{i=1}^n  1/p_{i})   }  .
	\end{equation*}
	If we need to indicate $Q$, we  say that $b$ is a $(\vec{p}^{\,\prime},t')$-block supported on $Q$.
\end{definition}

\begin{definition}\label{def mix block space}
	Let $ 1 \le \vec p <\infty$. Let   $n / ( \sum_{i=1}^n  1/p_{i})  \le t \le r \le  \infty $.
	The function space $\mathcal{H}_{\vec p \, ^\prime}^{t',r'} $
	is the set of  all measurable functions $f$ such that $f$ is realized
	as the sum
	\begin{equation}\label{eq: mix block f}
		f = \sum_{(j,k)\in\mathbb{Z}^{n+1}}\lambda_{j,k}b_{j,k}
	\end{equation}
	with some $\lambda=\{\lambda_{j,k}\}_{(j,k)\in\mathbb{Z}^{n+1}}\in\ell^{r'}(\mathbb{Z}^{n+1})$
	and $b_{j,k}$ is a $(\vec{p}^{\,\prime},t')$-block supported on  $Q_{j,k}$ where (\ref{eq: mix block f}) converges almost everywhere on $\rn$. The norm of $\mathcal{H}_{\vec p \, ^\prime}^{t',r'} $
	is defined by
	\[
	\|f\|_{\mathcal{H}_{\vec p \, ^\prime}^{t',r'} } :=\inf_{\lambda}\|\lambda\|_{\ell^{r'}},
	\]
	where the infimum is taken over all admissible sequence $\lambda$
	such that (\ref{eq: mix block f}) holds.
\end{definition}

\begin{remark} \label{block norm 1}
	Let  $g_{j_0, k_0}$ be a  $(\vec{p}^{\,\prime},t')$-block supported on some $Q_{j_0, k_0}$. 
	Then $ \|g_{j_0, k_0} \|_{\mathcal{H}_{\vec p \, ^\prime}^{t',r'} }  \le 1$. Indeed, let 
	\begin{equation*}
		\lambda_{j,k} = \begin{cases}
			1, &  \operatorname{if} j = j_0, k = k_0; \\
			0 , & \operatorname{else}.
		\end{cases}
	\end{equation*} 
	Hence  $ \|g_{j_0, k_0} \|_{\mathcal{H}_{\vec p \, ^\prime}^{t',r'} }  \le  \left( \sum_{(j,k)\in   \mathbb Z ^{n+1}  } |	\lambda_{j,k} |^{r'}  \right)^{1/r'} =1.  $
\end{remark}

\begin{proposition}\label{mix block sum converge}
	Let $ 1 \le \vec p <\infty$. Let   $n / ( \sum_{i=1}^n  1/p_{i})  < t < r \le  \infty $.
	Assume that $ \{\lambda_{j,k}\}_{(j,k)\in\mathbb{Z}^{n+1} }  \in \ell^{r'}$ and for each $(j,k)\in\mathbb{Z}^{n+1}$, $b_{j,k}$  is a  $(\vec p \, ^\prime,t')$-block supported on $Q_{j,k}$. Then
	
	{\rm (i)} the summation (\ref{eq: mix block f}) converges  almost everywhere   on $\rn$;
	
	{\rm (ii)} the summation (\ref{eq: mix block f}) converges   in $L^1_{\operatorname{loc}  }$. 
\end{proposition}
\begin{proof}
	We use the idea from \cite[Propositon 3.4]{ZSTYY23}. 
	
	(i) We first claim that, for any given dyadic cube $Q\in \D$,
	\begin{equation} \label{eq almost every}
		\int_Q \sum_{(j,k)\in\mathbb{Z}^{n+1} } |\lambda_{j,k} |  |b_{j,k} (x)| \d x <\infty .
	\end{equation}
	Without loss of generality, we may assume that $Q := [0,1)^n $. Applying the Tonelli theorem, we obtain
	\begin{equation*}
		\int_Q \sum_{(j,k)\in\mathbb{Z}^{n+1} } |\lambda_{j,k} |  |b_{j,k} (x)|\d x = \sum_{(j,k)\in\mathbb{Z}^{n+1}, Q_{j,k} \cap Q \neq \emptyset }	 |\lambda_{j,k} | \int_Q   |b_{j,k} (x)| \d x .
	\end{equation*}
	By H\"older's inequality and support of $b_{j,k}$, we have
	\begin{align*}
		\int_Q    |b_{j,k} (x)| \d x  & \le |Q \cap Q_{j,k}|^{ \frac{1}{n}  \sum_{i=1}^n   \frac{1}{p_i} }  \left\|  b_{j,k}  \right\|_{L^{\vec p \, ^\prime}} \\
		& \le \min\{ 1, 2^{ -j \sum_{i=1}^n   \frac{1}{p_i}  } \} 2^{-jn   ( 1/t - \frac{1}{n}  ( \sum_{i=1}^n  1/p_{i}) ) } .
	\end{align*}
	Thus by $r >t $, we have
	\begin{align*}
			\sum_{j =0}^\infty \sum_{k \in \mathbb Z^n, Q_{j,k} \subset Q }	 |\lambda_{j,k} | \int_Q   |b_{j,k} (x)| \d x & \le \sum_{j =0}^\infty   2^{-jn /t } \sum_{k \in \mathbb Z^n, Q_{j,k} \subset Q }	 |\lambda_{j,k} | \\
		& \le \sum_{j =0}^\infty    2^{-jn /t }  2^{jn/r} \left( \sum_{k \in \mathbb Z^n, Q_{j,k} \subset Q }	 |\lambda_{j,k} |^{r'} \right)^{1/r'} \\
		& \lesssim \| \{\lambda_{j,k} \}_{ (j,k)\in\mathbb{Z}^{n+1} } \|_{\ell^{r'}} < \infty .
	\end{align*}
	By $n / ( \sum_{i=1}^n  1/p_{i})  < t$, we have
	\begin{align*}
		& \sum_{j =-\infty}^{-1} \sum_{k \in \mathbb Z^n, Q \subset Q_{j,k} }	 |\lambda_{j,k} | \int_Q   |b_{j,k} (x)| \d x \\
		& \le \sum_{j =-\infty}^{-1} 2^{-jn   ( 1/t - \frac{1}{n}  ( \sum_{i=1}^n  1/p_{i}) ) }  \sum_{k \in \mathbb Z^n, Q \subset Q_{j,k} }	 |\lambda_{j,k} |  \\
		& = \sum_{j =-\infty}^{-1} 2^{-jn   ( 1/t - \frac{1}{n}  ( \sum_{i=1}^n  1/p_{i}) ) }   \left( \sum_{k \in \mathbb Z^n, Q \subset Q_{j,k} }	 |\lambda_{j,k} |^{r'} \right)^{1/r'} \\
		& \lesssim \| \{\lambda_{j,k} \}_{ (j,k)\in\mathbb{Z}^{n+1} } \|_{\ell^{r'}} < \infty .
	\end{align*}
	Hence, from the above two estimates, we prove (\ref{eq almost every})  and the claim holds true.
	Using this claim, we find that, for any $Q\in \D$,
	\begin{equation*}
		\left|  \int_Q \sum_{(j,k)\in\mathbb{Z}^{n+1} } \lambda_{j,k}   b_{j,k} (x)\d x \right|  \le \int_Q \sum_{(j,k)\in\mathbb{Z}^{n+1} } |\lambda_{j,k} |  |b_{j,k} (x)| \d x <\infty.
	\end{equation*} 
	Thus, $ \sum_{(j,k)\in\mathbb{Z}^{n+1} } \lambda_{j,k}   b_{j,k} $ converges  almost everywhere  on any $Q\in \D$. From the arbitrariness of $Q \in \D$, we infer that  $ \sum_{(j,k)\in\mathbb{Z}^{n+1} } \lambda_{j,k}   b_{j,k} $
	converges  almost everywhere on $\rn$. Thus we prove (i).
	
	(ii) Applying (\ref{eq almost every}) and the Lebesgue dominated convergence theorem, we have that for any $Q\in \D$,
	\begin{equation*}
		\lim_{N \to \infty} \int_Q \Big|  \sum_{ |j| < N , |k|_\infty < N}  \lambda_{j,k}   b_{j,k} (x) \Big|\d x = \int_Q \Big|  \sum_{ (j,k)\in\mathbb{Z}^{n+1} }  \lambda_{j,k}   b_{j,k} (x) \Big| \d x,
	\end{equation*}
	which, together with (\ref{eq almost every}) again, further implies that the summation (\ref{eq: mix block f}) converges in $L^1_{\operatorname{loc}  }$.
	Here $|k|_\infty = \max \{ k_1, \ldots, k_n\}$.
	Thus we prove (ii).
\end{proof}

\begin{remark}
	Based on Proposition \ref{mix block sum converge}, we find that, in Definition \ref{def mix block space}, instead of the requirement that   the summation (\ref{eq: mix block f}) converges  almost everywhere on $\rn$, if we require that (\ref{eq: mix block f}) converges in $L^1_{\operatorname{loc}  }$, we then obtain the same block space.
\end{remark}

\begin{lemma} \label{block = Lp mixed}
	Let $ 1 \le \vec p <\infty$. Let   $n / ( \sum_{i=1}^n  1/p_{i})  = t < r =  \infty $.
	Then  $\mathcal{H}_{\vec p \, ^\prime}^{t',1}  = L^{\vec p \, ^\prime} $ with coincidence
	of norms.
\end{lemma}

\begin{proof}
	We use the idea from	\cite[Proposition 339]{SDH20}.
	If $f \in \mathcal{H}_{\vec p \, ^\prime}^{t',1} $. Then there exist a sequence $ \{ \lambda_{j,k}\} _{  (j,k)  \in \mathbb Z^{1+n} } \in  \ell^1  $ and a sequence $\{ b_{j,k} \}_{(j,k)  \in \mathbb Z^{1+n}  }$ of  $(\vec p \, ^\prime,t') $-blocks such that  $f = \sum_{ (j,k) \in \mathbb Z^{1+n}  }   \lambda_{j,k}  b_{j,k} $  and
	\begin{equation*}
		\sum_{  (j,k) \in \mathbb Z^{1+n} }    | \lambda_{j,k}|    \le  (1+\epsilon)	\| f\|_{\mathcal{H}_{\vec p \, ^\prime}^{t',1}  }.
	\end{equation*}
	Then by Minkowski's inequality, we have
	\begin{align*}
		\left\| f  \right\|_{ L^{\vec p \, ^\prime}  }  
		& \le \sum_{ (j,k) \in \mathbb Z^{1+n}  }   | \lambda_{j,k}  | \|  b_{j,k} \|_{ L^{\vec p \, ^\prime}  } 
		\le 	\sum_{  (j,k) \in \mathbb Z^{1+n} }    | \lambda_{j,k}|    \le  (1+\epsilon)	\| f\|_{\mathcal{H}_{\vec p \, ^\prime}^{t',1}  } .
	\end{align*}	
	For the another direction, let $f \in  L^{\vec p \, ^\prime} $. Fix $\epsilon>0$. Then there is a decomposition 
	\begin{equation*}
		f = \chi_{Q_1} f + \sum_{j=2}^\infty  \chi_{Q_j \backslash Q_{j-1}} f, 
	\end{equation*}
	where $ \{ Q_j\}_{j=1}^\infty  $ is an increasing sequence of cubes centered at the origin satisfying 
	\begin{equation*}
		\| \chi_{\rn \backslash  Q_k } f \|_{  L^{\vec p \, ^\prime}  } \le \epsilon 2^{-k}
	\end{equation*}
	for all $k \in \mathbb N$. Then using this decomposition, we have
	\begin{equation*}
		\| f\|_{\mathcal{H}_{\vec p \, ^\prime}^{t',1}  }  \le \| f\|_{ L^{p'}  }  + \sum_{k=1}^\infty \epsilon 2^{-k}.
	\end{equation*}
	Thus $f \in \mathcal{H}_{\vec p \, ^\prime}^{t',1} $.  Thus the proof is complete.
\end{proof}

The following lemma indicates how to generate blocks.
\begin{lemma} \label{generate mix block}
	Let $ 1 < \vec p <\infty$. Let   $1 < n / ( \sum_{i=1}^n  1/p_{i})   < t <r<\infty $ or $ 1< n / ( \sum_{i=1}^n  1/p_{i})  \le t< r =\infty  $.
	Let $f$ be an $L^{\vec  p \; ^\prime } $ function supported	on a cube $Q \in \D$. Then $\| f\|_ { \H_{\vec p \, ^\prime}^{t',r'} }  \le \| f\| _{ L^{\vec p \, ^\prime} }   |Q| ^{ \frac{1}{n}  ( \sum_{i=1}^n  1/p_{i})  -  1/t  }   $.
\end{lemma}
\begin{proof}
	Assume that $f$ is not zero almost everywhere, let 
	\begin{equation*}
		b := \frac{   |Q| ^{  1/t - \frac{1}{n}  ( \sum_{i=1}^n  1/p_{i})   }    }{ \|f\|_{L^{\vec p \, ^\prime}    }  } f.
	\end{equation*}
	Then $b$ is supported on $Q \in \D$, and 
	\begin{equation*}
		\| b \|_{L^{\vec  p'}  }  =     |Q| ^{  1/t - \frac{1}{n}  ( \sum_{i=1}^n  1/p_{i})   }     .  
	\end{equation*}
	Hence $b$ is a $ (\vec p \, ^\prime,t')$-block and $\|b\|_{  \H_{\vec p \, ^\prime}^{t',r'} }  \le  1$. 
	Equating this inequality, we obtain the desired result.
\end{proof}

\begin{theorem} \label{dense block}
	Let $ 1 < \vec p <\infty$. Let   $1 < n / ( \sum_{i=1}^n  1/p_{i})   < t <r<\infty $ or $ 1< n / ( \sum_{i=1}^n  1/p_{i})  \le t< r =\infty  $. Then $L^{\vec p \, ^\prime}  _c $ is dense in $\H_{\vec p \, ^\prime}^{t',r'}  $.
\end{theorem}
\begin{proof}
	We use the idea from \cite[Theorem 345]{SDH20}.
	Since $ f \in \H_{\vec p \, ^\prime}^{t',r'} $, there exist a sequence $\{ \lambda_{j,k} \}_{ (j,k) \in \mathbb Z ^{1+n} }  \in \ell ^{r'} $  and a sequence  $(\vec p \, ^\prime,t')$-block $\{  b_{j,k} \}  _{ (j,k) \in \mathbb Z ^{1+n} } $ 
	such that $f 	=\sum_{(j,k)\in\mathbb{Z}^{n+1}}\lambda_{j,k}b_{j,k} $  and $\| \{ \lambda_{j,k} \} \|_{\ell ^{r'}}  \le \| f\|_{\H_{\vec p \, ^\prime}^{t',r'} }    +\epsilon$.	
	Define 
	\begin{equation} \label{f_N dense}
		f_N = \sum_{|(j,k)|_\infty  \le N}\lambda_{j,k}b_{j,k},
	\end{equation}
	where $|(j,k)|_\infty  = \max  \{ |j|, |k_1|, \ldots, |k_n| \}  $.
	Since $r' \in [1,\infty)$, we get
	\begin{equation*}
		\| f -f_N \|_{ \H_{\vec p \, ^\prime}^{t',r'} }  \le  \left( \sum_{|j| \le N ,|k|_\infty  > N} | \lambda_{j,k} |^{r'}  + \sum_{ |j| > N ,|k|_\infty  \le  N}| \lambda_{j,k} |^{r'}  \right)^{1/r'}   \to 0 
	\end{equation*}
	as $N\to \infty.$ It is not hard to show $f_N \in L^{\vec p \, ^\prime } $,  (for example, see \cite[Theorem 4.3]{BGX25}).	 
	Hence
	$L^{\vec p \, ^\prime}  _c $ is dense in $\H_{\vec p \, ^\prime}^{t',r'} $. 
\end{proof}

From the above theorem, when we investigate $  \H_{\vec p \, ^\prime}^{t',r'}$, the space $L^{\vec p \, ^\prime}  _c $ is a useful space. When considering the action of  linear operators defined  and continuous on    $ \H_{\vec p \, ^\prime}^{t',r'}$ and $L^{\vec p \, ^\prime}  _c $, it will be helpful to have a finite decomposition in $L^{\vec p \, ^\prime}  _c  $.
The following result says that each function $f \in L^{\vec p \, ^\prime}  _c $ has a finite admissible expression. That is, the sum is finite.
\begin{theorem}\label{finite decom}
	Let $ 1 < \vec p <\infty$. Let   $1 < n / ( \sum_{i=1}^n  1/p_{i})   < t <r<\infty $ or $ 1< n / ( \sum_{i=1}^n  1/p_{i})  \le t< r =\infty  $. Then each $ f \in L^{\vec p \, ^\prime}  _c $ admits the finite decomposition $ f  = \sum_{v=1}^M \lambda_v b_v $ where $ \lambda_1, \lambda_2, \ldots, \lambda_M \ge 0 $ and each $ b_v $ is a $(\vec p \, ^\prime,t')$-block. Furthermore, 
	\begin{equation*}
		\| f\|_{\H_{\vec p \, ^\prime}^{t',r'}} \approx \inf_\lambda \left( \sum_{v=1}^M \lambda_v ^{r'} \right)^{1/r'},
	\end{equation*}
	where $ \lambda = \{ \lambda_v\}_{v=1}^M $ runs over all finite admissible expressions:
	\begin{equation*}
		f= \sum_{v=1}^M \lambda_v b_v,
	\end{equation*}
	$ \lambda_1, \lambda_2, \ldots, \lambda_M \ge 0 $ and each $ b_v $ is a $(
	\vec p \, ^\prime,t')$-block for each $v = 1,2,\ldots ,M$.
\end{theorem}
\begin{proof}
	The proof is similar to \cite[Theorem 4.4]{BGX25} and we omit it here.	
\end{proof}

\begin{remark}
		Let $ 1 < \vec p <\infty$. Let   $1 < n / ( \sum_{i=1}^n  1/p_{i})   < t <r<\infty $ or $ 1< n / ( \sum_{i=1}^n  1/p_{i})  \le t< r =\infty  $. Then  $\H_{\vec p \, ^\prime}^{t',r'}  $ is separable. The proof is similar with Theorem \ref{separable mixed BM}.
\end{remark}

In \cite[Theorem 2.7]{N192}, Nogayama showed that the predual space of  $M_{\vec p}^{t,\infty} $ is $\mathcal{H}_{\vec p \, ^\prime}^{t',1} $.
Before proving the following result,   we give notation for the mixed Lebesgue norm $\| \cdot \|_{L^{\vec p} }$. The mapping
\begin{equation*}
	(x_2, \ldots , x_n) \in \mathbb R^{n-1} \mapsto \| f\|_{ (p_1) }  (x_2, \ldots , x_n) := \left(\int_{\mathbb R}  |f (x_1, x_2, \ldots, x_n) |^{p_1} \d x_1 \right) ^{1/p_1}
\end{equation*}
is a measurable function  on  $\mathbb R^{n-1}$. Moreover, let
\begin{equation*}
	\| f \|_{ \vec q } =\| f \|_{ (p_1, \ldots, p_j) } := \left\|  [ \| f\|_{ (p_1, \ldots, p_{j-1})  } ]  \right\|_{  (  p_j ) },
\end{equation*}
where $\| f \|_{ (p_1, \ldots, p_{j-1}) }$ denotes $|f|$ if $j=1$  and $\vec q = (p_1, \ldots, p_j), j\le n$.
Then $ \| f \|_{ \vec q  }$  is a measurable function of $ (x_{j+1}, \ldots, x_n )$ for $j <n$.

\begin{theorem} \label{predual mix BM}
	Let $ 1 < \vec p <\infty$. Let   $1 < n / ( \sum_{i=1}^n  1/p_{i})   < t <r<\infty $ or $ 1< n / ( \sum_{i=1}^n  1/p_{i})  \le t< r =\infty  $.
	Then the dual  space of
	$\mathcal{H}_{\vec p \, ^\prime}^{t',r'} $, denoted by $ \Big( \mathcal{H}_{\vec p \, ^\prime}^{t',r'} \Big)^* $, is
	$M_{\vec p}^{t,r} $, that is,  
	\begin{equation*}
		\Big(  \mathcal{H}_{\vec p \, ^\prime}^{t',r'} \Big)^*   = M_{\vec p}^{t,r} 
	\end{equation*}
	in the following sense:
	
	{\rm (i)} if $f \in M_{\vec p}^{t,r} $, then the linear functional 
	\begin{equation} \label{Jf}
		L_f : f \to  L_f (g) :=   \int_\rn   g (x) f (x)  \d x
	\end{equation}
	is bounded on $ \mathcal{H}_{\vec p \, ^\prime}^{t',r'} $.
	
	{\rm (ii)} conversely, any continuous linear functional on $\mathcal{H}_{\vec p \, ^\prime}^{t',r'} $ arises as in  (\ref{Jf}) with a unique $f \in  M_{\vec p}^{t,r}  $.
	
	Moreover, $\|f\|_{   M_{\vec p}^{t,r}  }  = \| L_f \|_{ ( \mathcal{H}_{\vec p \, ^\prime}^{t',r'}  )^*}$ and 
	\begin{equation} \label{H = max M le 1}
		\|g\|_{\mathcal{H}_{\vec p \, ^\prime}^{t',r'}   } =\max \left\{\left|\int_\rn  g (x) f (x)  \d x    \right|: f\in   M_{\vec p}^{t,r} , \|f\|_{   M_{\vec p}^{t,r} } \le 1 \right\}.
	\end{equation}
\end{theorem}
\begin{proof}
	(i) Let $g \in \mathcal{H}_{\vec p \, ^\prime}^{t',r'} $ and $\epsilon>0$. Then $g = \sum_{(j,k)\in\mathbb{Z}^{n+1}}\lambda_{j,k}b_{j,k}$ with $\lambda=\{\lambda_{j,k}\}_{(j,k)\in\mathbb{Z}^{n+1}}\in\ell^{r'}(\mathbb{Z}^{n+1})$
	and $b_{j,k}$ is a  $(\vec p \, ^\prime,t')$-block supported on $Q_{j,k}$ such that 
	\begin{equation*}
		\left( \sum_{(j,k)\in\mathbb{Z}^{n+1}} |\lambda_{j,k}|^{r'}\right)^{1/r'}  \le (1+\epsilon) \| g\|_{ \mathcal{H}_{\vec p \, ^\prime}^{t',r'} }.
	\end{equation*}
	By  H\"older's inequality,
	\begin{align} \label{f g le M H}
		\nonumber
		\left| \int_\rn g (x) f (x)  \d x\right| &  = 	\left|	\int_\rn   \sum_{(j,k)\in\mathbb{Z}^{n+1}}\lambda_{j,k}b_{j,k} (x)   f (x)  \d x\right| \\
		\nonumber
		& \le \sum_{(j,k)\in\mathbb{Z}^{n+1}} |\lambda_{j,k}| \int_{Q _{j,k} }   | b_{j,k} (x)| | f (x)| \d x \\
		\nonumber
		& \le \sum_{(j,k)\in\mathbb{Z}^{n+1}} |\lambda_{j,k}| |Q| ^{  1/t - \frac{1}{n}  ( \sum_{i=1}^n  1/p_{i})   }  \| f \|_{L^{\vec p } }    \\
		\nonumber
		& \le \left( \sum_{(j,k)\in\mathbb{Z}^{n+1}} |\lambda_{j,k}|^{r'}\right)^{1/r'} \| f\|_{M_{\vec p}^{t,r}  } \\
		& \le (1+\epsilon) \| g\|_{  \mathcal{H}_{\vec p \, ^\prime}^{t',r'}  }  \| f\|_{M_{\vec p}^{t,r}   }.
	\end{align}
	Letting $\epsilon\to 0^+$,  we prove (i).
	
	(ii) Let $L$ be a continuous linear functional on $\mathcal{H}_{\vec p \, ^\prime}^{t',r'}$.
	
	By Lemma \ref{generate mix block}, since we can regard the element of $ L^{\vec p \, ^\prime} (Q_{j,k})$ as a $ (\vec p \, ^\prime,t')$-block modulo multiplicative constant, the functional $g \mapsto L(g) $  is well defined and bounded on $ L^{\vec p \, ^\prime} (Q_{j,k})$. Thus, by the $L^{\vec p} -L^{\vec{p}^{\,\prime}} $  duality (Lemma \ref{dual mixed Lp}), there exists $f_{j,k} \in  L^{\vec p} (Q_{j,k})$ such that
	\begin{equation} \label{L g jk}
		L(g) = \int_{Q_{j,k}} f_{j,k} (x) g(x) \d x
	\end{equation}
	for all $g \in L^{\vec p \, ^\prime} (Q_{j,k})$. By the uniqueness of this theorem, we can find $  L^{\vec p}_{\operatorname{loc}} $-function $f$ such that 
	\begin{equation*}
		f \chi_{Q_{j,k}}  = f_{j,k} \quad \operatorname{a.e.}
	\end{equation*}
	for any $Q_{j,k} \in \D $. We shall prove $f \in M_{\vec p}^{t,r}$. 
	
	Fix  a finite set $K \subset \mathbb Z^{n+1}$. For each $(j,k) \in K$, 
	let 
	\begin{equation*}
		g_{j,k} := 
		\frac{|Q_{j,k} |^{  1/t - \frac{1}{n}  ( \sum_{i=1}^n  1/p_{i})  } }{ \| f\chi_{ Q_{j,k} } \|_{L^{\vec p}  } ^{q_n -1} } (\overline{ \sgn f}) |f|^{p_1 -1} \chi_{ Q_{j,k} }  \| f \chi_{ Q_{j,k} } \|_{(q_1) }^{q_2 -q_1} \cdots \| f \chi_{ Q_{j,k} } \|_{(q_1, \ldots , q_{n-1} ) }^{q_n -q_{n-1} }   
	\end{equation*}
	if $ \| f\chi_{ Q_{j,k} } \|_{L^{\vec p}  } >0 $  and let $g_{j,k} = 0$ if   $ \| f\chi_{ Q_{j,k} } \|_{L^{\vec p}  } =0 $.
	Then if $ \| f\chi_{ Q_{j,k} } \|_{L^{\vec p}  } >0 $,
	\begin{equation} \label{int f g jk}
		\int_{Q_{j,k}} f (x) g_{j,k}(x) \d x = |Q_{j,k} |^{  1/t - \frac{1}{n}  ( \sum_{i=1}^n  1/p_{i})  } 	\| f\chi_{ Q_{j,k} }\|_{L^{\vec p}  } ,
	\end{equation}
	and 
	\begin{equation*}
		\| g_{j,k} \|_{L^{\vec{p}^{\,\prime} }  }  =  |Q_{j,k} |^{  1/t - \frac{1}{n}  ( \sum_{i=1}^n  1/p_{i})  } .
	\end{equation*}
	Note that $   g_{j,k} $  is a $(\vec p \, ^\prime , t') $-block supported on $Q_{j,k}$.
	Take an arbitrary nonnegative
	sequence $\{  \lambda_{j,k}\} \in \ell^{r'} (\mathbb Z^{1+n})$
	supported on $K$ and set
	\begin{equation}\label{g K decom}
		g_{K}  = \sum_{(j,k) \in K}  \lambda_{j,k}   g_{j,k}  \in \mathcal{H}_{\vec p \, ^\prime}^{t',r'}  .
	\end{equation}
	Then  from  (\ref{L g jk}), (\ref{int f g jk}), the fact that $K$ is a finite set, and the linearity of $L$, we get 
	\begin{align*}
		&	\sum_{(j,k) \in K}  \lambda_{j,k}  |Q_{j,k}| ^{1/t - \frac{1}{n}  ( \sum_{i=1}^n  1/p_{i}) } \left\|   f \chi_{Q_{j,k}}    \right\|_{L^{\vec p} } \\
		& = \sum_{(j,k) \in K}  \lambda_{j,k}  \int_{Q_{j,k}} f (x) g_{j,k}(x) \d x  \\
		& =   \int_\rn   g_{K} (x) f(x ) \d x \\
		& = L (g_{K}  ).
	\end{align*}
	By the decomposition (\ref{g K decom}) and $L$ is a continuous linear functional on $ \mathcal{H}_{ \vec p \, ^\prime}^{t',r'} $, we obtain
	\begin{equation*}
		L (g_{K}  )  \le  \| L \|_{ (\mathcal{H}_{\vec p \, ^\prime}^{t',r'})^*  }  \| g_{K} \|_{  \mathcal{H}_{\vec p \, ^\prime}^{t',r'}  } \le \| L \|_{ (\mathcal{H}_{\vec p \, ^\prime}^{t',r'})^*  } \left( \sum_{(j,k) \in K}   \lambda_{j,k} ^ {r'} \right)^{1/r'}.
	\end{equation*}
	Since $r >1$ and $K  \subset \mathbb Z^{1+n}$ and $ \{\lambda_{j,k} \}_{(j,k) \in K } $ are arbitrary, we conclude 
	\begin{align*}
		& \| f \|_{M_{\vec p}^{t,r}  }  \\
		& = \sup \left\{  \sum_{ (j,k) \in \mathbb Z^{1+n}  }  \lambda_{j,k} |Q_{j,k}| ^{1/t - \frac{1}{n}  ( \sum_{i=1}^n  1/p_{i}) } \left\|   f \chi_{Q_{j,k}}    \right\|_{L^{\vec p} } :  \| \{\lambda_{j,k} \} \|_{\ell^{r' } (\mathbb Z^{1+n})}  = 1     \right\} \\
		&  \le   \| L \|_{ (\mathcal{H}_{\vec p \, ^\prime}^{t',r'})^*  }  < \infty .
	\end{align*}
	Thus $f \in M_{\vec p}^{t,r}$ and $\| f \|_{M_{\vec p}^{t,r}  } \le  \| L \|_{ (\mathcal{H}_{\vec p \, ^\prime}^{t',r'})^*  } $. Together (i), we obtain $\| f \|_{M_{\vec p}^{t,r}  } =  \| L \|_{ (\mathcal{H}_{\vec p \, ^\prime}^{t',r'})^*  } $ 
	
	Hence we conclude that $L$  is realized as $L = L_f$ for $f \in M_{\vec p}^{t,r}$ at least on $g \in L_c^{\vec p} $. Since $ L_c^{\vec p}$  is dense in $\mathcal{H}_{\vec p \, ^\prime}^{t',r'}$ by Theorem \ref{dense block}, we can obtain the desired result.
	
	Next we show that $f$ is unique. Suppose that there exists another $\tilde f \in M_{\vec p}^{t,r}$ such that $L$ arises as in (\ref{Jf}) with $f $ replaced by $\tilde f$.
	Then for any $ Q \in \D$, since $\chi_Q \overline{\sgn ( f(x) - \tilde f (x) ) }  \in  \mathcal{H}_{\vec p \, ^\prime}^{t',r'} $ , we have
	\begin{equation*}
		\int_Q	 \overline{\sgn ( f(x) - \tilde f (x) ) } ( f(x) - \tilde f (x) )  \d x = 	\int_Q	 | f(x) - \tilde f (x) |  \d x =  0.
	\end{equation*}
	This, combined with the arbitrariness of $ Q \in \D$ further implies that $f=  \tilde f $ almost everywhere.
	
	Using  \cite[Theorem 87, Existence of the norm attainer]{SDH20}, (\ref{H = max M le 1}) is included in what we have proven.
	Thus we complete the proof.
\end{proof}

\section{Properties of  block spaces}\label{property block}

In this section, we discuss the completeness,  Fatou property, lattice and associate space of 
the function space $\mathcal{H}_{\vec p \, ^\prime}^{t',r'} $.

\subsection{The completeness} \label{banach space}

We first show the completeness.

\begin{theorem} \label{Banach}
	Let $ 1 < \vec p <\infty$. Let   $1 < n / ( \sum_{i=1}^n  1/p_{i})   < t <r<\infty $ or $ 1< n / ( \sum_{i=1}^n  1/p_{i})  \le t< r =\infty  $. Then $\mathcal{H}_{\vec p \, ^\prime}^{t',r'} $ is a Banach space.
\end{theorem}

\begin{proof}
	We only  prove the completeness  since others are easy. 
	
	Let $ \{ f_i\}_{i\in \mathbb N}$  be a sequence such that for each $i \in \mathbb N$, $f_i \in \mathcal{H}_{\vec p \, ^\prime}^{t',r'}  $, $ \sum_{i\ge 1} \| f_i \|_{ \mathcal{H}_{\vec p \, ^\prime}^{t',r'}   }  <\infty  $.
	From  Theorem \ref{predual mix BM},
	for any ball $B$, we have
	\begin{equation*}
		\int_{B} \sum_{i \ge 1} |f_i (y) | \d y =\sum_{i \ge 1} \int_{B}  | f_i (y) | \d y \le \| \chi_B \|_{  M _{\vec p}^{t,r}   }  \sum_{i \ge 1} \|    f_i \|_{ \mathcal{H}_{\vec p \, ^\prime}^{t',r'}  }.
	\end{equation*}
	From $|  \sum_{i \ge 1}  f_i |  \le \sum_{i \ge 1} | f_i  |$,
	we obtain  $ f = \sum_{i \ge 1}  f_i  $ is a well defined,  measurable function and  $f\in  L^1_{\operatorname{loc}} $. So we only need to show that  $ f = \sum_{i \ge 1}  f_i  $ belongs to $\mathcal{H}_{\vec p \, ^\prime}^{t',r'} $.
	Fix $\epsilon > 0$. There exists a positive integer $N_0$ such that for all $N \ge N_0$, there holds
	\begin{equation*}
		\sum_{i \ge N } \|    f_i \|_{ 	\mathcal{H}_{\vec p \, ^\prime}^{t',r'}  }  <\epsilon.
	\end{equation*}
	For this $\epsilon > 0$, there exists a decomposition 
	\begin{equation*}
		f _i=  \sum_{(j,k)\in \mathbb Z ^{n+1}} \lambda_{i,j,k} b_{i,j,k},
	\end{equation*}
	where  $b_{i,j,k}$ is a  $(\vec p \, ^\prime,t')$-block supported on $Q_{j,k}$ and 
	\begin{equation*}
		\left(    \sum_{ (j,k)\in \mathbb Z ^{n+1}  }   | \lambda_{i,j,k}| ^{r'} \right)^{1/r'}  \le (1+\epsilon)  \|    f_i \|_{ \mathcal{H}_{\vec p \, ^\prime}^{t',r'}  }.
	\end{equation*}
	Furthermore, for any $ 1\le i \le N_0 $, there exists a index set $ M _i  \subset  \mathbb Z ^{n+1} $  such that 
	\begin{equation*}
		\bigg\|  f_i  - \sum_{(j,k)\in M_i} \lambda_{i,j,k} b_{i,j,k}  \bigg\|_{ \mathcal{H}_{\vec p \, ^\prime}^{t',r'}  } \le \bigg(  \sum_{(j,k)\in   \mathbb Z ^{n+1}  \backslash  M_i} |\lambda_{i,j,k}|^{r'} \bigg) ^{1/r'} < 2^{-i} \epsilon.
	\end{equation*}
	Therefore, for all ball $B$, 
	\begin{align*}
		&	\int_B \bigg| f(y)  -  \sum_{i=1}^{N_0}  \sum_{(j,k)\in M_i} \lambda_{i,j,k} b_{i,j,k} (y) \bigg| \d y \\
		& \le \int_B  \bigg| f(y) -\sum_{i=1}^{N_0} f_i (y)  \bigg| \d y + \int_B  \bigg| \sum_{i=1}^{N_0} f_i (y) - \sum_{i=1}^{N_0}  \sum_{(j,k)\in M_i} \lambda_{i,j,k} b_{i,j,k} (y)  \bigg| \d y  \\
		& \le  \int_B  \bigg| \sum_{i=N_0+1}^{\infty } f_i (y)  \bigg|\d y +  \sum_{i=1}^{N_0} \int_B  \bigg|  f_i (y) -   \sum_{(j,k)\in M_i} \lambda_{i,j,k} b_{i,j,k} (y)  \bigg| \d y  \\
		& \le \| \chi_B \|_{M_{\vec p}^{t,r} } \bigg( \sum_{i=N_0+1}^{\infty }  \|f_i\|_{ \mathcal{H}_{\vec p \, ^\prime}^{t',r'}   }  +   \sum_{i=1}^{N_0}  \bigg\|  f_i  -   \sum_{(j,k)\in M_i} \lambda_{i,j,k} b_{i,j,k}  \bigg\|_{\mathcal{H}_{\vec p \, ^\prime}^{t',r'}   }  \bigg) \\
		& \le \| \chi_B \|_{M_{\vec p}^{t,r} } ( \epsilon +\sum_{i=1}^{N_0}  2^{-i} \epsilon  ) \\
		& \lesssim  \| \chi_B \|_{M_{\vec p}^{t,r} }  \epsilon.
	\end{align*}
	As a consequence,  $\sum_{i\ge 1} \sum_{(j,k)\in \mathbb Z ^{n+1}} \lambda_{i,j,k} b_{i,j,k}$ converges to $f$ in  $L^1_{\operatorname{loc}}  $.  	
		
	Hence  $\sum_{i\ge 1} \sum_{(j,k)\in \mathbb Z ^{n+1}} \lambda_{i,j,k} b_{i,j,k}$ converges to $f$ locally in measure. Therefore, there exists a subsequence of  
	\begin{equation*}
		\left\{\sum_{i= 1}^{N} \sum_{(j,k)\in K \subset \mathbb Z ^{n+1} , \sharp K = M} \lambda_{i,j,k} b_{i,j,k}  \right\}_{ N \in \mathbb N ,M \in \mathbb N}
	\end{equation*}
	converges to $f$ a.e. 
	Recall that $\sharp K$ is the cardinal number of the set $K$.
	Furthermore, $\lambda_{i,j,k}$, $i\in \mathbb N, (j,k) \in \mathbb Z ^{n+1}$  satisfies
	\begin{equation*}
		\sum_{i\ge 1}	\bigg(    \sum_{ (j,k)\in \mathbb Z ^{n+1}  }   | \lambda_{i,j,k}| ^{r'} \bigg)^{1/r'}  \le \sum_{i\ge 1} (1+\epsilon)  \|    f_i \|_{ \mathcal{H}_{\vec p \, ^\prime}^{t',r'}  }.
	\end{equation*}
	That is, $ \sum_{i\ge  1}  f_i$ converges to $f$ in  $\mathcal{H}_{\vec p \, ^\prime}^{t',r'} $. Since $\epsilon$ is arbitrary, we obtain
	\begin{equation*}
		\| f\|_{ \mathcal{H}_{\vec p \, ^\prime}^{t',r'}   }  =  	\bigg\| \sum_{i\ge  1}  f_i \bigg\|_{ \mathcal{H}_{\vec p \, ^\prime}^{t',r'}   }  \le \sum_{i\ge  1} \|    f_i \|_{ \mathcal{H}_{\vec p \, ^\prime}^{t',r'}  }.
	\end{equation*}
	Then by the fact that a normed linear space is compete if and only if every absolutely summable sequence is summable (for example, see \cite[Theorem III.3]{RS72}), 	$\mathcal{H}_{\vec p \, ^\prime}^{t',r'}  $ is complete.
\end{proof}

\subsection{Fatou property}
In this subsection, the main result is the following theorem, which is called the Fatou property of $\H_{\vec p \, ^\prime}^{t',r'}$. Before proving it, we first recall the following lemma.

Using Banach and Alaoglu Theorem (see, for example \cite[Theorem 89]{SDH20}) and duality of mixed $L^{\vec p}$ (Lemma \ref{dual mixed Lp}), we obtain the following result.
\begin{lemma}\label{weak* compct}
 Let  $1<\vec  p <\infty $ and $\{ f_j \}_{j=1}^\infty$  be a bounded sequence in $L^{\vec p} $. Namely, there exists $M>0$ such that $ \|f_j\|_{L^{\vec p}} \le M$. Then we can find $f \in L^{\vec p}$  and a subsequence $\{ f_{j_k} \}_{k=1}^\infty$  such that 
	\begin{equation*}
		\lim_{k\to \infty} \int_\rn  f_{j_k}(x)  g (x)  \d x = \int_\rn  f(x)  g (x) \d  x
	\end{equation*}
	for all $g \in L^{\vec p \, ^\prime}  $.
\end{lemma}
\begin{theorem} \label{Fatou general}
		Let $ 1 < \vec p <\infty$. Let   $1 < n / ( \sum_{i=1}^n  1/p_{i})   < t <r<\infty $ or $ 1< n / ( \sum_{i=1}^n  1/p_{i})  \le t< r =\infty  $.
	If a bounded sequence $\{f_\ell \}_{\ell\in \mathbb N} \subset \H_{\vec p \, ^\prime}^{t',r'} \cap L_{\operatorname{loc}} ^{\vec p \, ^\prime}$ converges locally to $f$  in the weak topology of $  L^{\vec p \, ^\prime} $, then $f \in \H_{\vec p \, ^\prime}^{t',r'}$ and 
	\begin{equation*}
		\|f\|_{ \H_{\vec p \, ^\prime}^{t',r'} }  \le \liminf _{\ell \to \infty} \|f_\ell\|_{ \H_{\vec p \, ^\prime}^{t',r'} } .
	\end{equation*}
\end{theorem}

\begin{proof}
	We use the idea from \cite[Lemma 5.4]{HNSH23}. 
	By Lemma \ref{block = Lp mixed}, $\H_{\vec p \, ^\prime}^{t',r'} = L^{\vec p \, ^\prime} $ when $ 1 < n / ( \sum_{i=1}^n  1/p_{i}) = t <r =\infty$.
 We only prove the case  $1<  n / ( \sum_{i=1}^n  1/p_{i})<t<r<\infty$ or $1<  n / ( \sum_{i=1}^n  1/p_{i}) < t<r=\infty$.
	
	By normalization, we assume $\|f_\ell\|_{ \H_{\vec p \, ^\prime}^{t',r'} } \le 1$ for all $\ell \in \mathbb N$. It suffices to show  $f \in \H_{\vec p \, ^\prime}^{t',r'}$ and   $\|f\|_{ \H_{\vec p \, ^\prime}^{t',r'} } \le 1$.
	We write 
	\begin{equation*}
		f_\ell = \sum_{Q \in \D} \lambda_{\ell,Q} b_{\ell,Q}
	\end{equation*}
	where $b_{\ell,Q} $ is a $ (\vec p \, ^\prime,t' ) $-block supported on $ Q $ and  $ \left( \sum_{Q\in \D} | \lambda_{\ell,Q} |^{r'} \right)^{1/r'} \le  \| f_\ell \|_{  \H_{\vec p \, ^\prime}^{t',r'} }  +\epsilon \le 1+\epsilon $.
	We may suppose that $  \lambda_{\ell,Q} \ge 0$ since $ - b_{\ell,Q} $ is also a $ (\vec p \, ^\prime,t' ) $-block supported on $ Q $.	
	Using Lemma \ref{weak* compct}, by passing to a subsequence, we assume that 
	\begin{equation*}
		\lambda_Q := \lim_{\ell \to \infty} \lambda_{\ell,Q}
	\end{equation*}
	exists in $[0,\infty)$ and that
	\begin{equation*}
		b_Q := \lim_{\ell \to \infty} b_{\ell,Q}
	\end{equation*}
	exists in the weak topology of $ L^{\vec p \, ^\prime} $.
	It is easy to see that supp $b_Q \subset Q$. By the Lemma \ref{lem fatou}, 
	\begin{align*}
		\| b_Q \|_{L^{\vec p \, ^\prime}} & =\sup_{ \|h\|_{L^{\vec p} } \le 1 } \left| \int_\rn  b_Q (x) h(x)  \d x \right|  \\
		& \le \liminf_{ \ell \to \infty } \sup_{ \|h\|_{L^{\vec p} } \le 1 } \left| \int_\rn  b_{\ell,Q} (x) h(x)  \d x \right|  \\
		& \le |Q|^{1/t-\frac{1}{n}  ( \sum_{i=1}^n  1/p_{i})} .
	\end{align*}
	Hence $b_Q$ is a $ (\vec p \, ^\prime,t' ) $-block supported on $ Q $. 
	Let 
	\begin{equation*}
		g : = \sum_{Q\in \D} \lambda_Q b_Q.
	\end{equation*}
	We claim that $f=g$. Once this is achieved, we obtain $ f \in  \H_{\vec p \, ^\prime}^{t',r'}$ and
	\begin{equation} \label{f liminf 1+ eps}
		\|f\|_{ \H_{\vec p \, ^\prime}^{t',r'} } \le \left( \sum_{Q\in \D} |\lambda_Q|^{r'} \right)^{1/r'} \le \liminf_{\ell \to \infty}   \left( \sum_{Q\in \D} | \lambda_{\ell,Q} |^{r'} \right)^{1/r'} \le 1+\epsilon.
	\end{equation}
	By the Lebesgue differentiation theorem, this amounts to the proof of the equality:
	\begin{equation*}
		\int_Q f (x)\d x = \int_Q g (x)\d x
	\end{equation*}
	for all $Q \in \D$. Now fix $Q_0 \in \D$. Let $ j_{Q_0} := -\log_2 ( \ell (Q_0)) $. Then $ Q_0 \in \D _{ j_{Q_0}}$.
	By the definition of $ f_{\ell} $ and the fact that the sum defining $ f_{\ell} $ converges almost everywhere on $\rn$, we have
	\begin{align*}
		\int_{Q_0} f_{\ell} (x) \d x = \sum_{Q \in \D} \lambda_{\ell,Q} \int_{Q_0}   b_{\ell,Q} (x) \d x .
	\end{align*}
	By H\"older's inequality, we have
	\begin{align*}
		\left|  \int_{Q_0}  b_{\ell,Q} (x) \d x \right| &  \le \int_{Q_0}  | b_{\ell,Q} (x)| \d x  \\
		& \le  \|b_{\ell,Q} \|_{L^{ \vec{p}^{\,\prime}} }  \| \chi_{Q \cap Q_0 } \|_{L^{\vec p} }  \\
		& \le |Q|^{1/t-\frac{1}{n}  ( \sum_{i=1}^n  1/p_{i}) }  |Q \cap Q_{0} |^{\frac{1}{n}  ( \sum_{i=1}^n  1/p_{i})} .
	\end{align*}
	Since  $1/t-\frac{1}{n}  ( \sum_{i=1}^n  1/p_{i}) <0$  and $ 1- r/t < 0$,
	\begin{align*}
		& \left(  \sum_{m = -\infty} ^\infty \sum_{Q \in \D_m, Q \cap  Q_{0} \neq \emptyset }|Q|^{r/t-\frac{r}{n}  ( \sum_{i=1}^n  1/p_{i}) }  |Q \cap Q_{0} |^{\frac{r}{n}  ( \sum_{i=1}^n  1/p_{i})}\right)^{1/r} \\
		&= \left(\sum_{m= -\infty} ^{ j_{Q_0} } 2^{ -m r n (1/t-\frac{1}{n}  ( \sum_{i=1}^n  1/p_{i})  )} 2^{ - r j_{Q_0}   ( \sum_{i=1}^n  1/p_{i}) }  +\sum_{ m =j_{Q_0}  +1 } ^\infty   2^{ - j_{Q_0} n} 2^{m n  (1 -r/t) } \right)^{1/r} \\
		& \le  C <\infty.
	\end{align*}
	Thus \begin{align*}
		\left|  \int_{Q_0} f_{\ell } (x) \d x \right| \le C \left( \sum_{m = -\infty} ^\infty \sum_{Q \in \D _m , Q \cap Q_0  \neq \emptyset }  ( \lambda_{\ell,Q}  )^{r'} \right)^{1/r'} \le  C (1+\epsilon) <\infty.
	\end{align*}
	Now we can  use  the Lebesgue convergence theorem to  obtain
	\begin{align*}
		\lim_{\ell \to \infty} \int_{Q_0}  f_{\ell} (x) \d x  
		= \sum_{m= -\infty} ^\infty \left( \lim_{\ell \to \infty} \sum_{Q \in \D _m , Q \cap Q_0  \neq \emptyset } \lambda_{\ell,Q}  \int_{Q_0} b_{\ell,Q} (x) \d x    \right).
	\end{align*}
	Since 
	$
	\sum_{Q \in \D _m , Q \cap Q_0  \neq \emptyset }
	$
	is the symbol of summation over a finite set for each $m$, we have
	\begin{equation*}
		\lim_{\ell \to \infty} \sum_{Q \in \D} \lambda_{\ell,Q}  \int_{Q_0}  b_{\ell,Q} (x) \d x =  \sum_{Q \in \D} \lambda_{Q}   \int_{Q_0}  b_{Q}  (x) \d x = \int_{Q_0} g (x)\d x.
	\end{equation*}
	By  $\{f_\ell \}_{\ell\in \mathbb N} \subset \H_{\vec p \, ^\prime}^{t',r'} \cap L_{\operatorname{loc}} ^{\vec p \, ^\prime} $ converges locally to $f$  in the weak topology of $  L^{\vec p \, ^\prime}$,
	we obtain
	\begin{equation*}
		\int_{Q_0} f (x)\d x = \lim_{\ell\to \infty } \int_{Q_0} f_\ell (x)\d x = \lim_{\ell \to \infty} \sum_{Q \in \D} \lambda_{\ell,Q}  \int_{Q_0}  b_{\ell,Q} (x) \d x  = \int_{Q_0} g (x)\d x.
	\end{equation*}
	
	Finally, letting $\epsilon \to 0^+$ in  (\ref{f liminf 1+ eps}), the proof is complete.
\end{proof}


Now we use the Fatou property of $\H_{\vec p \, ^\prime}^{t',r'} $ to show the following result.
\begin{corollary}
		Let $ 1 < \vec p <\infty$. Let   $1 < n / ( \sum_{i=1}^n  1/p_{i})   < t <r<\infty $ or $ 1< n / ( \sum_{i=1}^n  1/p_{i})  \le t< r =\infty  $.  Suppose that  $ fg \in L^1 $ for all $f  \in  M_{\vec p}^{t,r} $. Then $g \in \H_{\vec p \, ^\prime}^{t',r'}  $.
\end{corollary}
\begin{proof}
	By the closed graph theorem, $ \|  fg \|_{L^1}   \le L \|f\|_{M_{\vec p}^{t,r}}  $  for some $L$ independent of $f  \in M_{\vec p}^{t,r}$. Set   $g_N := \chi_{B(0,N)}  \chi_{ [0,N] } (|g|) g $ for each $N \in \mathbb N$. Then using Lemma \ref{generate mix block} and the fact that $B(0,N)$ can be covered by at most $2^n$ dyadic cubes, we have $g_N \in \H_{\vec p \, ^\prime}^{t',r'} $. Using Theorem \ref{predual mix BM}, we obtain $ \| g_N \| _{ \H_{\vec p \, ^\prime}^{t',r'}  }  \le L. $
	Since $g_N$  converges locally to $g$  in the weak topology of $  L^{\vec p \, ^\prime} $,
	using Theorem \ref{Fatou general}, we see that $g \in  \H_{\vec p \, ^\prime}^{t',r'}  $.
\end{proof}

\subsection{Lattice property} \label{lattice space}
The following result is the  lattice property of block spaces $\H_{\vec p \, ^\prime}^{t',r'} $.
\begin{lemma} \label{lem lattice block}
	Let $ 1 < \vec p <\infty$. Let   $1 < n / ( \sum_{i=1}^n  1/p_{i})   < t <r<\infty $ or $ 1< n / ( \sum_{i=1}^n  1/p_{i})  \le t< r =\infty  $.
	Then a measurable function $ f$  belongs to $\H_{\vec p \, ^\prime}^{t',r'}  $ if and only if there exists a measurable non-negative function  $ g \in \H_{\vec p \, ^\prime}^{t',r'} $ such that  $|f(x)| \le g (x)$ a.e.
\end{lemma}

\begin{proof}
	The proof is similar to \cite[Lemma 4.10]{BGX25} and we omit it here.
\end{proof}

\subsection{Associate spaces and dual spaces} \label{sec associate and dual}
We first recall ball Banach function norm and ball Banach function space.
\begin{definition} \label{ball Bf norm}
	A mapping $\rho: \mathbb M ^+ \to [0,\infty]$  is called a ball Banach  function norm (over $\rn$) if, for all $f, g, f_k (k \in \mathbb N)$  in $\mathbb M ^+ $, for all constants $a \ge 0$ and for all balls $B$ in $\rn$, the	following properties hold:
	
	{\rm (P1)} $\rho (f) = 0 $ if and only if $f =0$ a.e., $\rho (a f) = a \rho (f)$, $\rho(f+g) \le \rho (f) +\rho (g)$,
	
		{\rm (P2)} $\rho (g) \le \rho (f)$  if $ g \le f $ a.e.,
		
		{\rm (P3)} the Fatou property; $\rho (f_j) \uparrow \rho (f)$ holds whenever $0 \le f_j \uparrow f$ a.e.,
		
		{\rm (P4)} For all balls $B$, $\rho (\chi_B) <\infty$,
		
			{\rm (P5)} $\| \chi_B f \|_{L^1} \lesssim_B \rho (f) $ with the implicit constant depending on $B$ and $\rho$ but independent of $f.$
			
		The space generated by such $\rho$ is called a ball Banach function
			space.
\end{definition}

\begin{definition}
	If $\rho$  is a a ball Banach function norm, its associate norm $\rho '$ is defined in  $ \mathbb M ^+$  by
	\begin{equation*}
		\rho ' (g) := \sup\left\{ \| fg\|_{L^1} :f \in \mathbb M ^+, \rho (f) \le 1 \right\}, g \in  \mathbb M ^+ .
	\end{equation*}
\end{definition}

\begin{definition}
	Let $\rho$ be a ball Banach function norm, and let $ \mathcal X = \mathcal X (\rho) $ be the ball Banach function space determined by $\rho$ as in Definition \ref{ball Bf norm}. Let $\rho '$ be the associate norm of $\rho$. The Banach function space $\mathcal X(\rho ') = \mathcal X ' (\rho)$ determined by $\rho '$ is called the associate space or the K\"othe dual of $\mathcal X$.
\end{definition}

\begin{lemma}[Theorem 352, \cite{SDH20}] \label{E= E''}
	Every ball Banach function space $E(\rn)$ coincides with its
	second associate space $E'' (\rn)$. In other words, a function $f$ belongs to $E(\rn)$
	if and only if it belongs to $E ''(\rn)$, and in that case $\|f\|_{E} = \|f\|_{E''}$ or all
	$f \in E (\rn) = E'' (\rn)$.
\end{lemma}


\begin{theorem} \label{associate space}
		Let $ 1 < \vec p <\infty$. Let   $1 < n / ( \sum_{i=1}^n  1/p_{i})   < t <r<\infty $ or $ 1< n / ( \sum_{i=1}^n  1/p_{i})  \le t< r =\infty  $.
	Then the associate space $( M_{\vec p}^{t,r} )  ' $ as a ball Banach function space coincides with the block space $\mathcal{H}_{\vec p \, ^\prime}^{t',r'}  $.
\end{theorem}

\begin{proof}
	We use the idea from \cite[Theorem 355]{SDH20}. 
	As we have seen in Theorem \ref{predual mix BM}, $\mathcal{H}_{\vec p \, ^\prime}^{t',r'} \subset  ( M_{\vec p}^{t,r} )  '  $. Hence, we will verify the converse. Suppose that $f\in L^0$ satisfies 
	\begin{equation} \label{fg le 1}
		\sup \left\{ \| fg\|_{L^1} : \|g\|_{ M_{\vec p}^{t,r} } \le  1 \right\}  \le 1.
	\end{equation}
Then we first see that $|f(x)|<\infty$, a.e. $x\in \rn$. Splitting $f$ into its real
and imaginary parts and each of these into its positive and negative parts,
we may assume without loss of generality that $f \in \mathbb  M^+.$ For $k \in \mathbb N$, set $f_k := \min(f,k) \chi_{B_k }$. By Theorem \ref{predual mix BM}, Lemma \ref{lem lattice block} and (\ref{fg le 1}), we have $f_k \in \mathcal{H}_{\vec p \, ^\prime}^{t',r'}$ and $ \| f_k \|_{\mathcal{H}_{\vec p \, ^\prime}^{t',r'}} \le 1$. Since $f_k \uparrow f$ a.e., by Theorem \ref{Fatou general}, we obtain $f \in \mathcal{H}_{\vec p \, ^\prime}^{t',r'}$ and $ \| f\|_{\mathcal{H}_{\vec p \, ^\prime}^{t',r'}} \le 1$. 

By Theorem \ref{Fatou general} and Lemma \ref{E= E''}, we have $ \mathcal{H}_{\vec p \, ^\prime}^{t',r'} = (\mathcal{H}_{\vec p \, ^\prime}^{t',r'})''$. This proves the theorem.
\end{proof}

Theorem \ref{associate space} and Lemma \ref{E= E''} yield $( \mathcal{H}_{\vec p \, ^\prime}^{t',r'} ) ' = ( M_{\vec p}^{t,r} )  '' = M_{\vec p}^{t,r}  $.

\begin{definition}[Absolutely continuous norm] 
	Let $ 0 < \vec p \le \infty $. Let $ 0 <t <\infty $.
	Then    $ M_{\vec p}^{t,r} \neq \{ 0\}$ if and only if $   0 < n / ( \sum_{i=1}^n  1/p_{i})    < t <r <\infty $ or $ 0< n / ( \sum_{i=1}^n  1/p_{i})    \le t < r=\infty  $.
	A function $f \in M_{\vec p}^{t,r} $ has absolutely continuous norm in $ M_{\vec p}^{t,r}$ if  $\| f\chi_{E_k}\| _{M_{\vec p}^{t,r}} \to 0  $ 
	for every sequence $\{ E_k \}_{k=1}^\infty $ satisfying $E_k \to \emptyset $ a.e., in the sense that
	$\lim_{k \to \infty} \chi_{E_k} =0$ for almost all $x\in \rn$.
\end{definition}

\begin{lemma} \label{lem Absolutely continuous norm}
		Let $1< \vec p < \infty $ and 
	$ 1 <  n / ( \sum_{i=1}^n  1/p_{i})    < t <r <\infty $.  For $f \in  M_{\vec p}^{t,r}$, $f$ has the
	absolutely continuous norm.
\end{lemma}
\begin{proof}
	Let $f \in M_{\vec p}^{t,r}$. By Lemma \ref{dense simple}, we can find a sequence of simple functions $\{f_k\}_{k\in \mathbb N} \subset \operatorname{Sim} (\rn)$ such that 
	\begin{equation*}
		\lim_{k\to \infty} \| f- f_k\|_{M_{\vec p}^{t,r}}  =0.
	\end{equation*}
Let $\{F_k \}_{k\in \mathbb N}$ be an arbitrary sequence for which $F_k \to \emptyset$ a.e. Then
\begin{equation*}
	\|f \chi_{F_j} \|_{ M_{\vec p}^{t,r}} \le \| f_k -f \|_{M_{\vec p}^{t,r}}  + \| f_k  \chi_{F_j} \|_{ M_{\vec p}^{t,r}} .
\end{equation*}
Letting $j \to \infty$, we obtain
\begin{equation*}
	\limsup_{j\to \infty} 	\|f \chi_{F_j} \|_{ M_{\vec p}^{t,r}} \le \| f_k -f \|_{M_{\vec p}^{t,r}} .
\end{equation*}
Since $k$ is arbitrary, we have $	\lim_{j\to \infty} 	\|f \chi_{F_j} \|_{ M_{\vec p}^{t,r}} =0$.
\end{proof}

\begin{theorem} \label{dual mixed BM}
		Let $1< \vec p < \infty $ and 
	$ 1 <  n / ( \sum_{i=1}^n  1/p_{i})    < t <r <\infty $.  
	Then the dual of $  M_{\vec p}^{t,r} $ is $ \H_{\vec p \, ^\prime}^{t',r'}  $ in the following sense:
	
	If $g \in \H_{\vec p \, ^\prime}^{t',r'} $, then $L_g : f \mapsto \int_\rn f(x) g(x) \d x $ is an element of $ (  M_{\vec p}^{t,r} ) ^\ast  $. Moreover, for any $L \in  (  M_{\vec p}^{t,r} ) ^\ast$, there exists $g \in\H_{\vec p \, ^\prime}^{t',r'}$ such that $L = L_g$.
\end{theorem}
\begin{proof}
	We use the idea from \cite[Theorem 356]{SDH20}.
	The first assertion is clear, since
 $ M_{\vec p}^{t,r}$ is the dual of  $ \H_{\vec p \, ^\prime}^{t',r'}  $ as in Theorem \ref{predual mix BM}. Hence, we will prove that $  (  M_{\vec p}^{t,r} ) ^\ast  \subset  \H_{\vec p \, ^\prime}^{t',r'}$. Thanks to Theorem \ref{associate space}, we need only to show $  (  M_{\vec p}^{t,r} ) ^\ast  \subset   ( M_{\vec p}^{t,r} )  '$. Let $L \in (  M_{\vec p}^{t,r} ) ^\ast$. We will exhibit a function $g  \in ( M_{\vec p}^{t,r} )  '$  such that 
 \begin{equation} \label{L f = int fg}
 	L (f) = \int_\rn f (x) g(x) \d x, \; f \in  M_{\vec p}^{t,r} .
 \end{equation}

For each $i \in  \mathbb Z^n$, let $\mathcal A_i$ denote the Lebesgue measurable subsets of $Q_{0,i}$ and define a set-function $\lambda_i$ on $\mathcal A_i$  by  $\lambda_i (A) := L (\chi_A)$, $A \in \mathcal A_i$.
Notice that $\lambda_i (A)$ is well defined for all $A \in \mathcal A_i$ because $\chi_A  \in L_c^\infty \subset M_{\vec p}^{t,r}$.

We claim that $\lambda_i$ is countably additive on $\mathcal A_i$ for any $i \in  \mathbb Z^n$. Indeed,
let $\{  A_k \}_{k \in \mathbb N}$ be a sequence of disjoint sets from $\mathcal A_i$, and let
\begin{equation*}
	B_\ell = \cup_{k=1}^\ell A_k, \; A := \cup_{\ell \in \mathbb N} B_\ell  = \cup_{k \in \mathbb N} A_k.
\end{equation*}
It follows from the Lebesgue dominated convergence theorem ($r<\infty$) that
\begin{equation*}
\lim_{\ell \to \infty}	\| \chi_A - \chi_{B_\ell} \|_{M_{\vec p}^{t,r} } = 0.
\end{equation*}
The continuity and linearity of $L$ give
\begin{equation*}
	\lambda_i (A) = L(\chi_A)  = \lim_{\ell \to \infty}  L(B_\ell) =\lim_{\ell \to \infty}  \sum_{k=1}^\ell   L(A_k) = \sum_{k=1}^\infty \lambda_i (A_k),
\end{equation*}
which establishes the claim. 

Since $| \lambda_i (A)|  \le \|L\|_{ (M_{\vec p}^{t,r})^\ast} \|\chi_A\|_{M_{\vec p}^{t,r} }$
for all $A \in \mathcal A_i$ and $ \lambda_i (A) = 0$ for
all $A \in \mathcal A_i$ such that $|A|=0$, by the Radon-Nikodym theorem (see \cite[Theorem 15.41]{Ko08}, for example), there uniquely exists $g\in L^0(Q_{0,i})$ such that
\begin{equation*}
	L(\chi_A)  = \lambda_i (A) =\int_\rn \chi_A (x) g_i (x) \d x, \; A \subset \mathcal A _i.
\end{equation*}
Since the sets $ \{Q_{0,i} \}_{i\in\mathbb Z^n} $ are disjoint, we may define a function  $g$ on $\rn $  by setting $g:= g_i $ on $Q_{0,i}$. Clearly, 
\begin{equation} \label{L chi E}
	L (\chi_E) = \int_\rn \chi_E (x) g(x) \d x
\end{equation}
for all characteristic functions of sets of finite measure $E$. We will show that $g$ belongs to $( M_{\vec p}^{t,r})  '$. Choose and fix $f$ in $ M_{\vec p}^{t,r} $. 
Let 
\begin{equation*}
	f_\ell := \sum_{k=1}^{4^\ell} 2^{-\ell} k \chi_{F_{k,\ell}}
\end{equation*}
for $\ell \in \mathbb N$, where $F_{k,\ell} :=\{ x\in B_{2^\ell} :  2^{-\ell} k \le |f(x)| \le 2^{-\ell} (k+1) \} $. We may suppose that $g$ is real-valued; otherwise split $g$ into its real
and imaginary parts. Then  $f_\ell \sgn (g)$ becomes a finite linear combination of characteristic functions of sets of finite measure. Hence, we may apply (\ref{L chi E}) and use the linearity of $L$ to obtain
\begin{equation*}
	\int_\rn f_\ell (x) | g(x) |\d x = L (f_\ell \sgn (g) )  \le \| L \|_{ ( M_{\vec p}^{t,r} )^\ast }  \| f_\ell \|_{  M_{\vec p}^{t,r}} .
\end{equation*}
Letting $\ell \to \infty$, we obtain 
\begin{equation*}
	\| f g\|_{L^1} \le \| L \|_{ ( M_{\vec p}^{t,r} )^\ast }  \| f \|_{  M_{\vec p}^{t,r}} 
\end{equation*}
from the monotone
convergence theorem and the Fatou property of mixed Bourgain-Morrey norm. This means $g \in ( M_{\vec p}^{t,r} ) '$.

For a simple function $f$, write
\begin{equation*}
	L(f) =\int_\rn f(x) g(x) \d x,
\end{equation*}
and observe the continuity of both sides on $M_{\vec p}^{t,r} $. By Lemma \ref{lem Absolutely continuous norm}, we obtain (\ref{L f = int fg}) holds.
This completes the proof of the theorem.
\end{proof} 

\begin{corollary}\label{reflexive BMX}
	Let $1< \vec p < \infty $ and 
	$ 1 <  n / ( \sum_{i=1}^n  1/p_{i})     < t <r <\infty $.     Then  $   M_{\vec p}^{t,r} $ is reflexive.
\end{corollary}
\begin{remark}
	Let $\vec p = p$, then Theorem \ref{dual mixed BM} becomes \cite[Theorem 5.5]{HNSH23} and Corollary \ref{reflexive BMX}  becomes \cite[Remark 5.7]{HNSH23}.
\end{remark}

\section{Operators} \label{sec operator}
In this section, we discuss some operators on mixed Bourgain-Morrey spaces and their preduals.
\subsection{The Hardy-Littlewood maximal operator on mixed Bourgain-Morrey spaces} 
In this subsection, we consider the boundedness property of the Hardy-Littlewood maximal operator and iterated maximal operator on mixed Bourgain-Morrey spaces.

For the boundedness of Hardy-Littlewood maximal operator,
we define a equivalent norm of $M_{\vec p}^{t,r}  $ below. To do so, we use a dyadic grid $\mathcal D_{k,\vec a}$, $k \in \mathbb Z$, $\vec a \in \{ 0,1,2\}^n$. More precisely, let
\begin{equation*}
	\mathcal D_{k,\vec a}^0  :=  \{ 2^{-k} [ m +a/3, m+a/3 +1) : m\in \mathbb Z   \}
\end{equation*}
for $k \in \mathbb Z$  and $a = 0,1,2$. Consider
\begin{equation*}
	\mathcal D_{k,\vec a} := \{ Q_1 \times Q_2 \times \cdots \times Q_n :Q_j \in	\mathcal D_{k,\vec a_j}^0 ,  j =1,2,\ldots , n \}
\end{equation*}
for $k \in \mathbb Z$  and  $\vec a = (a_1, a_2, \ldots, a_n) \in  \{ 0,1,2\}^n$. Hereafter,  a  dyadic grid is the family $\mathcal D _{\vec a} := \cup_{k\in \mathbb Z} \mathcal D_{k,\vec a} $ for $\vec a \in \{0,1,2\}^n$.

The following result is an important property of the dyadic grids to prove the equivalent norm of the mixed Bourgain-Morrey space.
\begin{lemma}[Lemma 2.8, \cite{HNSH23}] \label{cube be covered}
	For any cube $Q$ there exists $R\in \bigcup_{\vec a \in \{0,1,2\}^n}  \mathcal D _{\vec a}$ such that $Q \subset R$  and $|R| \le 6^n |Q|$.
\end{lemma}

Fix the dyadic grid $  \mathcal D _{\vec a}$ where $ \vec a \in \{0,1,2\}^n  $ and  define
\begin{equation*}
	\| f \|_{  M_{\vec p}^{t,r} (\mathcal D _{\vec a}) } : = \left\|\left\{|Q|^{1/t-1/p}  \| f \chi_Q \|_{L^{\vec p }  }  \right\}_{Q \in \mathcal D _{\vec a} }    \right\|_{\ell^r}
\end{equation*}
for all $f \in L^0 $.

Using Lemma \ref{cube be covered}, we obtain the following result. Since its proof is similar to  \cite[ (2.1)]{HNSH23},  we omit it here.
\begin{lemma} \label{D equivalence}
	Let $ 0 < \vec p \le \infty $. Let    $0 < n / ( \sum_{i=1}^n  1/p_i )   < t <r <\infty $ or $ 0< n / ( \sum_{i=1}^n  1/p_i )  \le t < r=\infty  $.
	Then $ \|\cdot \|_{M_{\vec p}^{t,r}   }$  and $ \|\cdot \|_{M_{\vec p}^{t,r} (\mathcal D _{\vec a})  } $ are equivalent for any dyadic grid $\mathcal D _{\vec a} $ where $ \vec a \in \{0,1,2\}^n  $. That is, for each $f \in L^0 $, $ \|f\|_{M_{\vec p}^{t,r}   }   \approx   \|f\|_{M_{\vec p}^{t,r} (\mathcal D _{\vec a})  }  $.
\end{lemma}

Denote by $\mathcal M_{\operatorname{dyadic}}$ the dyadic maximal operator generated by the dyadic cubes  in $\mathcal D$. We do not have a pointwise estimate to control $\mathcal M$ in terms of $\mathcal M_{\operatorname{dyadic}}$. 
The maximal operator generated by a family $\mathcal D _{\vec a}$ is defined by 
\begin{equation*}
	\mathcal M_{\mathcal D _{\vec a}} f (x) = \sup_{Q \in \mathcal D _{\vec a}} \frac{\chi_Q (x)}{|Q|} \int_Q |f(y) | \d y
\end{equation*}
for $f\in L^0 $ and $ \vec a \in \{0,1,2\}^n  $.
As in \cite{LN19},
\begin{equation} \label{M le Mdya}
	\mathcal M f (x) \lesssim \sum_{ \vec a \in \{0,1,2\}^n } \M_{\mathcal D _{\vec a}} f (x).
\end{equation}

\begin{theorem} \label{HL M}
	Let $ 1< \vec p < \infty $. Let    $1 < n / ( \sum_{i=1}^n  1/p_i )   < t <r <\infty $ or $ 1< n / ( \sum_{i=1}^n  1/p_i )  \le t < r=\infty  $.
 Then $\mathcal M$  is bounded on  $M_{\vec p}^{t,r} $.
\end{theorem}

\begin{proof}
	We only need to prove the case $r<\infty$, since $r =\infty$ is in \cite[Theorem 4.5]{N19}. Due to (\ref{M le Mdya}) and Lemma \ref{D equivalence}, it sufficient to show the boundedness for the dyadic maximal operator $ 	\mathcal M_{\mathcal D _{\vec a}}$  instead of the Hardy-Littlewood maximal function $\mathcal M$.	
	Let $f \in M_{\vec p}^{t,r}  $, $Q \in \mathcal D _{\vec a}$, $f_1 = f \chi_Q$ and $f_2 = f \chi_{\rn \backslash Q}$. By \cite[Theorem 4.5]{N19}, we have
	\begin{equation} \label{M f_1}
		\| 	\mathcal M_{\mathcal D _{\vec a}} f_1 \|_{ L^{\vec p} (	Q)} \lesssim 	\| f_1 \|_{ L^{\vec p} } = \| f \|_{ L^{\vec p} (Q) }.
	\end{equation}
	For $k \in \mathbb N$, let $Q_k$  be the $k^{\operatorname{th}}$ dyadic parent of $Q$, which is the dyadic cube in $\mathcal D _{\vec a}$ satisfying $Q \subset Q_k$ and $\ell (Q_k) = 2^k \ell (Q)$. 
	Hence, for $x \in Q$, applying the H\"older inequality (Lemma \ref{Holder mixed}), we obtain
	\begin{align*}
		\mathcal M_{\mathcal D _{\vec a}} f_2 (x) & \le \sum_{k=1} \frac{1}{|Q_k|} \int_{Q_k} |f(y) | \d y \\
		& \le \sum_{k=1} |Q_k|^{ - \frac{1}{n}  \sum_{i=1}^n   \frac{1}{p_i} } \| f \|_{ L^{\vec p} (Q_k)}  .
	\end{align*}
	Thus,
	\begin{align*}
		\|	\mathcal M_{\mathcal D _{\vec a}} f_2 \|_{ L^{\vec p} (Q) } \le |Q|^{ \frac{1}{n}  \sum_{i=1}^n   \frac{1}{p_i}} \sum_{k=1} |Q_k|^{ - \frac{1}{n}  \sum_{i=1}^n   \frac{1}{p_i} } \| f \|_{ L^{\vec p} (Q_k)} .
	\end{align*}
	Let $R \in \mathcal D _{\vec a}$ and $k\in \mathbb N$. A geometric observation shows there are $2^{kn}$ dyadic cubes $Q$ such that $Q_k = R$. Namely,
	\begin{equation} \label{geo 2kn}
		\sum_{  Q  \in \mathcal D _{\vec a}, Q_k = R}  |Q_k|^{r/t -\frac{r}{n}  \sum_{i=1}^n   \frac{1}{p_i} } \|f\|_{  L^{\vec p} (Q_k)} ^r  = 2^{kn} |R|^{r/t -\frac{r}{n}  \sum_{i=1}^n   \frac{1}{p_i}} \|f\|_{  L^{\vec p} (R)} ^r .
	\end{equation}
	Adding (\ref{geo 2kn}) over  $R \in \mathcal D _{\vec a}$ gives
	\begin{equation*}
		\sum_{Q \in \mathcal D _{\vec a}} |Q_k| ^{r/t -\frac{r}{n}  \sum_{i=1}^n   \frac{1}{p_i} } \|f\|_{  L^{\vec p } (Q_k)} ^r  =2^{kn} \sum_{R \in \mathcal D _{\vec a} } |R|^{r/t-\frac{r}{n}  \sum_{i=1}^n   \frac{1}{p_i} } \|f\|_{  L^{\vec p } (R)} ^r  = 2^{kn}  \|f\|_{M_{\vec p}^{t,r}  } ^r .
	\end{equation*}
	Then 
	\begin{align*}
		& \left(\sum_{Q\in \mathcal D _{\vec a}}  |Q|^{r/t-\frac{r}{n}  \sum_{i=1}^n   \frac{1}{p_i} } \|	\mathcal M_{\mathcal D _{\vec a}} f_2 \|_{ L^{\vec p} (Q) } ^r  \right)^{1/r} \\
		& \le \left(\sum_{Q\in \mathcal D _{\vec a}}  |Q|^{r/t} \left(  \sum_{k=1} |Q_k|^{ - \frac{1}{n}  \sum_{i=1}^n   \frac{1}{p_i} } \| f \|_{ L^{\vec p} (Q_k)} \right)^r  \right)^{1/r} \\
		& \le \sum_{k=1}^\infty 2^{-kn/t}  \left(\sum_{Q\in \mathcal D _{\vec a}}  |Q_k|^{r/t -   \frac{1}{n}  \sum_{i=1}^n   \frac{1}{p_i}}   \| f \|_{ L^{\vec p} (Q_k)}^r  \right)^{1/r} \\
		& \le \sum_{k=1}^\infty 2^{-kn/t} 2^{kn/r}  \left(\sum_{Q_k\in \mathcal D _{\vec a}}  |Q_k|^{r/t  - \frac{1}{n}  \sum_{i=1}^n   \frac{1}{p_i} }   \| f \|_{ L^{\vec p} (Q_k)}^r  \right)^{1/r} \\
		& \lesssim \| f\|_{M_{\vec p}^{t,r} (\mathcal D _{\vec a})  } .
	\end{align*}
	Together this with (\ref{M f_1}), we prove 
	\begin{equation*}
		\| \mathcal M_{\mathcal D _{\vec a}} f\|_{M_{\vec p}^{t,r} (\mathcal D _{\vec a})   } 	\lesssim \| f\|_{M_{\vec p}^{t,r} (\mathcal D _{\vec a}) } 
	\end{equation*}
	as desired.
\end{proof}

We  define the maximal operator$ \M_{(k)}$ for  $x_k$  as follows:
\begin{equation*}
	\M_{(k)} f (x) := \sup_{x_k \in I}  \int _I  | f(x_1, \ldots, y_k ,\ldots, x_n) | \d y_k,
\end{equation*}
where interval $I$ ranges over all intervals containing $x_k$. Furthermore, for all measurable functions
$f$, define the iterated maximal operator  $\M^{(\operatorname{it})}$ by
\begin{equation*}
	\M^{(\operatorname{it})} f(x) := \M_{(n)} \cdots \M_{(1)} (|f|) (x) .
\end{equation*}

\begin{proposition} \label{mixed BM equ norm weight}
		Let $ 1< \vec p < \infty $. Let    $ 1 < n / ( \sum_{i=1}^n  1/p_i )   < t <r <\infty $ or $ 1 < n / ( \sum_{i=1}^n  1/p_i )  \le t < r=\infty  $. Let $\eta \in (0,1)  $  such that
		\begin{equation*}
				0<   \left(\frac{1}{n} \sum_{i=1 } ^n \frac{1}{p_i} -\frac{1}{t}  +  \frac{1}{r} \right)  <\eta <1 .
		\end{equation*}
		Then for $f\in L^0(\rn)$, we have
		\begin{equation*}
			\|f\|_{M_{\vec p}^{t,r}  }  \approx  \left(\sum_{Q \in \D} |Q|^{ \frac{r}{t}  - \frac{r}{n} \sum_{i=1 } ^n \frac{1}{p_i}  }  \| f ( \M^{\operatorname{it} } \chi_Q) ^\eta \|_{L^{\vec p}  } ^r  \right)^{1/r} .
		\end{equation*}
\end{proposition}
\begin{proof}
	We use the idea from \cite[Proposition 4.10]{N19}.
	We only need to show the case $r<\infty$  since  case $r=\infty$ is similar. 
	From $ \chi_Q \le  ( \M^{\operatorname{it} } \chi_Q) ^\eta$, we obtain
	\begin{equation*}
		\|f\|_{M_{\vec p}^{t,r}  }  \le \left(\sum_{Q \in \D} |Q|^{ \frac{r}{t}  - \frac{r}{n} \sum_{i=1 } ^n \frac{1}{p_i}  }  \| f ( \M^{\operatorname{it} } \chi_Q) ^\eta \|_{L^{\vec p}  } ^r  \right)^{1/r} .
	\end{equation*}
For the opposite inequality, fix a cube $Q = I_1 \times \cdots \times I_n$. Given  $(\ell_1, \ldots, \ell_n) \in \mathbb N^n$, we write $\ell = \max (\ell_1, \ldots, \ell_n)$. Then we get
\begin{align*}
 &	|Q|^{ \frac{1}{t}  - \frac{1}{n} \sum_{i=1 } ^n \frac{1}{p_i}  }  \| f ( \M^{\operatorname{it} } \chi_Q) ^\eta \|_{L^{\vec p}  } \\
	& \lesssim |Q|^{ \frac{1}{t}  - \frac{1}{n} \sum_{i=1 } ^n \frac{1}{p_i}  }  \left\| f  \prod_{j=1}^n  \left( \frac{ \ell (I_j)}{\ell(I_j) + |  \cdot_j - c(I_j) | |}\right) ^\eta  \right\|_{L^{\vec p}  } \\
		& \lesssim |Q|^{ \frac{1}{t}  - \frac{1}{n} \sum_{i=1 } ^n \frac{1}{p_i}  } 
		\sum_{\ell_1, \ldots, \ell_n =1}^\infty  \frac{1}{2^{ (\ell_1 + \ldots, \ell_n )\eta  } }
		 \left\| f  \chi_{ 2^{\ell_1} I_1 \times \cdots \times  2^{\ell_n} I_n }    \right\|_{L^{\vec p}  } \\
		 	& \lesssim |Q|^{ \frac{1}{t}  - \frac{1}{n} \sum_{i=1 } ^n \frac{1}{p_i}  } 
		 \sum_{\ell_1, \ldots, \ell_n =1}^\infty  \frac{1}{2^{ (\ell_1 + \ldots, \ell_n )\eta  } }
		 \left\| f  \chi_{ 2^{\ell} Q }    \right\|_{L^{\vec p}  } \\
		 & \lesssim 
		 \sum_{\ell_1, \ldots, \ell_n =1}^\infty  \frac{  2^{\ell n (\frac{1}{n} \sum_{i=1 } ^n \frac{1}{p_i} -\frac{1}{t}  ) }   }{2^{ (\ell_1 + \ldots, \ell_n )\eta  } } 
		  |2^\ell Q|^{ \frac{1}{t}  - \frac{1}{n} \sum_{i=1 } ^n \frac{1}{p_i}  } 
		 \left\| f  \chi_{ 2^{\ell} Q }    \right\|_{L^{\vec p}  } ,
\end{align*}
where $c(I_j)$ denote the center of $I_j$.
Adding over all cubes $Q\in \D$, we obtain
\begin{align*}
	& \left(\sum_{Q \in \D} |Q|^{ \frac{r}{t}  - \frac{r}{n} \sum_{i=1 } ^n \frac{1}{p_i}  }  \| f ( \M^{\operatorname{it} } \chi_Q) ^\eta \|_{L^{\vec p}  } ^r  \right)^{1/r}  \\
	& \lesssim \sum_{\ell_1, \ldots, \ell_n =1}^\infty  \frac{  2^{\ell n (\frac{1}{n} \sum_{i=1 } ^n \frac{1}{p_i} -\frac{1}{t}  ) }   }{2^{ (\ell_1 + \ldots, \ell_n )\eta  } }  \left(\sum_{Q \in \D}  |2^\ell Q|^{ \frac{r}{t}  - \frac{r}{n} \sum_{i=1 } ^n \frac{1}{p_i}  } 
	\left\| f  \chi_{ 2^{\ell} Q }    \right\|_{L^{\vec p}  } ^r    \right)^{1/r}
	\\
	& \lesssim \sum_{\ell_1, \ldots, \ell_n =1}^\infty  \frac{  2^{\ell n (\frac{1}{n} \sum_{i=1 } ^n \frac{1}{p_i} -\frac{1}{t}  ) }   }{2^{ (\ell_1 + \ldots, \ell_n )\eta  } }  2^{\ell n/r}  \| f\|_{ M_{\vec p}^{t,r} } \\
	& \lesssim  \| f\|_{ M_{\vec p}^{t,r} } .
\end{align*}
Next we show that $\eta $ can be chosen.
Since  $ 0 <1/r  <  ( \sum_{i=1 } ^n \frac{1}{p_i}  )/n <1 $, we get 
\begin{equation*}
	0< 1/r <  \left (\frac{1}{n} \sum_{i=1 } ^n \frac{1}{p_i} -\frac{1}{t}  +  \frac{1}{r} \right )  < 1- 1/t+1/r <1.
\end{equation*}
Thus the proof is finished.
\end{proof}
\begin{remark}
	We give some explanations about the sum $\sum_{\ell_1, \ldots, \ell_n =1}^\infty  \frac{  2^{\ell n a }   }{2^{ (\ell_1 + \ldots, \ell_n )b  } } $.
	If $ \ell_1 =2 , \ell_2 =\cdots = \ell_n =1  $, then $\ell = 2$ and 
	\begin{equation*}
		\frac{  2^{\ell n a }   }{2^{ (\ell_1 + \ldots, \ell_n )b  } } = 	\frac{  2^{\ell n a }   }{2^{ (n \ell + 1-n)b  } } \le 2^{nb} 	\frac{  2^{\ell n a }   }{2^{ n \ell b  } } .
	\end{equation*}
Hence if $ b-a >0 $,
\begin{equation*}
	\sum_{\ell_1, \ldots, \ell_n =1}^\infty  \frac{  2^{\ell n a }   }{2^{ (\ell_1 + \ldots, \ell_n )b  } }  <\infty.
\end{equation*}
\end{remark}

\begin{lemma}[Proposition 4.9, \cite{N19}] \label{Mit mixed weight}
	Let $1<\vec p <\infty$. Let $f \in L^0$ and $ \omega_k \in A_{p_k} (\mathbb R) $  for $k =1 ,\ldots ,n.$ Then
	\begin{equation*}
		\left\| \M^{\operatorname{it}} f \cdot \bigotimes_{k=1}^n \omega_k ^{1/p_k} \right\|_{L^{\vec p}} \lesssim \left\| f \cdot \bigotimes_{k=1}^n \omega_k ^{1/p_k} \right\|_{L^{\vec p}} .
	\end{equation*}
\end{lemma}

\begin{theorem}
		Let $ 1< \vec p < \infty $. Let    $ 1 < n / ( \sum_{i=1}^n  1/p_i )   < t <r <\infty $ or $ 1 < n / ( \sum_{i=1}^n  1/p_i )  \le t < r=\infty  $. 
		Furthermore, let 		
		\begin{equation*}
		\left(\frac{1}{n} \sum_{i=1 } ^n \frac{1}{p_i} -\frac{1}{t}  +  \frac{1}{r} \right) <  \frac{1}{\max \vec p}.
		\end{equation*}
		 Set $\eta \in (0,1)  $  such that
	\begin{equation*}
		0<  \left (\frac{1}{n} \sum_{i=1 } ^n \frac{1}{p_i} -\frac{1}{t}  +  \frac{1}{r} \right)  <\eta  < \frac{1}{\max \vec p}<1 .
	\end{equation*}
Then 
\begin{equation*}
		\|  \Mit f \|_{ M_{\vec p}^{t,r}   } \lesssim \| f\|_{ M_{\vec p}^{t,r} } .
\end{equation*}
\end{theorem}
\begin{proof}
		Let $ 1< \vec p <\infty $. Let    $ 1 < n / ( \sum_{i=1}^n  1/p_i )   < t <r <\infty $ or $ 1 < n / ( \sum_{i=1}^n  1/p_i )  \le t < r=\infty  $. Let $\eta \in (0,1)  $  such that
	\begin{equation*}
		0<   \left(\frac{1}{n} \sum_{i=1 } ^n \frac{1}{p_i} -\frac{1}{t}  +  \frac{1}{r} \right )  <\eta <1 .
	\end{equation*}
Let $Q = I_1 \times \cdots \times I_n$. Then
\begin{equation*}
	( \Mit \chi_Q) ^{\eta }  = \left(  \bigotimes_{i=1}^n \M _{(i)} \chi_{I_i} \right)^{\eta }  =  \bigotimes_{i=1}^n ( \M _{(i)} \chi_{I_i} )^{\eta } .
\end{equation*}
From \cite[Theorem 288]{SDH20}, $ ( \M _{(i)} \chi_{I_i} )^{\eta  p_i} $  belongs to $A_1(\mathbb R)$ if  $\eta > 0$ such that
\begin{equation*}
   0 < \eta  p_i < 1 .
\end{equation*}
Hence $ ( \M _{(i)} \chi_{I_i} )^{\eta  p_i} \in A_1(\mathbb R) \subset A_{p_i}(\mathbb R) $ for all $p_i$. By Lemma \ref{Mit mixed weight}, we obtain
\begin{align*}
	\| \Mit f ( \Mit \chi_Q) ^{\eta } \|_{L^{ \vec p } } & =  \left\| ( \Mit f ) \bigotimes_{i=1}^n ( \M _{(i)} \chi_{I_i} )^{\eta }    \right\|_{ L^{ \vec p } } \\
	& \lesssim  \left\|  f \bigotimes_{i=1}^n ( \M _{(i)} \chi_{I_i} )^{\eta }    \right\|_{ L^{ \vec p } } \\
	& = \left\|  f   ( \Mit  \chi_Q )^{\eta }    \right\|_{ L^{ \vec p } } .
\end{align*}
Then by Proposition \ref{mixed BM equ norm weight},
we have
\begin{align*}
	\|  \Mit f \|_{ M_{\vec p}^{t,r}   } & \approx \left(\sum_{Q \in \D} |Q|^{ \frac{r}{t}  - \frac{r}{n} \sum_{i=1 } ^n \frac{1}{p_i}  }  \| \Mit f ( \M^{\operatorname{it} } \chi_Q) ^\eta \|_{L^{\vec p}  } ^r  \right)^{1/r}  \\
	& \lesssim \left(\sum_{Q \in \D} |Q|^{ \frac{r}{t}  - \frac{r}{n} \sum_{i=1 } ^n \frac{1}{p_i}  }  \left\|  f   ( \Mit  \chi_Q )^{\eta }    \right\|_{ L^{ \vec p } }  ^r  \right)^{1/r}  \\
	& \approx \| f\|_{ M_{\vec p}^{t,r}   } .
\end{align*}
We finish the proof.
\end{proof}

The vector-valued case is important because it is useful in establishing the theory of related function spaces. 
\begin{definition}
		Let $ 0 < \vec p \le \infty $. Let $   0 < n / ( \sum_{i=1}^n  1/p_{i})     < t <r <\infty $ or $ 0< n / ( \sum_{i=1}^n  1/p_{i})    \le t < r=\infty  $. Let $ 0< u \le \infty$.
	The mixed vector-valued  Bourgain-Morrey norm $\| \cdot \|_{ M_{\vec p} ^{t,r}  (\ell^u) } $  is defined by 
	\begin{align*}
		\| \{f_k \}_{k\in \mathbb N} \|_{  M_{\vec p} ^{t,r}  (\ell^u) } 
		& :=  \left(\sum_{Q \in \D} |Q|^{ \frac{r}{t}  - \frac{r}{n} \sum_{i=1 } ^n \frac{1}{p_i}  }  \left\| \left( \sum_{k=1}^\infty  |f_k|^u \right)^{1/u}  \chi_Q \right\|_{L^{\vec p}  } ^r  \right)^{1/r} ,
	\end{align*}
	where the sequence $\{f_k \}_{k=1}^\infty \subset L^0 $. The mixed vector-valued   Bourgain-Morrey space $  M_{\vec p} ^{t,r}  (\ell^u)  $ is the set of all measurable sequences  $\{f_k \}_{k\in \mathbb N}$ with finite norm $ \|\{f_k \}_{k\in \mathbb N}\|_{  M_{\vec p} ^{t,r}  (\ell^u)  }$.
	
\end{definition}

\begin{theorem} \label{HL seq mix BM}
	Let $ 1< \vec p \le \infty $ and $1<u \le \infty $. Let    $1 < n / ( \sum_{i=1}^n  1/p_i )   < t <r <\infty $ or $ 1< n / ( \sum_{i=1}^n  1/p_i )  \le t < r=\infty  $.
	Then for all $\{f_k \}_{k\in \mathbb N} \in M_{\vec p} ^{t,r}  (\ell^u)  $, we have
	\begin{equation*}
\|\{\M f_k \}_{k\in \mathbb N}\|_{  M_{\vec p} ^{t,r}  (\ell^u) }  \lesssim 	\|\{f_k \}_{k\in \mathbb N}\|_{  M_{\vec p} ^{t,r}  (\ell^u) } .
	\end{equation*}
\end{theorem}
\begin{proof}
	When $r=\infty$, Theorem \ref{HL seq mix BM} is contained in \cite[Theorem 1.8]{N19}.	
	We can identify $\M $ with the maximal operator $\mathcal M_{\mathcal D _{\vec a}} $  as in Theorem \ref{HL M}.
	
	Case $u=\infty$. Simply using Theorem \ref{HL M} and 
	\begin{equation*}
		\sup_{k \in \mathbb N} \M f_k  \le  \M \sup_{k \in \mathbb N} f_k, 
	\end{equation*}
we get the result.

Case $1<u <\infty $. Let $Q \in \D$. Denote by $Q_m$ the $m^{\operatorname{th}}$ dyadic parent. Observe that
\begin{equation*}
	\left(  \sum_{k=1}^\infty (\M f_k ) ^u  \right)^{1/u} \le \left(  \sum_{k=1}^\infty (\M [ \chi_Q  f_k ] ) ^u  \right)^{1/u} + \sum_{m=1}^\infty \left(  \sum_{k=1}^\infty \left( \frac{1}{|Q_m|} \int_{Q_m} |f_k (y) |\d y   \right) ^u  \right)^{1/u} .
\end{equation*}
Consequently,
\begin{align*}
	 & |Q|^{ \frac{1}{t}  - \frac{1}{n} \sum_{i=1 } ^n \frac{1}{p_i}  }  \left\| \left( \sum_{k=1}^\infty  |\M f_k|^u \right)^{1/u}  \chi_Q \right\|_{L^{\vec p}  } \\
	& \le |Q|^{ \frac{1}{t}  - \frac{1}{n} \sum_{i=1 } ^n \frac{1}{p_i}  }  \left\| \left(  \sum_{k=1}^\infty (\M [ \chi_Q  f_k ] ) ^u  \right)^{1/u}  \chi_Q \right\|_{L^{\vec p}  } \\
	& \quad + |Q|^{1/t}  \sum_{m=1}^\infty \left(  \sum_{k=1}^\infty \left( \frac{1}{|Q_m|} \int_{Q_m} |f_k (y) |\d y   \right) ^u  \right)^{1/u} =: I+ II.
\end{align*}
For the first part, by the Fefferman-Stein vector-valued maximal inequality for mixed spaces (for example, see \cite[Theorem 1.7]{N19}), we obtain
\begin{equation*}
	I \lesssim |Q|^{ \frac{1}{t}  - \frac{1}{n} \sum_{i=1 } ^n \frac{1}{p_i}  }  \left\| \left(  \sum_{k=1}^\infty |  f_k | ^u  \right)^{1/u}  \chi_Q \right\|_{L^{\vec p}  }.
\end{equation*}
By the Minkowski inequality, the H\"older inequality and $|Q| = 2^{-mn} |Q_m| $, we obtain
\begin{align*}
	II & \le  |Q|^{1/t}  \sum_{m=1}^\infty 
	 \frac{1}{|Q_m|} \int_{Q_m} 	\left(  \sum_{k=1}^\infty |f_k (y) | ^u  \right) ^{1/u} \d y  \\
	 & \le  |Q|^{1/t}  \sum_{m=1}^\infty 
	 \frac{1}{|Q_m|}   |Q_m|^{\frac{1}{n}  \sum_{i=1}^n \frac{1}{p_i '} }  \left\|\left(  \sum_{k=1}^\infty |f_k (y) | ^u  \right) ^{1/u} \chi_{Q_m}  \right\|_{L^{\vec p } }    \\
	 & =    \sum_{m=1}^\infty  2^{-mn/t}  |Q_m|^{1/t - \frac{1}{n}  \sum_{i=1}^n \frac{1}{p_i } }   \left\|\left(  \sum_{k=1}^\infty |f_k (y) | ^u  \right) ^{1/u} \chi_{Q_m}  \right\|_{L^{\vec p } }  .
\end{align*}
Writing $Q_0 = Q$, we have
\begin{align*}
	 & |Q|^{ \frac{1}{t}  - \frac{1}{n} \sum_{i=1 } ^n \frac{1}{p_i}  }  \left\| \left( \sum_{k=1}^\infty  |\M f_k|^u \right)^{1/u}  \chi_Q \right\|_{L^{\vec p}  }  \\
	& \lesssim \sum_{m=0 }^\infty  2^{-mn/t}  |Q_m|^{1/t - \frac{1}{n}  \sum_{i=1}^n \frac{1}{p_i } }   \left\|\left(  \sum_{k=1}^\infty |f_k (y) | ^u  \right) ^{1/u} \chi_{Q_m}  \right\|_{L^{\vec p } }  .
\end{align*}
Adding this estimate over $Q\in \D $, then
\begin{align*}
	 & \left(\sum_{Q \in \D} |Q|^{ \frac{r}{t}  - \frac{r}{n} \sum_{i=1 } ^n \frac{1}{p_i}  }  \left\| \left( \sum_{k=1}^\infty  |\M f_k|^u \right)^{1/u}  \chi_Q \right\|_{L^{\vec p}  } ^r  \right)^{1/r}   \\
	 & \lesssim \sum_{m=0 }^\infty  2^{-mn/t}  \left(\sum_{Q \in \D} |Q_m|^{ \frac{r}{t}  - \frac{r}{n} \sum_{i=1 } ^n \frac{1}{p_i}  }  \left\| \left( \sum_{k=1}^\infty  | f_k|^u \right)^{1/u}  \chi_{Q_m} \right\|_{L^{\vec p}  } ^r  \right)^{1/r}.
\end{align*}
Using the fact that there exist $2^{mn}$ children of the cube $Q_m$, namely, there exists $2^{mn}$ dyadic cubes in $\D_{ -\log_2  \ell (Q)  +m } $ that contains $Q$, we see that
\begin{align*}
	 & \left(\sum_{Q \in \D} |Q|^{ \frac{r}{t}  - \frac{r}{n} \sum_{i=1 } ^n \frac{1}{p_i}  }  \left\| \left( \sum_{k=1}^\infty  |\M f_k|^u \right)^{1/u}  \chi_Q \right\|_{L^{\vec p}  } ^r  \right)^{1/r}   \\
	 & \lesssim \sum_{m=0 }^\infty  2^{-mn/t}  \left( 2^{mn} \sum_{Q \in \D} |Q|^{ \frac{r}{t}  - \frac{r}{n} \sum_{i=1 } ^n \frac{1}{p_i}  }  \left\| \left( \sum_{k=1}^\infty  | f_k|^u \right)^{1/u}  \chi_{Q} \right\|_{L^{\vec p}  } ^r  \right)^{1/r} \\
	 & \le \sum_{m=0 }^\infty  2^{-mn/t} 2^{mn /r} 	\|\{f_k \}_{k\in \mathbb N}\|_{  M_{\vec p} ^{t,r}  (\ell^u) }  \\
	 & \lesssim 	\|\{f_k \}_{k\in \mathbb N}\|_{  M_{\vec p} ^{t,r}  (\ell^u) }.
\end{align*}
Thus the proof is finished.
\end{proof}

Moreover, to consider the application of the wavelet characterization we need the following vector-valued inequality.

\begin{theorem} \label{mixed vector HL mixed BM}
		Let $ 1< \vec p \le \infty $ and $1<u_1, u_2 < \infty $. Let    $1 < n / ( \sum_{i=1}^n  1/p_i )   < t <r <\infty $ or $ 1< n / ( \sum_{i=1}^n  1/p_i )  \le t < r=\infty  $.
	Then for all $\{ f_{k_1, k_2}  \}_{k_1, k_2 \in \mathbb N} \subset L^0 $, we have
	\begin{equation*}
	  \left\| \left( \sum_{k_2 =1}^\infty \left( \sum_{k_1 =1}^\infty  |\M f_{k_1, k_2}|^{u_1} \right) ^{u_2 /u_1 } \right)^{1/u_2}  \right\|_{ M_{\vec p}^{t,r} }   \lesssim  \left\| \left( \sum_{k_2 =1}^\infty \left( \sum_{k_1 =1}^\infty  |f_{k_1, k_2}|^{u_1} \right) ^{u_2 /u_1 } \right)^{1/u_2}  \right\|_{ M_{\vec p}^{t,r} }  .
	\end{equation*}
\end{theorem}
\begin{proof}
	The case $1< n / ( \sum_{i=1}^n  1/p_i )  \le t < r=\infty $ is proved in \cite[Proposition 2.15]{N24}.
	
		We can identify $\M $ with the maximal operator $\mathcal M_{\mathcal D _{\vec a}} $  as in Theorem \ref{HL M} where $ \vec a \in \{0,1,2\}^n  $.		
		 Let $Q \in \D_{\vec a}$. Denote by $Q_m$ the $m^{\operatorname{th}}$ dyadic parent of $Q$. By the sublinear of $\M$,
		\begin{align*}
			\left( \sum_{k_2 =1}^\infty \left( \sum_{k_1 =1}^\infty  |\M f_{k_1, k_2}|^{u_1} \right) ^{u_2 /u_1 } \right)^{1/u_2} 
		& \le \left( \sum_{k_2 =1}^\infty \left( \sum_{k_1 =1}^\infty  |\M ( f_{k_1, k_2} \chi_Q )|^{u_1} \right) ^{u_2 /u_1 } \right)^{1/u_2} \\
		& + \left( \sum_{k_2 =1}^\infty \left( \sum_{k_1 =1}^\infty  |\M ( f_{k_1, k_2} \chi_{Q^c} ) |^{u_1} \right) ^{u_2 /u_1 } \right)^{1/u_2} \\
		& =: I_1 +I_2.
		\end{align*}
	For $I_1$, using \cite[A.2]{N24}, we have
	\begin{equation*}
		\| I_1 \|_{L^{\vec p} } \lesssim 	\left\|  \left( \sum_{k_2 =1}^\infty \left( \sum_{k_1 =1}^\infty  | f_{k_1, k_2}  |^{u_1} \right) ^{u_2 /u_1 } \right)^{1/u_2} \chi_Q \right\|_{L^{\vec p} } .
	\end{equation*}
For $I_2$, for $x\in Q$, note that
\begin{align*}
	\M ( f_{k_1, k_2} \chi_{Q^c} ) (x) \le \sum_{m=1}^\infty \frac{1}{|Q_m|} \int_{Q_m}  |f_{k_1, k_2} (y)| \d y.
\end{align*}
Using Minkowski's inequality twice, we get
\begin{align*}
	\| 	\M ( f_{k_1, k_2} \chi_{Q^c} ) (x) \|_{\ell^{(u_1,u_2) } } & \le  \left\| 	\sum_{m=1}^\infty \frac{1}{|Q_m|} \int_{Q_m}  |f_{k_1, k_2} (y)| \d y  \right\|_{\ell^{(u_1,u_2) } } \\
	& \le \sum_{m=1}^\infty \frac{1}{|Q_m|} \int_{Q_m}  \left\| 	f_{k_1, k_2} (y)  \right\|_{\ell^{(u_1,u_2) } } \d y .
\end{align*}
Here and what follows, $ \left\| 	f_{k_1, k_2} (y)  \right\|_{\ell^{(u_1,u_2) } } : =  \left( \sum_{k_2 =1}^\infty \left( \sum_{k_1 =1}^\infty  |f_{k_1, k_2}|^{u_1} \right) ^{u_2 /u_1 } \right)^{1/u_2} $.
Thus by H\"older's inequality (Lemma \ref{Holder mixed}), and $|Q| = 2^{-mn} |Q_m| $, we obtain
\begin{align*}
			& |Q|^{ \frac{1}{t}  - \frac{1}{n} \sum_{i=1 } ^n \frac{1}{p_i}  }  \left\| \left( \sum_{k_2 =1}^\infty \left( \sum_{k_1 =1}^\infty  |\M ( f_{k_1, k_2} \chi_{Q^c} ) |^{u_1} \right) ^{u_2 /u_1 } \right)^{1/u_2} \chi_Q \right\|_{L^{\vec p}  }  \\
			& \le  |Q|^{ \frac{1}{t}} \sum_{m=1}^\infty \frac{1}{|Q_m|} \int_{Q_m}  \left\| 	f_{k_1, k_2} (y)  \right\|_{\ell^{(u_1,u_2) } } \d y \\
			& \le  \sum_{m=1} ^{\infty}{2^{-mn/t} } |Q_m|^{1/t - \frac{1}{n}  \sum_{i=1}^n \frac{1}{p_i } } \|  \left\| 	f_{k_1, k_2} (y)  \right\|_{\ell^{(u_1,u_2) } } \chi_{Q_m}\|_{L^{\vec p}} .
\end{align*}
		Writing $Q_0 = Q$, we have
		\begin{align*}
			& |Q|^{ \frac{1}{t}  - \frac{1}{n} \sum_{i=1 } ^n \frac{1}{p_i}  }  \left\| 	\left( \sum_{k_2 =1}^\infty \left( \sum_{k_1 =1}^\infty  |\M f_{k_1, k_2}|^{u_1} \right) ^{u_2 /u_1 } \right)^{1/u_2}   \chi_Q \right\|_{L^{\vec p}  }  \\
			& \lesssim \sum_{m=0 }^\infty  {2^{-mn/t} } |Q_m|^{1/t - \frac{1}{n}  \sum_{i=1}^n \frac{1}{p_i } } \|  \left\| 	f_{k_1, k_2} (y)  \right\|_{\ell^{(u_1,u_2) } } \chi_{Q_m}\|_{L^{\vec p}} .
		\end{align*}
		Adding this estimate over $Q\in \D_{\vec a} $, then
		\begin{align*}
			& \left(\sum_{Q \in \D_{\vec a}} |Q|^{ \frac{r}{t}  - \frac{r}{n} \sum_{i=1 } ^n \frac{1}{p_i}  }  \left\| 	\left( \sum_{k_2 =1}^\infty \left( \sum_{k_1 =1}^\infty  |\M f_{k_1, k_2}|^{u_1} \right) ^{u_2 /u_1 } \right)^{1/u_2}   \chi_Q \right\|_{L^{\vec p}  } ^r  \right)^{1/r}   \\
			& \lesssim \sum_{m=0 }^\infty  2^{-mn/t}  \left(\sum_{Q \in \D_{\vec a}} |Q_m|^{ \frac{r}{t}  - \frac{r}{n} \sum_{i=1 } ^n \frac{1}{p_i}  }  \left\|	 \left\| 	f_{k_1, k_2}   \right\|_{\ell^{(u_1,u_2) } } \chi_{Q_m} \right\|_{L^{\vec p}  } ^r  \right)^{1/r}.
		\end{align*}
		Using the fact that there exist $2^{mn}$ children of the cube $Q_m$, namely, there exists $2^{mn}$ dyadic cubes in $\D_{ -\log_2  \ell (Q)  +m } $ that contains $Q$, we see that
		\begin{align*}
			& \left(\sum_{Q \in \D_{\vec a}} |Q|^{ \frac{r}{t}  - \frac{r}{n} \sum_{i=1 } ^n \frac{1}{p_i}  }  \left\| 	\left( \sum_{k_2 =1}^\infty \left( \sum_{k_1 =1}^\infty  |\M f_{k_1, k_2}|^{u_1} \right) ^{u_2 /u_1 } \right)^{1/u_2}  \chi_Q \right\|_{L^{\vec p}  } ^r  \right)^{1/r}   \\
			& \lesssim \sum_{m=0 }^\infty  2^{-mn/t}  \left( 2^{mn} \sum_{Q \in \D_{\vec a}} |Q|^{ \frac{r}{t}  - \frac{r}{n} \sum_{i=1 } ^n \frac{1}{p_i}  }  \left\| 	\left\| 	f_{k_1, k_2}   \right\|_{\ell^{(u_1,u_2) } }     \chi_{Q} \right\|_{L^{\vec p}  } ^r  \right)^{1/r} \\
			& \lesssim \left(  \sum_{Q \in \D_{\vec a}} |Q|^{ \frac{r}{t}  - \frac{r}{n} \sum_{i=1 } ^n \frac{1}{p_i}  }  \left\| \left\| 	f_{k_1, k_2}   \right\|_{\ell^{(u_1,u_2) } }   \chi_{Q} \right\|_{L^{\vec p}  } ^r  \right)^{1/r}.
		\end{align*}
			Thus the proof is finished.		
\end{proof}

\begin{lemma}[Proposition 6.4, \cite{N19}]
	\label{Mit vector}
	 Let $ 1 < \vec p <\infty$  and $\omega _k  \in A_{p_k} (\mathbb R)$ for $k =1 , \ldots, n$. Then for all $f\in L^0 $,
	\begin{equation*}
	\left\|  \left(  \sum_{j=1}^\infty [ \Mit f_j ]^u  \right) ^{1/u}  \cdot \bigotimes_{k=1}^n \omega_k ^{1/p_k} \right\|_{L^{\vec p} } \lesssim \left\|  \left(  \sum_{j=1}^\infty |f_j |^u  \right) ^{1/u}  \cdot \bigotimes_{k=1}^n \omega_k ^{1/p_k} \right\|_{L^{\vec p} }
	\end{equation*}
\end{lemma}

\begin{theorem}
		Let $ 1< \vec p < \infty $. Let    $ 1 < n / ( \sum_{i=1}^n  1/p_i )   < t <r <\infty $ or $ 1 < n / ( \sum_{i=1}^n  1/p_i )  \le t < r=\infty  $. 
	Furthermore, let  
	\begin{equation*}
		\left(\frac{1}{n} \sum_{i=1 } ^n \frac{1}{p_i} -\frac{1}{t}  +  \frac{1}{r} \right) <  \frac{1}{\max \vec p}.
	\end{equation*}
Let $1 <u \le \infty $.
	Then for $ \{ f_k\}_{k\in \mathbb N} \in  M_{\vec p}^{t,r} (\ell^u)$,
	\begin{equation*}
		 \left(\sum_{Q \in \D} |Q|^{ \frac{r}{t}  - \frac{r}{n} \sum_{i=1 } ^n \frac{1}{p_i}  }  \left\| \left( \sum_{k=1}^\infty  |\Mit f_k|^u \right)^{1/u}  \chi_Q \right\|_{L^{\vec p}  } ^r  \right)^{1/r}   \lesssim \| \{ f_k\}_{k\in \mathbb N} \|_{ M_{\vec p}^{t,r}  } .
	\end{equation*}
\end{theorem}
\begin{proof}
	Let $\eta \in (0,1)  $  such that
	\begin{equation*}
		0<   \left(\frac{1}{n} \sum_{i=1 } ^n \frac{1}{p_i} -\frac{1}{t}  +  \frac{1}{r} \right)  <\eta  < \frac{1}{\max \vec p}<1 .
	\end{equation*}
		By Proposition \ref{mixed BM equ norm weight}, it suffices to show 
	\begin{equation*}
		\left\| \left( \sum_{k=1}^\infty  |\Mit f_k|^u \right)^{1/u}   (\Mit \chi_Q )^\eta  \right\|_{L^{\vec p}  }  \lesssim 	\left\| \left( \sum_{k=1}^\infty  | f_k|^u \right)^{1/u}   (\Mit \chi_Q )^\eta  \right\|_{L^{\vec p}  } .
	\end{equation*}
Let $Q = I_1 \times \dots \times I_n $. Then 
\begin{equation*}
	( \Mit \chi_Q) ^{\eta }  = \left(  \bigotimes_{i=1}^n \M _{(i)} \chi_{I_i} \right)^{\eta }  =  \bigotimes_{i=1}^n ( \M_{(i)} \chi_{I_i} )^{\eta } .
\end{equation*}
From \cite[Theorem 288]{SDH20}, $ ( \M _{(i)} \chi_{I_i} )^{\eta  p_i} $  belongs to $A_1(\mathbb R)$ if  $\eta > 0$ such that
\begin{equation*}
	0 < \eta  p_i < 1 .
\end{equation*}
Hence $ ( \M _{(i)} \chi_{I_i} )^{\eta  p_i} \in A_1(\mathbb R) \subset A_{p_i}(\mathbb R) $ for all $p_i$. By Lemma \ref{Mit vector}, we obtain
\begin{align*}
	\left\| \left( \sum_{k=1}^\infty  |\Mit f_k|^u \right)^{1/u}   (\Mit \chi_Q )^\eta  \right\|_{L^{\vec p}  }	&  = \left\| \left( \sum_{k=1}^\infty  |\Mit f_k|^u \right)^{1/u}   \bigotimes_{i=1}^n (\M_{(i)} \chi_{I_i } ) ^\eta \right\|_{L^{\vec p}  }	  \\
	& \lesssim \left\| \left( \sum_{k=1}^\infty  | f_k|^u \right)^{1/u}   \bigotimes_{i=1}^n (\M_{(i)} \chi_{I_i } ) ^\eta \right\|_{L^{\vec p}  }	  \\
	&=  \left\| \left( \sum_{k=1}^\infty  | f_k|^u \right)^{1/u}   	( \Mit \chi_Q) ^{\eta }  \right\|_{L^{\vec p}  }	 .
\end{align*}
We finish the proof.
\end{proof}

\subsection{The Hardy-Littlewood maximal operator on  block spaces}
We first give the definition of vector valued block space, which is similar with Definition \ref{def mix block space}.

\begin{definition}
	Let $ 1 \le \vec p <\infty$ and $1 < u' \le \infty $. Let   $n / ( \sum_{i=1}^n  1/p_{i})  \le t \le r \le \infty $. A measurable vector valued function $\vec b =\{b_k \}_{k\in \mathbb N_0}$ is said to be a $\ell^{u' }$ valued
	$(\vec{p}^{\,\prime}, t')$-block if there exists a cube $Q$ that supports $\vec b$ such that 
	\begin{equation*}
		\| \vec b \|_{L^{\vec{p}^{\,\prime} } (\ell^{u'}) } : =  \left\| \left( \sum_{m =0 } ^\infty | b_m| ^{u'} \right)^{1/u'} \right\|_{L^{\vec{p}^{\,\prime} } }  \le |Q| ^{  1/t - \frac{1}{n}  ( \sum_{i=1}^n  1/p_{i})   }  .
	\end{equation*}
	If we need to indicate $Q$, we  say that $ \vec b$ is a  $\ell^{u' }$ valued $(\vec{p}^{\,\prime},t')$-block supported on $Q$.

	The function space $\mathcal{H}_{\vec p \, ^\prime}^{t',r'} (\ell^ {u'}) $
	is the set of  all measurable functions $\vec f = \{f_m \}_{m \in \mathbb N_0}$ such that $\vec f$ is realized
	as the sum
	\begin{equation}\label{eq: mix block vector f}
	\vec	f = \sum_{(j,k)\in\mathbb{Z}^{n+1}}\lambda_{j,k} \vec b_{j,k}
	\end{equation}
	with some $\lambda=\{\lambda_{j,k}\}_{(j,k)\in\mathbb{Z}^{n+1}}\in\ell^{r'}(\mathbb{Z}^{n+1})$
	and $\vec b_{j,k}$ is a $\ell^{u'}$ valued $(\vec{p}^{\,\prime},t')$-block supported on  $Q_{j,k}$ where (\ref{eq: mix block vector f}) converges almost everywhere on $\rn$. The norm of $\mathcal{H}_{\vec p \, ^\prime}^{t',r'}(\ell^ {u'}) $
	is defined by
	\[
	\|f\|_{\mathcal{H}_{\vec p \, ^\prime}^{t',r'} (\ell^ {u'})} :=\inf_{\lambda}\|\lambda\|_{\ell^{r'}},
	\]
	where the infimum is taken over all admissible sequences $\lambda$
	such that (\ref{eq: mix block vector f}) holds.
\end{definition}

\begin{lemma}[Theorem 2.18, \cite{N24}] \label{HL mix block r = infty}
		Let $ 1< \vec p < \infty $ and $1<u' \le \infty $. Let    $ 1< n / ( \sum_{i=1}^n  1/p_i )  \le t < r=\infty  $.
	Then for all $f  \in \mathcal{H}_{\vec p \, ^\prime}^{t',r'}  (\ell^{u'})  $, we have
	\begin{equation*}
		\left\|  \left(   \sum_{j=0}^\infty \M (f_j) ^{u'}\right)^{1/ u'}   \right\|_{\mathcal{H}_{\vec p \, ^\prime}^{t',r'}  } \lesssim 	\left\|  \left(   \sum_{j=0}^\infty |f_j| ^{u'}\right)^{1/ u'}   \right\|_{\mathcal{H}_{\vec p \, ^\prime}^{t',r'}  } .
	\end{equation*}
\end{lemma}

\begin{theorem} \label{HL mix block r le infty}
		Let $ 1< \vec p <\infty $ and $1<u' \le \infty $. Let    $1 < n / ( \sum_{i=1}^n  1/p_i )   < t <r <\infty $ or $ 1< n / ( \sum_{i=1}^n  1/p_i )  \le t < r=\infty  $.
	Then for all $\{f_j \}_{j\in \mathbb N_0}  \in \mathcal{H}_{\vec p \, ^\prime}^{t',r'}  (\ell^{u'})  $, we have
	\begin{equation*}
		\left\|  \left(   \sum_{j=0}^\infty \M (f_j) ^{u'}\right)^{1/ u'}   \right\|_{\mathcal{H}_{\vec p \, ^\prime}^{t',r'}  } \lesssim 	\left\|  \left(   \sum_{j=0}^\infty |f_j| ^{u'}\right)^{1/ u'}   \right\|_{\mathcal{H}_{\vec p \, ^\prime}^{t',r'}  } .
	\end{equation*}
\end{theorem}

\begin{remark} \label{remark HL block}
	In \cite[Theorem 4.3]{ST09}, Sawano and Tanaka proved the Hardy-Littlewood  maximal operator $\M$ on vector valued on block spaces $ \mathcal{H}_{p'}^{t',1} (\ell^{u'})$; see also \cite[Theorem 2.12]{IST15}. 
	Since Lemma \ref{HL mix block r = infty}, 
	we only need to prove the case  $1 < n / ( \sum_{i=1}^n  1/p_i )   < t <r <\infty $.  We use the idea from \cite[Section 6.2 Besov-Bourgain-Morrey spaces]{ZYY242}. Before proving  Theorem \ref{HL mix block r le infty}, we give some preparation.
\end{remark}

\begin{lemma} \label{norm equ}
		Let $ 0 < \vec p \le \infty $. Let  $   0 < n / ( \sum_{i=1}^n  1/p_{i})     < t <r <\infty $ or $ 0< n / ( \sum_{i=1}^n  1/p_{i})   \le t < r=\infty  $.
	  Then $f \in M_{\vec p}^{t,r} $ if and only if $f \in L_{\mathrm{loc}}^{\vec p} $ and
	\begin{equation*}
		\|f\|'_{ M_{\vec p}^{t,r} } := \left(  \int_0^\infty \int_\rn \left( |B(y,t)| ^{1/t - \frac{\sum_{i=1}^n \frac{1}{ p_i} }{n}  -1/r}  \left\| f \chi_{B(y,t)  } \right\|_{L^{\vec p} }  \right) ^r \d y \frac{\d t}{t}   \right)^{1/r} <\infty,
	\end{equation*}
	with the usual modifications made when $r =\infty$. Moreover,
	\begin{equation*}
		\|f \|_{ M_{\vec  p}^{t,r}}  \approx \|f\|'_{ M_{\vec p}^{t,r} } .
	\end{equation*}
\end{lemma}
\begin{proof}
	Repeating  \cite[Theorem 2.9]{ZSTYY23} and replacing  $ \| \cdot \|_{L^p}$ by $\| \cdot \|_{L^{\vec p}}$, we obtain the result. 
\end{proof}

\begin{remark} 
	Let $1<u\le \infty $. 	Let $ 1< \vec p < \infty $ and    $1 < n / ( \sum_{i=1}^n  1/p_i )   < t <r <\infty $.  From Theorem \ref{HL seq mix BM}, we know that  the vector  Hardy-Littlewood maximal  operator $\M$ is bounded on  $ M_{\vec p}^{t,r} (\ell^{u})$. By Lemma \ref{norm equ}, we obtain
	\begin{equation*}
		\|\{ \M  f_k \}_{k\in\mathbb N} \|_{ M_{p}^{t,r} (\ell^u) } ' \lesssim \| \{   f_k \}_{k\in\mathbb N}\|_{ M_{p}^{t,r} (\ell^u) } '.
	\end{equation*}		
\end{remark}
\begin{lemma}\label{HL each j}
	Let $1<u \le \infty $ and $ 1< \vec p < \infty $.   Let    $1 < n / ( \sum_{i=1}^n  1/p_i )   < t <r <\infty $ or $ 1< n / ( \sum_{i=1}^n  1/p_i )  \le t < r=\infty  $.  Then for each $v \in \mathbb Z $,
	\begin{align*}
		& \left(  \sum_{m \in \mathbb Z^n} \left(  |Q_{v,m}|^{1/t-\frac{\sum_{i=1}^n \frac{1}{ p_i} }{n} } \left\| \left( \sum_{\ell =1}^\infty ( \M f_\ell) ^u \right)^{1/u}  \chi_{ Q_{v,m} }\right\|_{L^{\vec p} }    \right) ^r \right) ^{1/r} \\
		& \lesssim  \left(  \sum_{m \in \mathbb Z^n} \left(  |Q_{v,m}|^{1/t-\frac{\sum_{i=1}^n \frac{1}{ p_i} }{n} } \left\| \left( \sum_{\ell =1}^\infty |f_\ell| ^u \right)^{1/u}  \chi_{ Q_{v,m} }\right\|_{L^{\vec p} }    \right) ^r \right) ^{1/r}.
	\end{align*}
\end{lemma}
\begin{proof}
	Repeating  the proof of  Theorem \ref{HL seq mix BM}, the result is obtained and we omit the detail here.
\end{proof}
\begin{remark}
	Let $1< u \le \infty $ and $ 1< \vec p < \infty $.  Let    $1 < n / ( \sum_{i=1}^n  1/p_i )   < t <r <\infty $ or $ 1< n / ( \sum_{i=1}^n  1/p_i )  \le t < r=\infty  $.
	By Lemmas \ref{norm equ} and \ref{HL each j}, we have that for each $j\in\mathbb Z$,
	\begin{align*}
		& 	 \int_\rn \left( |B(y,2^{-j})| ^{1/t -\frac{\sum_{i=1}^n \frac{1}{ p_i} }{n}  -1/r}
		\left\| \left(  \sum_{\ell=1}^\infty ( \M f_\ell ) ^u \right)^{1/u} \chi_{B(y, 2^{-j})  } \right\|_{L^{\vec p} }\right)^{r} \d y
		   \\
		& \lesssim 	 \int_\rn \left( |B(y,2^{-j})| ^{1/t -\frac{\sum_{i=1}^n \frac{1}{ p_i} }{n}  -1/r}
		\left\| \left(  \sum_{\ell=1}^\infty | f_\ell | ^u \right)^{1/u} \chi_{B(y, 2^{-j})  } \right\|_{L^{\vec p} }\right)^{r} \d y  .
	\end{align*}
\end{remark}

\begin{definition}\label{def vec skice}
		
	Let $1< u' <\infty$ and $1 <\vec p <\infty$.
	 Let    $1 < n / ( \sum_{i=1}^n  1/p_i )   < t <r <\infty $ and $j \in \mathbb Z$.	 
  The slice space $  (\mathcal E_{\vec p \, ^\prime}^{t',r'}   )_j  (\ell^{u'}) $ is defined to be the set of all  $ \{ f_w \}_{w\in \mathbb N} \in L_{\operatorname{loc} } ^{\vec p \, ^\prime} (\ell^{u'})$  such that 
	\begin{align*}
		& \|\{ f_w \}_{w\in \mathbb N} \|_{  (\mathcal E_{\vec p \, ^\prime}^{t',r'}   )_j  (\ell^{u'} ) } \\
		& := \left( \sum_{k\in \mathbb Z^n } \left( |Q_{j,k}|^{1/t' -\frac{\sum_{i=1}^n \frac{1}{ p_i '} }{n} }   \left\| \left(\sum_{w=1} |f_w|^{u'}  \right)^{1/u'}  \chi_{ Q_{j,k} } \right \|_{ L^{\vec p \, ^\prime}  }   \right) ^{r'}  \right) ^{1/r'} <\infty.
	\end{align*}
\end{definition}

From Lemma \ref{HL each j},  we know that $\M$ is bounded on   $  (\mathcal E_{\vec p \, ^\prime}^{t',r'}   )_j  (\ell^{u'}) $.

In the following lemma, we establish a characterization of  slice space  $  (\mathcal E_{\vec p \, ^\prime}^{t',r'}   )_j  (\ell^{u'}) $.

\begin{lemma} \label{char vec slice}
	Let $1< u' <\infty$ and $1 <\vec p <\infty$.
Let    $1 < n / ( \sum_{i=1}^n  1/p_i )   < t <r <\infty $ and $j \in \mathbb Z$.  Then $\vec f = \{ f_w\}_{w\in \mathbb N} \in (\mathcal E_{\vec p \, ^\prime}^{t',r'}   )_j  (\ell^{u'})  $ if and only if there exist a sequence $  \{ \lambda_{j,k}\}_{k\in \mathbb Z^n}  \subset \mathbb R$ satisfying
	\begin{equation*}
		\left(  \sum_{m\in \mathbb Z^n} |\lambda_{j,k}|^{r'} \right)^{1/r'} <\infty
	\end{equation*}
	and a sequence $\{\vec  b_{j,k} \}_{k\in \mathbb Z^n}$ of vector functions on $\rn$ satisfying, for each $k\in \mathbb Z^n$, supp $\vec b_{j,k}\subset Q_{j,k}$  and 
	\begin{equation} \label{b jk = |Q|}
		\left\|  \| \vec b_{j,k}\|_{\ell^{u'}}  \right\|_{ L^{ \vec p \, ^\prime} (Q_{j,k}) }  = |Q_{j,k}|^{1/t - \frac{\sum_{i=1}^n \frac{1}{ p_i} }{n} }
	\end{equation}
	such that $\vec f = \sum_{k \in \mathbb Z^n} \lambda_{j,k} \vec  b_{j,k} $ almost everywhere on $\rn$; moreover, for such $\vec f$,
	\begin{equation*}
		\| \vec  f\|_{ (\mathcal E_{p'}^{t',r'}   )_j (\ell^{q'}) }  = 	\left(  \sum_{m\in \mathbb Z^n} |\lambda_{j,k}|^{r'} \right)^{1/r'}  .
	\end{equation*}
\end{lemma}
\begin{proof}
	We first show the necessity. Let $\vec f = \{ f_w\}_{w\in \mathbb N}  \in  (\mathcal E_{\vec p \, ^\prime}^{t',r'}   )_j  (\ell^{u'}) $. For each $k \in \mathbb Z^n$, when $  \left\|  \|\vec f\|_{\ell^{u'}} \chi_{Q_{j,k} } \right\|_{ L^{\vec p \, ^\prime}  } >0 $, let $ \lambda_{j,k} := |Q_{j,k}|^{ \frac{\sum_{i=1}^n \frac{1}{ p_i} }{n} -1/t }   \| \chi_{ Q_{j,k} } \vec f\|_{ L^{p'} (\ell^{u'}) }  $,   $\vec  b_{j,k} := \lambda_{j,k} ^{-1} \vec f \chi_{Q_{j,k} }$ and, when $  \left\|  \|\vec f\|_{\ell^{u'}} \chi_{Q_{j,k} } \right\|_{ L^{\vec p \, ^\prime}  } = 0 $, let $ \lambda_{j,k} := 0$ and $\vec  b_{j,k} : = |Q_{j,k}|^{ -1/t'} \chi_{ Q_{j,k} }  \vec a $ where $\vec a = (1, 0, 0, \ldots) \in \ell^{u'} $ with $\|\vec a\|_{ \ell^{u'}} =1$. Then by the Definition \ref{def vec skice}, it is easy to show that supp $\vec b_{j,k}\subset Q_{j,k}$, (\ref{b jk = |Q|}) and 
	\begin{equation*}
		\left(  \sum_{k\in \mathbb Z^n} |\lambda_{j,k}|^{r'} \right)^{1/r'}  = 	\|\vec f\|_{  (\mathcal E_{\vec p \, ^\prime}^{t',r'}   )_j  (\ell^{u'})  }  	  <\infty .
	\end{equation*} 	
	Next we show the sufficiency. Assume that there exist a sequence $  \{ \lambda_{j,k}\}_{k\in \mathbb Z^n}  \in \ell^{r'}$ and a sequence $\{\vec  b_{j,k} \}_{k\in \mathbb Z^n}$ satisfying supp $\vec b_{j,k} \subset Q_{j,k}$ and  (\ref{b jk = |Q|}) such that $\vec f =  \sum_{k \in \mathbb Z^n} \lambda_{j,k} \vec b_{j,k}$  almost everywhere on $\rn$. Note that $ \{ Q_{j,k}\}_{k\in \mathbb Z^n}$  are disjoint. Then 
	\begin{align*}
		\|\vec  f\|_{  (\mathcal E_{\vec p \, ^\prime}^{t',r'}   )_j  (\ell^{u'})  }   & = \left( \sum_{k\in \mathbb Z^n } \left( |Q_{j,k}|^{\frac{\sum_{i=1}^n \frac{1}{ p_i} }{n}  - 1/t }   \| \lambda_{j,k}\vec  b_{j,k}\|_{ L^{\vec p \, ^\prime} (\ell^{u'}) }   \right) ^{r'}  \right) ^{1/r'}  \\
		& = 	\left(  \sum_{k\in \mathbb Z^n} |\lambda_{j,k}|^{r'} \right)^{1/r'}  <\infty.
	\end{align*}
	Hence $\vec f \in (\mathcal E_{\vec p \, ^\prime}^{t',r'}   )_j  (\ell^{u'})$. Thus we finish the proof.
\end{proof}
Using the characterization of vector valued slice space $(\mathcal E_{\vec p \, ^\prime}^{t',r'}   )_j  (\ell^{u'})$, we obtain the following characterization of block spaces.
\begin{lemma}\label{char block}
	Let $1< u' <\infty$ and $1 <\vec p <\infty$.
Let    $1 < n / ( \sum_{i=1}^n  1/p_i )   < t <r <\infty $.  Then $\vec f \in \mathcal{H}_{\vec p \, ^\prime}^{t',r'} (\ell^{u'})$ if and only if 
	\begin{align*}
		\| \vec f\|^{\blacklozenge}_{\mathcal{H}_{\vec p \, ^\prime}^{t',r'} (\ell^{u'}) } & : = \inf\left\{ \left( \sum_{j\in \mathbb Z} \|\vec  f_j\|_{  (\mathcal E_{\vec p \, ^\prime}^{t',r'}   )_j  (\ell^{u'})  } ^{r'} \right) ^{1/r' }  :  \vec f = \sum_{j\in \mathbb Z} \vec f_j, \vec f_j \in  (\mathcal E_{\vec p \, ^\prime}^{t',r'}   )_j  (\ell^{u'})  \right\} \\
		& <\infty ;
	\end{align*}
	moreover, for such $\vec f$, $ \|\vec  f\|_{ \mathcal{H}_{\vec p \, ^\prime}^{t',r'} (\ell^{u'}) } = 	\| \vec f\|^{\blacklozenge}_{\mathcal{H}_{\vec p \, ^\prime}^{t',r'} (\ell^{u'}) }$.
\end{lemma}
\begin{proof}
	We first show the necessity. Let $\vec f \in \mathcal{H}_{\vec p \, ^\prime}^{t',r'} (\ell^{u'}) $. Then	there exist a  sequence $\lambda=\{\lambda_{j,k}\}_{(j,k)\in\mathbb{Z}^{n+1}}\in\ell^{r'}(\mathbb{Z}^{n+1})$ and a  $\ell^{u'} $-valued sequence $\{\vec  b_{j,k}  \}_{(j,k)\in\mathbb{Z}^{n+1} }  $  of  $(\vec p \, ^\prime,t' )$-block such that 
	\begin{equation*}
		\vec f=\sum_{(j,k)\in\mathbb{Z}^{n+1}}\lambda_{j,k} \vec b_{j,k}  
	\end{equation*}
	almost everywhere on $\rn$ and $ \left( \sum_{ (j,k)\in\mathbb{Z}^{n+1} } |\lambda_{j,k}|^{r'}\right)^{1/r'}  < (1+\epsilon) \|\vec  f\|_{ \mathcal{H}_{\vec p \, ^\prime}^{t',r'} (\ell^{u'})}$.
	For each $j \in \mathbb Z$, let $\vec f_j : = \sum_{k \in \mathbb Z^n}  \lambda_{j,k} \vec  b_{j,k}  $. Then we obtain
	\begin{align*}
		\left(  \sum_{j\in \mathbb Z}	\| \vec f_j \|_{ (\mathcal E_{\vec p \, ^\prime}^{t',r'}   )_j  (\ell^{u'}) }  ^{r'}  \right)^{1/r'} 
		& = \left(  \sum_{j\in \mathbb Z} \left(  \sum_{k\in \mathbb Z^n } \left( |Q_{j,k}|^{\frac{\sum_{i=1}^n \frac{1}{ p_i} }{n}  - 1/t}   \| \lambda_{j,k}\vec  b_{j,k}\|_{ L^{\vec p \, ^\prime} (\ell^{u'}) }   \right) ^{r'}  \right)^{r'/r'}   \right) ^{1/r'} \\
		& \le 	\left(  \sum_{j\in \mathbb Z} \sum_{m\in \mathbb Z^n} |\lambda_{j,k}|^{r'} \right)^{1/r'} < (1+\epsilon) \| \vec  f\|_{ \mathcal{H}_{\vec p \, ^\prime}^{t',r'} (\ell^{u'}) }.
	\end{align*}
	Letting $\epsilon \to 0^+$, we obtain $  	\| \vec f\|^{\blacklozenge}_{\mathcal{H}_{\vec p \, ^\prime}^{t',r'} (\ell^{u'})} \le  \|\vec  f\|_{ \mathcal{H}_{\vec p \, ^\prime}^{t',r'} (\ell^{u'})} $.
	
	Next we show the sufficiency. Let $\vec f $ be a measurable vector function with $ 	\| \vec f\|^{\blacklozenge}_{\mathcal{H}_{\vec p \, ^\prime}^{t',r'} (\ell^{u'})} <\infty$. Then there exists a sequence $\{\vec f_j \}_{j\in \mathbb Z}$ satisfying that $\vec f_j \in  (\mathcal E_{\vec p \, ^\prime}^{t',r'}   )_j (\ell^{u'})  $ for each $j$ such that $\vec f =\sum_{j\in \mathbb Z} \vec f_j $ almost everywhere on $\rn$ and 
	\begin{equation*}
		\left( \sum_{j\in \mathbb Z} \|\vec  f_j\|_{ (\mathcal E_{\vec p \, ^\prime}^{t',r'}   )_j  (\ell^{u'}) } ^{r'} \right) ^{1/r' } < (1+\epsilon) 	\| \vec f\|^{\blacklozenge}_{\mathcal{H}_{\vec p \, ^\prime}^{t',r'} (\ell^{u'})}.
	\end{equation*}
	By  Lemma \ref{char vec slice}, for each $j\in \mathbb Z$, there exist a sequence 
	$\{ \lambda_{j,k} \}_{ (j,k)\in\mathbb{Z}^{n+1} }  \in \ell^{r'}$  and 
	a vector valued sequence $ \{\vec  b_{j,k}  \}_{ (j,k) \in \mathbb{Z}^{n+1} }$   satisfying (\ref{b jk = |Q|}) such that $\vec f_j = \sum_{k \in \mathbb Z^n} \lambda_{j,k}\vec  b_{j,k} $ almost everywhere on $\rn$ and 
	\begin{equation*}
		\|\vec  f_j \|_{  (\mathcal E_{\vec p \, ^\prime}^{t',r'}   )_j  (\ell^{u'})  }  = 	\left(  \sum_{m\in \mathbb Z^n} |\lambda_{j,k}|^{r'} \right)^{1/r'} .
	\end{equation*}
	Hence, $\vec f = \sum_{j\in \mathbb Z}\vec f_j = \sum_{j\in \mathbb Z}\sum_{k \in \mathbb Z^n} \lambda_{j,k}\vec  b_{j,k}  $ almost everywhere on $\rn$ and 
	\begin{equation*}
		\|\vec  f\|_{ \mathcal{H}_{\vec p \, ^\prime}^{t',r'} (\ell^{u'})  }  \le 	\left(  \sum_{j\in \mathbb Z} \sum_{m\in \mathbb Z^n} |\lambda_{j,k}|^{r'} \right)^{1/r'} < (1+\epsilon) 	\| \vec f\|^{\blacklozenge}_{\mathcal{H}_{\vec p \, ^\prime}^{t',r'} (\ell^{u'}) }.
	\end{equation*}
	Letting $\epsilon \to 0^+$, we obtain $	\| \vec f\|_{ \mathcal{H}_{\vec p \, ^\prime}^{t',r'} (\ell^{u'}) } \le   	\| \vec f\|^{\blacklozenge}_{\mathcal{H}_{\vec p \, ^\prime}^{t',r'} (\ell^{u'}) } $. This finish the proof of sufficiency. Thus we complete the proof of Lemma \ref{char block}.
\end{proof}
Now we are ready to show  Theorem \ref{HL mix block r le infty}.
\begin{proof}[Proof of Theorem \ref{HL mix block r le infty}]
		Just as in Remark \ref{remark HL block}, we only need to show the case  $1<u' \le \infty $ and    $1 < n / ( \sum_{i=1}^n  1/p_i )   < t <r <\infty $.
	
Subcase $1<u' <\infty $ and $1 < n / ( \sum_{i=1}^n  1/p_i )   < t <r <\infty $.	

	Let $0<\epsilon <1$.
	Let $\vec f\in \mathcal{H}_{\vec p \, ^\prime}^{t',r'} (\ell^{u'}) $. From Lemma \ref{char block}, there exists a sequence $\{\vec f_j\}_{j\in \mathbb Z} \subset (\mathcal E_{\vec p \, ^\prime}^{t',r'}   )_j (\ell^{u'}) $ such that $\vec f = \sum_{j\in \mathbb Z} \vec f_j$ almost everywhere on $\rn$ and 
	\begin{equation*}
		\| \vec f\|_{\mathcal{H}_{\vec p \, ^\prime}^{t',r'} (\ell^{u'})  } = 	\| \vec f\|^{\blacklozenge}_{\mathcal{H}_{\vec p \, ^\prime}^{t',r'} (\ell^{u'})  } > (1 - \epsilon) \left( \sum_{j\in \mathbb Z} \| \vec f_j\|_{ (\mathcal E_{\vec p \, ^\prime}^{t',r'}   )_j (\ell^{u'})  } ^{r'} \right) ^{1/r' } .
	\end{equation*}
	By the sublinear of Hardy-Littlewood maximal function and Lemma \ref{HL each j}, we obtain
	\begin{align*}
		\| \M \vec f\|_{\mathcal{H}_{\vec p \, ^\prime}^{t',r'} (\ell^{u'}) }  & \le \left\| \sum_{j\in   \mathbb Z}  \M \vec f _j  \right\|_{ \mathcal{H}_{\vec p \, ^\prime}^{t',r'} (\ell^{u'}) }  = \left\| \sum_{j\in   \mathbb Z}  \M  \vec f _j \right\|^{\blacklozenge}_{ \mathcal{H}_{\vec p \, ^\prime}^{t',r'} (\ell^{u'})   }  \\
		& \le \left( \sum_{j\in   \mathbb Z} \| \M  \vec f _j  \|_{ (\mathcal E_{\vec p \, ^\prime}^{t',r'}   )_j (\ell^{u'}) } ^{r'} \right)^{1/r'}  \lesssim  \left( \sum_{j\in   \mathbb Z} \|  \vec f _j  \|_{ (\mathcal E_{\vec p \, ^\prime}^{t',r'}   )_j (\ell^{u'})} ^{r'} \right)^{1/r'}  \\
		& < \frac{1}{1-\epsilon} 	\| \vec f\|_{\mathcal{H}_{\vec p \, ^\prime}^{t',r'} (\ell^{u'}) }
	\end{align*}
	where the implicit positive constants are independent of both $\vec f$ and $\epsilon$. Letting $\epsilon \to 0^+$, we obtain $	\| \M \vec  f\|_{\mathcal{H}_{\vec p \, ^\prime}^{t',r'} (\ell^{u'})} \lesssim  	\| \vec f\|_{\mathcal{H}_{\vec p \, ^\prime}^{t',r'} (\ell^{u'})} $.
	
	Let $ \vec g =\{ g_j \}_{j =0}^\infty  $ be defined by $	g_0 = f, g_j =0$ for $j \in \mathbb N$. Then from the above, we obtain the scalar version:
	\begin{equation} \label{scalar HL mixed block}
	\| \M   f\|_{\mathcal{H}_{\vec p \, ^\prime}^{t',r'} }	=	\| \M \vec  g\|_{\mathcal{H}_{\vec p \, ^\prime}^{t',r'} (\ell^{u'})} \lesssim  	\| \vec g\|_{\mathcal{H}_{\vec p \, ^\prime}^{t',r'} (\ell^{u'})} = \|   f\|_{\mathcal{H}_{\vec p \, ^\prime}^{t',r'} } .
	\end{equation}

Subcase $u' = \infty $ and $1 < n / ( \sum_{i=1}^n  1/p_i )   < t <r <\infty $.	
By 	\begin{equation*}
	\sup_{k \in \mathbb N} \M f_k  \le  \M \sup_{k \in \mathbb N} f_k, 
\end{equation*} and (\ref{scalar HL mixed block}),
	we obtain 
	\begin{equation*}
			\| \M \vec  f\|_{\mathcal{H}_{\vec p \, ^\prime}^{t',r'} (\ell^{\infty})} \lesssim  	\| \vec f\|_{\mathcal{H}_{\vec p \, ^\prime}^{t',r'} (\ell^{\infty})}. 
	\end{equation*}
	Hence, we finish the proof.	
\end{proof}

\subsection{Fractional integral operator}
For $ 0<\alpha <n $, define the fractional integral operator $I_\alpha$ of order $\alpha$ by
\begin{equation*} 
	I_\alpha f (x) := \int_\rn \frac{f (y)}{|x-y|^{n-\alpha} }  \d y
\end{equation*}
for $f \in L^1_{ \operatorname{loc}} $  as long as the right-hand side makes sense.

From the standard argument for fractional integral operators on mixed Morrey spaces (see \cite[proof of Theorm 1.11]{N19}), we get that the pointwise estimate for non-negative functions $f$   is 
\begin{equation} \label{frac estimate}
	I_\alpha f (x)  \lesssim \| f\|_{ M_{\vec p}^{t,\infty} }^{ t\alpha /n  } \M f (x) ^{ 1 - t\alpha /  n }
\end{equation}
where $1<\vec p <\infty $ and  $ 1< n / ( \sum_{i=1}^n  1/p_{i})  \le t < \infty  $.

\begin{theorem}
	Let $ 1< \vec p_1 , \vec p_2< \infty $.  
	For $j=1,2$, let   $1 < n / ( \sum_{i=1}^n  1/p_{j,i})   < t_j <r_j <\infty $ or $ 1< n / ( \sum_{i=1}^n  1/p_{j,i})  \le t_j < r_j =\infty  $ where $p_{j,i}$  is $i$-th item of the vector $\vec p_j$.
		Let $ 0<\alpha <n $ and $1/ t_2 = 1 /t_1 - \alpha /n $. Assume that
	\begin{align*}
	 \frac{t_1}{t_2} \vec p_2 = \vec p_1, 
		&  \quad  \frac{t_1}{t_2} r_2 = r_1. 
	\end{align*}	
	Then \begin{equation*}
		\|  I_\alpha  f \|_{M_{\vec p_2 }^{t_2,r_2	}  } \lesssim   \| f\|_{ M_{\vec p_1}^{t_1,r_1} } .
	\end{equation*}
	
\end{theorem}

\begin{proof}
	We have a pointwise estimate (\ref{frac estimate}) for  $|f|$. 	
	Using the embedding  $M_{\vec p}^{t,r}  \hookrightarrow  M_{\vec p}^{t,\infty} $
	 and the condition  $1/ t_2 = 1 /t_1 - \alpha /n  $, 
	 we obtain the  estimate
	 \begin{equation*}
	 	| I_\alpha  f(x) | \le  I_\alpha  [ | f| ](x) \lesssim \| f\|_{ M_{\vec p_1}^{t_1,r_1} }^{ t_1 \alpha /n  } \M f (x) ^{ 1 - t_1 \alpha /  n } .
	 \end{equation*} 
 Hence
 \begin{align*}
 	\|  I_\alpha  f \|_{M_{\vec p_2 }^{t_2,r_2	}  }  & \lesssim \| f\|_{ M_{\vec p_1}^{t_1,r_1} }^{ t_1 \alpha /n  }  	\|  \M f  ^{ 1 - t_1 \alpha /  n }  \|_{M_{\vec p_2 }^{t_2,r_2	}  }  \\
 	& 	= \| f\|_{ M_{\vec p_1}^{t_1,r_1} }^{ t_1 \alpha /n  }  	\|  \M f   \|_{M_{  (1 - t_1 \alpha /  n ) \vec p_2 }^{ (1 - t_1 \alpha /  n) t_2, (1 - t_1 \alpha /  n) r_2	}  } ^{ 1 - t_1 \alpha /  n }  \\
 	 	& 	\lesssim  \| f\|_{ M_{\vec p_1}^{t_1,r_1} }^{ t_1 \alpha /n  }  	\|   f   \|_{M_{  (1 - t_1 \alpha /  n ) \vec p_2 }^{ (1 - t_1 \alpha /  n) t_2, (1 - t_1 \alpha /  n) r_2	}  } ^{ 1 - t_1 \alpha /  n }  \\
 	 	& =  \| f\|_{ M_{\vec p_1}^{t_1,r_1} }^{ t_1 \alpha /n  }  	\|   f   \|_{M_{ \vec p_1  }^{ t_1 , r_1	}  } ^{ 1 - t_1 \alpha /  n }  =  \| f\|_{ M_{\vec p_1}^{t_1,r_1} } .
 \end{align*}
We finish the proof.
\end{proof}

\subsection{Singular integral operators}
Let us consider the boundedness of singular integral operators.
\begin{definition}
	A singular integral operator is an $L^2$-bounded linear operator $T$ with a kernel $K(x,y)$ which satisfies the following conditions:
	\begin{itemize}
		\item[(i)] There exist  $\epsilon, c>0$  such that $ |K(x,y) | \le c / |x-y|^n $ for $x \neq y$ and 
				\begin{equation*}
			| K (x,y) - K (z,y) | + | K (y,x) -K (y,z) |  \le c \frac{|x-z| ^\epsilon}{|x-y| ^{ \epsilon+n  } }
		\end{equation*}
		if $ |x-y| \ge 2 |x-z| >0$.
		\item [(ii)] If $f\in L_c^\infty  $, then
		\begin{equation*}
			T f(x) = \int_\rn K (x,y) f(y) \d y ,  \quad  x \notin  \operatorname{supp} (f).
		\end{equation*}
	\end{itemize}
\end{definition}

\begin{theorem}
	Let $ 1< \vec p <\infty$. Let   $1 < n / ( \sum_{i=1}^n  1/p_{i})   < t <r<\infty $ or $ 1< n / ( \sum_{i=1}^n  1/p_{i})  \le t< r =\infty  $. Let $T$ be a singular integral operator. Then for all $f \in M_{\vec p}^{t,r}$, we have
	\begin{equation*}
		\| Tf\|_{M_{\vec p}^{t,r} }  \lesssim 	\| f\|_{M_{\vec p}^{t,r} } .
	\end{equation*}
\end{theorem}

\begin{proof}
	The case $r=\infty$ is included in \cite[Theorem 1.12]{N19}. We may assume that $r<\infty$. Since $M_{\vec p}^{t,r} \hookrightarrow M_{\vec p}^{t,\infty}  $,  the singular integral operator $T$, initially defined on $ M_{\vec p}^{t,\infty} $, can be defined on $M_{\vec p}^{t,r} $ by restriction. Thus, the matters are reduced to the norm estimate $ \| Tf\|_{M_{\vec p}^{t,r} }  \lesssim 	\| f\|_{M_{\vec p}^{t,r} } .$
	
	Fix $Q \in \D$. Denote by $Q_k$ the $k^{\operatorname{th}}$  dyadic parent of $Q$. 
	We may consider one  of $2^n$ directions about a fixed $Q \in \D$ by the definition of singular integral operators. 
	Then we decompose $f = f\chi_Q  + f \chi_{Q^c} =: f_1 +f_2$. 
	By the $L^{ \vec p}$-boundedness of the singular integral operator $T $ (see \cite[Theorem 7.1]{N19}), 
	\begin{equation}\label{T est 1}
		 |Q|^{ \frac{1}{t}  - \frac{1}{n} \sum_{i=1 } ^n \frac{1}{p_i}  }  \| (Tf_1) \chi_Q \|_{L^{\vec p}  }  \lesssim  |Q|^{ \frac{1}{t}  - \frac{1}{n} \sum_{i=1 } ^n \frac{1}{p_i}  }  \| f \chi_Q \|_{L^{\vec p}  } .
	\end{equation}
Meanwhile, by H\"older's inequality,  for all $x\in Q$, we have
\begin{align*}
	| T f_2 (x ) | & \lesssim \int_{Q^c} \frac{ | f(y) |}{|x-y|^n} \d y \\
	&\approx \sum_{k=1}^\infty  \int_{Q_k \backslash Q_{k-1}} \frac{ | f(y) |}{|x-y|^n} \d y \\
	& \lesssim  \sum_{k=1}^\infty  \frac{1}{|Q_k| } \int_{Q_k } | f(y) | \d y \\
	& \lesssim \sum_{k=1} |Q_k|^{ - \frac{1}{n}  \sum_{i=1}^n   \frac{1}{p_i} } \| f \|_{ L^{\vec p} (Q_k)}  .
\end{align*}
Then we conclude 
\begin{equation} \label{T est 2}
	 |Q|^{ \frac{1}{t}  - \frac{1}{n} \sum_{i=1 } ^n \frac{1}{p_i}  }  \| (Tf_2) \chi_Q \|_{L^{\vec p}  }  \lesssim \sum_{k=1}  2^{-kn/t}   |Q_k|^{1/t - \frac{1}{n}  \sum_{i=1}^n   \frac{1}{p_i} } \| f \|_{ L^{\vec p} (Q_k)} .
\end{equation}
Combining   estimates (\ref{T est 1}) and (\ref{T est 2})  and Minkowski's inequality, we obtain
\begin{align*}
	&  \left(\sum_{Q \in \D} |Q|^{ \frac{r}{t}  - \frac{r}{n} \sum_{i=1 } ^n \frac{1}{p_i}  }  \| (Tf)\chi_Q \|_{L^{\vec p}  } ^r  \right)^{1/r}  \\
	& \lesssim  \left(\sum_{Q \in \D} |Q|^{ \frac{r}{t}  - \frac{r}{n} \sum_{i=1 } ^n \frac{1}{p_i}  }  \| f \chi_Q \|_{L^{\vec p}  } ^r  \right)^{1/r} +  \sum_{k=1}  2^{-kn/t}  \left(  \sum_{Q\in \D}  |Q_k|^{r/t - \frac{r}{n}  \sum_{i=1}^n   \frac{1}{p_i} } \| f \|_{ L^{\vec p} (Q_k)} ^r  \right) ^{1/r} \\
	& \lesssim \| f\|_{M_{\vec p}^{t,r}   } + \sum_{k=1}  2^{-kn/t}  2^{kn/r } \| f\|_{M_{\vec p}^{t,r}   } \lesssim \| f\|_{M_{\vec p}^{t,r}   }.
\end{align*}
Thus the proof is complete.
\end{proof}

\begin{theorem}
	Let $1 < u <\infty $
	and $ 1< \vec p <\infty$. Let   $1 < n / ( \sum_{i=1}^n  1/p_{i})   < t <r<\infty $ or $ 1< n / ( \sum_{i=1}^n  1/p_{i})  \le t< r =\infty  $. Let $T$ be a singular integral operator. Then for all $\vec f \in M_{\vec p}^{t,r} (\ell^u)$, we have
	\begin{equation*}
		\| T \vec f\|_{M_{\vec p}^{t,r} (\ell^u) }  \lesssim 	\| \vec f\|_{M_{\vec p}^{t,r} (\ell^u) } .
	\end{equation*}
\end{theorem}

\begin{proof}
	The case $ 1< n / ( \sum_{i=1}^n  1/p_{i})  \le t< r =\infty  $ is proved in \cite[Theorem 2.17]{N24}. Thus we only need to show the case $1 < n / ( \sum_{i=1}^n  1/p_{i})   < t <r<\infty $. Since we have the embedding
$ M_{\vec p}^{t,r} (\ell^u)  \hookrightarrow M_{\vec p}^{t,\infty} (\ell^u)  $, we can consider the singular integral operators on mixed Bourgain-Morrey spaces by restriction.

 Let $Q \in \D$. We may consider one  of $2^n$ directions about a fixed $Q \in \D$ by the definition of singular integral operators.  Denote by $Q_m$ the $m^{\operatorname{th}}$ dyadic parent. Observe that for $x\in Q$,
\begin{align*}
	\left(  \sum_{k=1}^\infty (T f_k ) ^u  (x) \right)^{1/u} 
	& \lesssim \left(  \sum_{k=1}^\infty (T [ \chi_Q  f_k ] ) ^u (x) \right)^{1/u} +  \left(  \sum_{k=1}^\infty \left( \sum_{m=1}^\infty  \int_{Q_m}  \frac{|f_k (y) |}{ |x-y|^n}\d y   \right) ^u  \right)^{1/u} \\
	& \lesssim \left(  \sum_{k=1}^\infty (T [ \chi_Q  f_k ] ) ^u  \right)^{1/u} + \sum_{m=1}^\infty \left(  \sum_{k=1}^\infty \left( \frac{1}{|Q_m|} \int_{Q_m} |f_k (y) |\d y   \right) ^u  \right)^{1/u} .
\end{align*}
Consequently,
\begin{align*}
	& |Q|^{ \frac{1}{t}  - \frac{1}{n} \sum_{i=1 } ^n \frac{1}{p_i}  }  \left\| \left( \sum_{k=1}^\infty  |T f_k|^u \right)^{1/u}  \chi_Q \right\|_{L^{\vec p}  } \\
	& \lesssim |Q|^{ \frac{1}{t}  - \frac{1}{n} \sum_{i=1 } ^n \frac{1}{p_i}  }  \left\| \left(  \sum_{k=1}^\infty (T [ \chi_Q  f_k ] ) ^u  \right)^{1/u}  \chi_Q \right\|_{L^{\vec p}  } \\
	& \quad + |Q|^{1/t}  \sum_{m=1}^\infty \left(  \sum_{k=1}^\infty \left( \frac{1}{|Q_m|} \int_{Q_m} |f_k (y) |\d y   \right) ^u  \right)^{1/u} =: I+ II.
\end{align*}
For the first part, by the boundedness of singular integral operator for mixed spaces (for example, see \cite[Theorem 7.1]{N19}), we obtain
\begin{equation*}
	I \lesssim |Q|^{ \frac{1}{t}  - \frac{1}{n} \sum_{i=1 } ^n \frac{1}{p_i}  }  \left\| \left(  \sum_{k=1}^\infty |  f_k | ^u  \right)^{1/u}  \chi_Q \right\|_{L^{\vec p}  }.
\end{equation*}
Then repeating the proof of Theorem \ref{HL seq mix BM}, we obtain the result.
\end{proof}

\begin{theorem}\label{T H block}
		Let $ 1< \vec p <\infty$. Let   $1 < n / ( \sum_{i=1}^n  1/p_{i})   < t <r<\infty $ or $ 1< n / ( \sum_{i=1}^n  1/p_{i})  \le t< r =\infty  $. Let $T$ be a singular integral operator. Then for all $g \in \mathcal{H}_{\vec p \, ^\prime}^{t',r'}$, we have
	\begin{equation*}
		\| Tg\|_{\mathcal{H}_{\vec p \, ^\prime}^{t',r'} }  \lesssim 	\| g\|_{\mathcal{H}_{\vec p \, ^\prime}^{t',r'} } .
	\end{equation*}
\end{theorem}
\begin{proof}
	By Theorem \ref{predual mix BM},
	we obtain 
	\begin{align*}
		\| Tg\|_{\mathcal{H}_{\vec p \, ^\prime}^{t',r'} }  & = \max_{\|f\|_{   M_{\vec p}^{t,r} } \le 1  } \left|\int_\rn T g (x) f (x)  \d x    \right| \\
		& = \max_{\|f\|_{   M_{\vec p}^{t,r} } \le 1  } \left|\int_\rn  g (x) T^* f (x)  \d x    \right| \\
		& \le \max_{\|f\|_{   M_{\vec p}^{t,r} } \le 1  } \| g\|_{\mathcal{H}_{\vec p \, ^\prime}^{t',r'} } \|T^* f \|_{  M_{\vec p}^{t,r} } \\
		& \lesssim \| g\|_{\mathcal{H}_{\vec p \, ^\prime}^{t',r'} }.
	\end{align*}
Thus the proof is complete.
\end{proof}
\begin{theorem}\label{T H vector block}
	Let $1<u ' <\infty$ and $ 1< \vec p <\infty$. Let   $1 < n / ( \sum_{i=1}^n  1/p_{i})   < t <r<\infty $ or $ 1< n / ( \sum_{i=1}^n  1/p_{i})  \le t< r =\infty  $. Let $T$ be a singular integral operator. Then for all $ \vec g=\{ g_j\}_{j =0 }^\infty  \in \mathcal{H}_{\vec p \, ^\prime}^{t',r'} (\ell^ {u'} )$, we have
	\begin{equation*}
		\| T \vec g \|_{\mathcal{H}_{\vec p \, ^\prime}^{t',r'} (\ell^ {u'} ) }  \lesssim 	\|\vec  g\|_{\mathcal{H}_{\vec p \, ^\prime}^{t',r'} (\ell^ {u'} ) } .
	\end{equation*}
\end{theorem}

\begin{proof}
	The proof is similar to  that of Theorem \ref{HL mix block r le infty}, and we omit it here.	
\end{proof}

\section{Littlewood-Paley characterization} \label{sec LP}

In this section, we study the 
Littlewood-Paley characterization for mixed Bourgain-Morrey spaces and their preduals. We first recall the partition of unity on $\rn$.
\begin{definition} \label{homo unity}
	Let $ \psi \in \mathscr S $ satisfy $ \chi_{B(0,2)  } \le \psi \le \chi_{B(0,4)}$. Define $\varphi := \psi -\psi (2\cdot)$. For $j \in \mathbb Z$, let $\varphi_j = \varphi (2^{-j} \cdot)$.
	Then supp $\varphi_j \subset\{ \xi: 2^{j-1} \le |\xi| \le 2^{j+1} \}$ for all $j \in \mathbb Z$ and 
	\begin{equation*}
		\sum_{j \in \mathbb Z} \varphi_j (\xi )  = 1 \quad \operatorname{for} \; \xi \neq 0.
	\end{equation*}
\end{definition}

We have proved  the dual of $ \mathcal{H}_{\vec p \, ^\prime}^{t',r'} $ is $M_{\vec p}^{t,r}$ in Theorem \ref{predual mix BM}. By the functional analysis, a sequence $\{g_j\}_{j\in \mathbb N} \subset \mathcal{H}_{\vec p \, ^\prime}^{t',r'}$ weak converges to $g \in \mathcal{H}_{\vec p \, ^\prime}^{t',r'}$ if for all $f \in M_{\vec p}^{t,r}$, 
\begin{equation*}
	\lim_{j \to \infty} \int_\rn f (x) (g_j (x) - g(x))  \d x = 0 .
\end{equation*}
A sequence $ \{ f_k \}_{k\in \mathbb N} \subset M_{\vec p}^{t,r}$ weak-$*$ converges to $f \in M_{\vec p}^{t,r}$ if for all $g \in \mathcal{H}_{\vec p \, ^\prime}^{t',r'}$, we have
\begin{equation*}
	\lim_{k \to \infty} \int_\rn (f_k (x) - f(x)) h (x) \d x = 0 .
\end{equation*}

The main result of this section is the following.
\begin{theorem} \label{LP char mixed BM}
		Let $ 1 < \vec p <\infty$. Let   $1 < n / ( \sum_{i=1}^n  1/p_{i})   < t <r<\infty $ or $ 1< n / ( \sum_{i=1}^n  1/p_{i})  \le t< r =\infty  $. Let $ \varphi $ be the same as in Definition \ref{homo unity}.
		
		{\rm (i)} For all $f \in M_{\vec p}^{t,r}$, we have $f = \sum_{j =-\infty}^\infty  \F^{-1} (\varphi_j \F f)$ in the weak-$*$ topology of $M_{\vec p}^{t,r} $ and 
		\begin{equation*}
			\left\| \left(  \sum_{j\in \mathbb Z} | \F^{-1} (\varphi_j \F f) |^2  \right) ^{1/2}   \right\|_{M_{\vec p}^{t,r}} \approx	\| f  \|_{M_{\vec p}^{t,r}} .
		\end{equation*}
	
	{\rm (ii)} For all $g \in \mathcal{H}_{\vec p \, ^\prime}^{t',r'}$, we have $g = \sum_{j =-\infty}^\infty  \F^{-1} (\varphi_j \F g)$ in the weak topology of $\mathcal{H}_{\vec p \, ^\prime}^{t',r'}$ and 
		\begin{equation*}
		\left\| \left(  \sum_{j\in \mathbb Z} | \F^{-1} (\varphi_j \F g) |^2  \right) ^{1/2}   \right\|_{\mathcal{H}_{\vec p \, ^\prime}^{t',r'}} \approx	\| g  \|_{\mathcal{H}_{\vec p \, ^\prime}^{t',r'}} .
	\end{equation*}

{\rm (iii)} Suppose that $f \in \mathscr S'  $ satisfies 	
\begin{equation*}
	\left\|  \left(  \sum_{j\in \mathbb Z} | \F^{-1} (\varphi_j \F f) |^2  \right) ^{1/2}    \right\|_{ M_{\vec p}^{t,r}} <\infty.
\end{equation*}
Then the limit $F = \lim_{J \to \infty} \sum_{j =-J} ^J  \F^{-1} (\varphi_j \F f)  $ exists in the weak-$*$ topology of $M_{\vec p}^{t,r}$ and we have
\begin{equation*} 
	\|F\|_{ M_{\vec p}^{t,r}} \approx \left\|  \left(  \sum_{j\in \mathbb Z} | \F^{-1} (\varphi_j \F f) |^2  \right) ^{1/2}    \right\|_{ M_{\vec p}^{t,r}} .
\end{equation*}

{\rm (iv)} 	Assume that $g \in \mathscr S ' $ satisfies 
\begin{equation*} 
	\left\|  \left(  \sum_{j\in \mathbb Z} | \F^{-1} (\varphi_j \F g) |^2  \right) ^{1/2}    \right\|_{ \mathcal{H}_{\vec p \, ^\prime}^{t',r'} } <\infty.
\end{equation*}
Then the limit $G = \lim_{J \to \infty} \sum_{j =-J} ^J  \F^{-1} (\varphi_j \F g)  $ exists in the weak topology of $\mathcal{H}_{\vec p \, ^\prime}^{t',r'}$ and we have
\begin{equation*} 
	\|G  \|_{ \mathcal{H}_{\vec p \, ^\prime}^{t',r'}} \approx \left\|  \left(  \sum_{j\in \mathbb Z} | \F^{-1} (\varphi_j \F g) |^2  \right) ^{1/2}    \right\|_{ \mathcal{H}_{\vec p \, ^\prime}^{t',r'} } .
\end{equation*}
\end{theorem}

Before proving Theorem \ref{LP char mixed BM}, we recall some lemmas.

We first recall the Rademacher sequences. For $t\in [0,1]$, the  Rademacher  functions are defined by $\tilde r_k (t) = \sgn ( \sin ( 2^k t \pi ) ) $  for $k \in \mathbb N_0$. Here, we rearrange $ \{ \tilde r_k \} _{ k \in \mathbb N_0}$  into  $ \{  r_k \} _{ k \in \mathbb Z}$. The following property of the  Rademacher sequences is of fundamental importance and with far-reaching consequences in analysis.

\begin{lemma}[C.2, \cite{G14}] \label{lem rade}
	Let $0<p<\infty $. Then for any $\ell^2 $-sequences $\{ \alpha_j \}_{j\in \mathbb Z} \subset \mathbb C$, we have
	\begin{equation*}
		\left(\int_0^1  \left|  \sum_{j\in\mathbb Z} \alpha_j r_j (t) \right|^p \d t  \right) ^{1/p} \approx \left(  \sum_{j\in \mathbb Z}  |\alpha_j|^2  \right)^{1/2}.
	\end{equation*}
\end{lemma}

\begin{lemma} [Lemma 2.8, \cite{IST15}] \label{T epsilo AS CZ}
	Let $\tau \in \mathscr S$ be a function supported away from the origin. Set $\tau_j = \tau (2^{-j} \cdot)$ for $j \in \mathbb Z$. Suppose that $\epsilon = \{\epsilon_j \}_{j \in \mathbb Z}$ is an $\ell^1$ sequence taking its value in $\{ -1,0,1\}$. Define 
	\begin{equation*}
		K_\epsilon := \sum_{j \in \mathbb Z } \epsilon_j \F^{-1} (\tau_j), \; T_\epsilon f := \sum_{j \in \mathbb Z }  \epsilon_j \F^{-1} (\tau_j \F f)=  \sum_{j \in \mathbb Z }  \epsilon_j \F^{-1} (\tau_j) *f
	\end{equation*}
for $f\in L^2. $  Then the following holds:

{\rm (i)} For all $x \in \rn$, 
\begin{equation} \label{size K_epsilon}
	|K_\epsilon (x)| \le \sum_{j \in \mathbb Z } | \F^{-1} (\tau_j) (x)| \le C |x|^{-n}.
\end{equation}

{\rm (ii)} For all $x \in \rn$ and $k = 1,2, \ldots, n$,
\begin{equation}\label{smooth K_epsilon}
	| \partial_k K_\epsilon (x)| \le  \sum_{j \in \mathbb Z } |\partial_k \F^{-1} (\tau_j) (x)| \le C |x|^{-n-1}.
\end{equation} 
Here, in (\ref{size K_epsilon}) and (\ref{smooth K_epsilon}), the constant $C$ does not depend on $\epsilon$.

{\rm (iii)} The limit defining $ T_\epsilon f$ converges in the topology of $L^2$.

{\rm (iv)} Let $f \in L^2 $  have a compact  support. If $x \notin \operatorname{supp} f$, then
\begin{equation*}
	T_\epsilon f (x) = \int_\rn K_\epsilon (x-y) f(y) \d y.
\end{equation*}
\end{lemma}

\begin{example} \label{exm (1+|x|) M in block}
	Let 
	\begin{equation*}
		f (x) := \frac{1}{ 1+ \max (|x_1| , \ldots , |x_n|)^M}.
	\end{equation*}
If $M >n /t' $, then $f \in \mathcal{H}_{\vec p \, ^\prime}^{t',r'} $. For $j \in \mathbb N$, let $Q_j $ be the cube centered at origin with side length $2^{j+1} $.
Since $f$ is symmetrical, we may consider one quadrant  $\operatorname{Qua} := \{ x\in \rn : x_1 \ge0, x_2 \ge 0, \ldots, x_n \ge 0 \} $. Let 
\begin{equation*}
	b_0 (x)= \chi_{Q_1 \cap \operatorname{Qua} } (x) f(x)
\end{equation*}
and
\begin{equation*}
	b_j (x) = 2^{ j  M } 2^{- jn/t' }  \chi_{ ( Q_j \backslash Q_{j-1})  \cap \operatorname{Qua} } f (x) .
\end{equation*}
Then
\begin{align*}
	\| b_0 \|_{ L^ \infty  }\le 1,
\end{align*}
and
\begin{equation*}
	\| b_j \|_{ L^ {\infty}  }  \le \frac{ 2^{j M } 2^{ -jn/t' }   }{1+2^{ j  M } }  <  2^{-jn/t'}.
\end{equation*}
Hence
\begin{equation*}
	\| b_0 \|_{ L^ {\vec p}  }\le  1 =  | Q_1 \cap \operatorname{Qua}| ^{ 1/t- \frac{1}{n} \sum_{i=1}^{n} 1/p_i },
\end{equation*}
and 
\begin{equation*}
	\| b_j \|_{ L^ {\vec p \, ^\prime}  }\le 2^{-jn/t'  }   | Q_j \cap \operatorname{Qua}  |^{ \frac{1}{n} \sum_{i=1}^{n} 1/p_i' }  =  | Q_j \cap \operatorname{Qua}  |^{ 1/t- \frac{1}{n} \sum_{i=1}^{n} 1/p_i } .
\end{equation*}
This means $b_j$ is a $( \vec p \, ^\prime, t')$-block supported on dyadic cube $Q_j \cap \operatorname{Qua}$.
Since 
\begin{equation*}
	f \chi_{\operatorname{Qua}} = b_0 + \sum_{ j=1 } ^\infty \frac{2^{jn/t' } }{2^{j  M }  } b_j
\end{equation*} 
if we provide $ M > n /t'  $, 
\begin{align*}
	\|f\|_{\mathcal{H}_{\vec p \, ^\prime}^{t',r'}  } \le \left( 1 +   \sum_{j=1}^\infty  \frac{   2^{jn r' /t' }   }{2^{j  M  r' }}    \right)^{1/r'}  <\infty .
\end{align*}

\end{example}

{\bf Proof of Theorem \ref{LP char mixed BM}.}

First, we show that for $f \in M_{\vec p}^{t,r}$ (or $\mathcal{H}_{\vec p \, ^\prime}^{t',r'}$), its square function is bounded on  $M_{\vec p}^{t,r}$ (or $\mathcal{H}_{\vec p \, ^\prime}^{t',r'}$).

\begin{lemma} \label{ell 2 le BM and block}
		Let $ 1 < \vec p <\infty$. Let   $1 < n / ( \sum_{i=1}^n  1/p_{i})   < t <r<\infty $ or $ 1< n / ( \sum_{i=1}^n  1/p_{i})  \le t< r =\infty  $. 
		Let $\tau \in \mathscr S $ be a function supported away from the origin. For $j \in \mathbb Z$, define $\tau_j = \tau (2^{-j } \cdot)$. 
		
		{\rm (i)} For $f \in M_{\vec p}^{t,r}$, we have
		\begin{equation*}
			\left\| \left(  \sum_{j\in \mathbb Z} | \F^{-1} (\tau_j \F f) |^2  \right) ^{1/2}   \right\|_{M_{\vec p}^{t,r}} \lesssim 	\| f  \|_{M_{\vec p}^{t,r}} .
		\end{equation*}
	
		{\rm (ii)} For $g \in\mathcal{H}_{\vec p \, ^\prime}^{t',r'}$, we have
	\begin{equation*}
		\left\| \left(  \sum_{j\in \mathbb Z} | \F^{-1} (\tau_j \F g) |^2  \right) ^{1/2}   \right\|_{\mathcal{H}_{\vec p \, ^\prime}^{t',r'} } \lesssim 	\| g  \|_{\mathcal{H}_{\vec p \, ^\prime}^{t',r'} } .
	\end{equation*}
\end{lemma}
\begin{proof}
	(i)
	Let $\{r_j  \}_{j\in \mathbb Z}$ be the Rademacher sequence. Fix $J\in \mathbb N$. 
	Then by Lemma \ref{lem rade}, we have
	\begin{equation*}
		\left\| \left(  \sum_{j =-J} ^J  | \F^{-1} (\tau_j \F f) |^2  \right) ^{1/2}   \right\|_{M_{\vec p}^{t,r}} \lesssim 
			\left\| \int_0^1 \left|  \sum_{j =-J} ^J   \F^{-1} (\tau_j \F f) r_j (s)   \right| \d s \right\|_{M_{\vec p}^{t,r}} .
	\end{equation*}
	Note that $\sum_{j =-J} ^J   \F^{-1} (\tau_j \F f) r_j (s) : (0,1)  \mapsto M_{\vec p}^{t,r} $ is measurable since $r_j $ assumes its value in $\{ -1,0,1 \}$ and $r_j$  is constant almost everywhere on each dyadic interval $I_j = [ 2^ {j-1 - 2^{2J +1}} , 2^ {j - 2^{2J +1}}  ) $ for $j = 1, \ldots, 2^{2J +1}$, and 
	\begin{equation*}
		\left\|    \sum_{j =-J} ^J   \F^{-1} (\tau_j \F f) r_j (s)    \right\|_{M_{\vec p}^{t,r}} 
	\end{equation*}
is integrable on $(0,1)$. Thus, by Bochner's Theorem (for example, see \cite[Theorem 1.1.4]{ABHN11}),
\begin{equation*}
	 \sum_{j =-J} ^J   \F^{-1} (\tau_j \F f) r_j (s) : (0,1) \mapsto M_{\vec p}^{t,r}
\end{equation*}
is Bochner integrable and we have
\begin{equation*}
	\left\| \left(  \sum_{j =-J} ^J  | \F^{-1} (\tau_j \F f) |^2  \right) ^{1/2}   \right\|_{M_{\vec p}^{t,r}} \lesssim   \int_0^1 \left\|  \sum_{j =-J} ^J   \F^{-1} (\tau_j \F f) r_j (s)    \right\|_{M_{\vec p}^{t,r}}  \d s .
\end{equation*}
Then by Theorem \ref{predual mix BM}, we can find the dual function $G_k \in \mathcal{H}_{\vec p \, ^\prime}^{t',r'}$ corresponding to $\sum_{j =-J} ^J   \F^{-1} (\tau_j \F f) r_j (s)$ on $I_k$ for $k =  1, \ldots, 2^{ 2J +1 } $ such that $ \| G_k \|_{ \mathcal{H}_{\vec p \, ^\prime}^{t',r'}  }  \le 1 $. Define 
\begin{equation*}
	g_s^J := \sum_{ k =1 }^ { 2^{ 2J+1  }  } G_k \chi_{ I_k } (s).
\end{equation*}
Then $g_s^J$  is measurable for $s$ and satisfies $ \|g_s^J \|_{ \mathcal{H}_{\vec p \, ^\prime}^{t',r'}  }  \le 1 $.
	Moreover, we have
	\begin{align*}
		\left\| \left(  \sum_{j =-J} ^J  | \F^{-1} (\tau_j \F f) |^2  \right) ^{1/2}   \right\|_{M_{\vec p}^{t,r}} & \lesssim  \int_0^1 \left| \int_\rn   \sum_{j =-J} ^J   \F^{-1} (\tau_j \F f) r_j (s)  g_s^J (x) \d x  \right|   \d s \\
		& = \int_0^1 \left| \int_\rn   f(x) \sum_{j =-J} ^J   \F^{-1} (\tilde \tau_j \F g_s^J) (x) r_j (s)  \d x  \right|   \d s \\
		& \le \int_0^1
		\| f\|_{ M_{\vec p}^{t,r}}  \left\| \sum_{j =-J} ^J   \F^{-1} (\tilde \tau_j \F g_s^J) r_j (s)  \right\|_{ \mathcal{H}_{\vec p \, ^\prime}^{t',r'}   }
		    \d s .
	\end{align*}
Here for $w \in L^0 $, define  $ \tilde{w} = w (-\cdot )$. Define $\alpha_j (s) : = r_j (s) \chi_{ [J,J]} ( j)$ for $j \in \mathbb Z$ and $ s \in (0,1)$ and 
\begin{equation*}
	T_s g : = \sum_{j \in \mathbb Z} \F^{-1} (\tilde \tau_j \F g ) \alpha_j (s).
\end{equation*}

We can regard the operator $T_s $ as a singular integral operator in the view of Lemma \ref{T epsilo AS CZ}. Thus by Theorem \ref{T H block}, we have
\begin{equation*}
		\left\| \left(  \sum_{j =-J} ^J  | \F^{-1} (\tau_j \F f) |^2  \right) ^{1/2}   \right\|_{M_{\vec p}^{t,r}}  \lesssim \int_0^1
		\| f\|_{ M_{\vec p}^{t,r}}  \left\| T_s (g_s^J ) \right\|_{ \mathcal{H}_{\vec p \, ^\prime}^{t',r'}   }
		\d s \lesssim \| f\|_{ M_{\vec p}^{t,r}} 
\end{equation*}
where the implicit constants in the above inequalities are independent of $s$ and $J$. Finally, by using the dominated convergence theorem, we prove (i).

(ii) The same argument works for the predual space $ \mathcal{H}_{\vec p \, ^\prime}^{t',r'}  $. We omit the proof.
\end{proof}

The next result is  the convergence in the predual spaces.

\begin{lemma} \label{LP block}
		Let $ 1 < \vec p <\infty$. Let   $1 < n / ( \sum_{i=1}^n  1/p_{i})   < t <r<\infty $ or $ 1< n / ( \sum_{i=1}^n  1/p_{i})  \le t< r =\infty  $. 
			Let $\varphi$ be the same as in Definition \ref{homo unity}.  If $g \in \mathcal{H}_{\vec p \, ^\prime}^{t',r'} $, then we have
	\begin{equation} \label{topo L-P block}
		\lim_{J \to \infty} \left\| g - \sum_{j =-J} ^J  \F^{-1} (\varphi_j \F g)  \right\|_{\mathcal{H}_{\vec p \, ^\prime}^{t',r'} } =0 .
	\end{equation}
\end{lemma}
\begin{proof}
	Let $g \in \mathcal{H}_{\vec p \, ^\prime}^{t',r'}$. Then $g = \sum_{(j,k)\in\mathbb{Z}^{n+1}}\lambda_{j,k}b_{j,k}$ where $\lambda=\{\lambda_{j,k}\}_{(j,k)\in\mathbb{Z}^{n+1}}\in\ell^{r'}(\mathbb{Z}^{n+1})$
	and $b_{j,k}$ is a  $(\vec p \, ^\prime,t')$-block supported on $Q_{j,k}$ such that 
	\begin{equation*}
		\left( \sum_{(j,k)\in\mathbb{Z}^{n+1}} |\lambda_{j,k}|^{r'}\right)^{1/r'}  \le 2 \| g\|_{ \mathcal{H}_{\vec p \, ^\prime}^{t',r'} }.
	\end{equation*}
For $K,J \in \mathbb N$, define
\begin{equation*}
	g_K := \sum_{(j,k)\in\mathbb{Z}^{n+1},|(j,k)|_\infty \le K  }\lambda_{j,k}b_{j,k}
\end{equation*}
and 
\begin{equation*}
	S_J g := \sum_{j =-J} ^J  \F^{-1} (\varphi_j \F g) .
\end{equation*}
Since $1 \le r' <\infty $, we choose $K$ sufficiently large such that
\begin{equation*}
\| g - g_K \|_{\mathcal{H}_{\vec p \, ^\prime}^{t',r'} } \le \epsilon.
\end{equation*}
Since $S_J$ is a  singular integral operator by Lemma \ref{T epsilo AS CZ}, there exists a constant $C>0$ such that
\begin{equation*}
	\|  S_J  ( g - g_K) \|_{\mathcal{H}_{\vec p \, ^\prime}^{t',r'} }  \le C \| g - g_K \|_{\mathcal{H}_{\vec p \, ^\prime}^{t',r'} } .
\end{equation*}
Thus we have
\begin{align*}
	\| g - S_J g\|_{\mathcal{H}_{\vec p \, ^\prime}^{t',r'} } & \le \| g - g_K \|_{\mathcal{H}_{\vec p \, ^\prime}^{t',r'} }  +\| g_K - S_J g_K \|_{\mathcal{H}_{\vec p \, ^\prime}^{t',r'} } +\|  S_J g_K - S_J g\|_{\mathcal{H}_{\vec p \, ^\prime}^{t',r'} } \\
	& \le \| g_K - S_J g_K \|_{\mathcal{H}_{\vec p \, ^\prime}^{t',r'} } + (C+1) \epsilon \\
	& \le \sum_{(j,k)\in\mathbb{Z}^{n+1},|(j,k)|_\infty \le K  } | \lambda_{j,k} |  \|S_J b_{j,k} - b_{j,k}  \|_{ \mathcal{H}_{\vec p \, ^\prime}^{t',r'} }  + (C+1) \epsilon .
\end{align*}
Next we estimate $\|S_J b_{j,k} - b_{j,k}  \|_{ \mathcal{H}_{\vec p \, ^\prime}^{t',r'} }$. 
Let $\epsilon' \in (0,1) $ sufficiently close to $0$, say $ 0< \epsilon' < n/t$.
Then from \cite[proof of Lemma 3.1]{IST15}, we have
\begin{align*}
&	| S_J b_{j,k} (x) - b_{j,k} (x) |   \lesssim  | S_J b_{j,k} (x) - b_{j,k} (x) |  \chi_{ 3 Q_{j,k}} (x)\\
	& + \left(  \frac{1}{  2^J( 2^{-j} + |x - c ( Q_{j,k} ) |  ) ^{n+1}}  +  \frac{1}{  2^{ J \epsilon' } ( 2^{-j} + |x - c ( Q_{j,k} ) |  ) ^{n -\epsilon' }} \right) \|  b_{j,k}\|_{L^1} .
\end{align*}
For the first term, we can invoke the Littlewood-Paley theory for $L^ {\vec p} $. Since $ (a+|x|) ^{-n+ \epsilon'  }  \in \mathcal{H}_{\vec p \, ^\prime}^{t',r'} $ for all $a>0$ by Example \ref{exm (1+|x|) M in block}, we have
\begin{equation*}
	\lim_{J\to \infty}  \| S_J b_{j,k} - b_{j,k}\|_{\mathcal{H}_{\vec p \, ^\prime}^{t',r'}  } =0
\end{equation*}
for fixed $(j,k) \in \{  (j,k)\in \mathbb Z^{1+n} : |(j,k)|_\infty \le K \}$. Hence
\begin{equation*}
	\limsup_{J \to \infty} \| g - S_J g\|_{\mathcal{H}_{\vec p \, ^\prime}^{t',r'} } \le (C+2) \epsilon .
\end{equation*}
Since $ \epsilon > 0 $ being arbitrary, we have (\ref{topo L-P block}).
\end{proof}

Using these lemmas, we show Theorem \ref{LP char mixed BM} (i) and (ii).

\begin{lemma} \label{lem f le LP}
		Let $ 1 < \vec p <\infty$. Let   $1 < n / ( \sum_{i=1}^n  1/p_{i})   < t <r<\infty $ or $ 1< n / ( \sum_{i=1}^n  1/p_{i})  \le t< r =\infty  $. 
		Let $\varphi$ be the same as in Definition \ref{homo unity}.

	{\rm (i)}
		Assume that $f\in M_{\vec p}^{t,r} $ satisfies
		\begin{equation} \label{con sum 2}
			\left\|  \left(  \sum_{j\in \mathbb Z} | \F^{-1} (\varphi_j \F f) |^2  \right) ^{1/2}    \right\|_{ M_{\vec p}^{t,r}} <\infty.
		\end{equation}
		Then the limit $f = \lim_{J \to \infty} \sum_{j=-J } ^J \F^{-1} (\varphi_j \F f) $ exists in the weak-$*$ topology in $ M_{\vec p}^{t,r}$ and we have
	the estimate
	\begin{equation} \label{f le LP}
		\|f\|_{ M_{\vec p}^{t,r}} \lesssim \left\|  \left(  \sum_{j\in \mathbb Z} | \F^{-1} (\varphi_j \F f) |^2  \right) ^{1/2}    \right\|_{ M_{\vec p}^{t,r}} .
	\end{equation}

{\rm (ii)} 
	Assume that $g \in \mathcal{H}_{\vec p \, ^\prime}^{t',r'} $ satisfies
\begin{equation*}
	\left\|  \left(  \sum_{j\in \mathbb Z} | \F^{-1} (\varphi_j \F g) |^2  \right) ^{1/2}    \right\|_{ \mathcal{H}_{\vec p \, ^\prime}^{t',r'} } <\infty.
\end{equation*}
Then the limit $g = \lim_{J \to \infty} \sum_{j=-J } ^J \F^{-1} (\varphi_j \F g) $ exists in the weak topology in $ \mathcal{H}_{\vec p \, ^\prime}^{t',r'} $ and we have
\begin{equation*} 
	\|g\|_{ \mathcal{H}_{\vec p \, ^\prime}^{t',r'}} \lesssim \left\|  \left(  \sum_{j\in \mathbb Z} | \F^{-1} (\varphi_j \F g) |^2  \right) ^{1/2}    \right\|_{ \mathcal{H}_{\vec p \, ^\prime}^{t',r'}} .
\end{equation*}
\end{lemma}
\begin{proof}
	(i)	We first prove the existence of the limit. For all $g \in \mathcal{H}_{\vec p \, ^\prime}^{t',r'} $, by the duality
	\begin{align*}
		\left| \langle f,g\rangle - \left\langle   \sum_{j=-J } ^J \F^{-1} (\varphi_j \F f) ,g \right\rangle    \right| & =	\left| \left\langle f, g -\sum_{j=-J } ^J \F^{-1} (\tilde \varphi_j \F g)  \right\rangle     \right| \\
		& \le \| f\|_{  M_{\vec p}^{t,r}} \left\|g -\sum_{j=-J } ^J \F^{-1} (\tilde \varphi_j \F g)    \right\|_{ \mathcal{H}_{\vec p \, ^\prime}^{t',r'} }.
	\end{align*}
By Lemma \ref{LP block}, the limit $f =\sum_{j=-\infty } ^\infty \F^{-1} (\varphi_j \F f) $ exists in the weak-$*$ topology in $ M_{\vec p}^{t,r}$. 

Next we show the estimate (\ref{f le LP}). 
Put $f_J = \sum_{j=-J } ^J \F^{-1} (\varphi_j \F f)$ for $J \in \mathbb N$. Since $f_J$ converges to  $f$ in the weak-$*$ topology of $ M_{\vec p}^{t,r}$, it suffices to show 
\begin{equation} \label{f_J lesssim LP}
	\|f_J \|_{M_{\vec p}^{t,r} } \lesssim  \left\|\left(  \sum_{j\in \mathbb Z} | \F^{-1} (\varphi_j \F f) |^2  \right) ^{1/2}    \right\|_{ M_{\vec p}^{t,r}}.
\end{equation}
For a fixed $J$, we obtain $f_J \in M_{\vec p}^{t,r} $ by the condition (\ref{con sum 2}). Indeed, since $\varphi \in \mathscr S $, by \cite[Theorem 2.1.10]{G14}, we have
\begin{equation*}
	| \F^{-1} (\varphi_j \F f_J) | = \left|\F^{-1} \left( \sum_{k=j-1}^{j+1} \varphi_k  \varphi_j \F f \right)  \right| \lesssim \M (\F^{-1} (\varphi_j \F f) ),
\end{equation*}
where the implicit constant is independent of $J$. Using  Theorem \ref{HL seq mix BM} and condition (\ref{con sum 2}), we have
\begin{equation*}
	\|  \F^{-1} (\varphi_j \F f_J) \|_{  M_{\vec p}^{t,r} } \lesssim 	\|  \M (\F^{-1} (\varphi_j \F f) )  \|_{  M_{\vec p}^{t,r} }  \lesssim 	\| \F^{-1} (\varphi_j \F f)   \|_{  M_{\vec p}^{t,r} }  <\infty .
\end{equation*}
Note that $\F^{-1} (\varphi_j \F f_J) =0 $ if $|j| >J+1$  and $0 \notin \operatorname{supp} ( \F f_J ) $. Hence we get
\begin{equation*}
	f_J = \sum_{j=-\infty } ^\infty \F^{-1} (\varphi_j \F f_J)  = \sum_{j=-J-1 } ^{J+1}\F^{-1} (\varphi_j \F f_J)  \in  M_{\vec p}^{t,r} .
\end{equation*}
Then taking $g \in \mathcal{H}_{\vec p \, ^\prime}^{t',r'} $ with  $\| g\|_{ \mathcal{H}_{\vec p \, ^\prime}^{t',r'} }  \le 1$ arbitrary, we obtain
\begin{align*}
	\int_\rn f_J (x) g(x) \d x & =  \lim_{K \to \infty} \sum_{k=-K}^K \langle f_J , \F^{-1} (\tilde \varphi_k \F g)  \rangle \\
	& = \lim_{K \to \infty} \sum_{k=-K}^K \left\langle \F^{-1} \left(  \varphi_k  \sum_{j=-J } ^J \varphi_j \F f \right) , g  \right\rangle
\end{align*}
by Lemma \ref{LP block}. Note that $ \varphi_k \varphi_j =0$ if $|j-k| >1 $.  Hence
\begin{align*}
	\int_\rn f_J (x) g(x) \d x & = \sum_{k=-J-1}^{J+1} \left\langle \F^{-1} \left(  \varphi_k (\varphi_{k-1} + \varphi_k+\varphi_{k+1} )\F f \right) , g  \right\rangle \\
	& =  \sum_{k=-J-1}^{J+1} \left\langle \F^{-1} \left(  (\varphi_{k-1} + \varphi_k+\varphi_{k+1} )\F f \right) , \F^{-1} (\tilde \varphi_k \F g )  \right\rangle \\
	& = \int_\rn \sum_{k=-J-1}^{J+1}  \F^{-1} \left(  (\varphi_{k-1} + \varphi_k+\varphi_{k+1} )\F f \right) (x) \F^{-1} (\tilde \varphi_k \F g ) (x)  \d x .
\end{align*}
Here, for $h \in L^0$, denote $\tilde h = h (-\cdot)$ again.  Using the Cauchy-Schwartz inequality and the duality (Theorem \ref{predual mix BM}), we have
\begin{align*}
	& \left| \int_\rn f_J (x) g(x) \d x \right| \\
	&\lesssim  \int_\rn \left( \sum_{k=-J-2}^{J+2}  | \F^{-1} \left(   \varphi_k \F f \right) (x)|^2 \right)^{1/2}   \left( \sum_{k=-J-1}^{J+1}  |\F^{-1} (\tilde \varphi_k \F g ) (x) |^2 \right)^{1/2}    \d x \\
	& \le \left\| \left( \sum_{k=-J-2}^{J+2}  | \F^{-1} \left(   \varphi_k \F f \right) |^2 \right)^{1/2}   \right\| _{ M_{\vec p}^{t,r}   }  \left\| \left( \sum_{k=-J-1}^{J+1}  |\F^{-1} (\tilde \varphi_k \F g ) (x) |^2 \right)^{1/2} \right\|_{ \mathcal{H}_{\vec p \, ^\prime}^{t',r'}  }.
\end{align*}
By Lemma \ref{ell 2 le BM and block}  and $\| g\|_{ \mathcal{H}_{\vec p \, ^\prime}^{t',r'} }  \le 1$, we have
\begin{align*}
	\left| \int_\rn f_J (x) g(x) \d x \right| 
 \lesssim \left\| \left( \sum_{k=-\infty}^{ \infty }  | \F^{-1} \left(   \varphi_k \F f \right) |^2 \right)^{1/2}   \right\| _{ M_{\vec p}^{t,r}   }  .
\end{align*}
Then by Theorem \ref{predual mix BM}, we obtain
\begin{equation*}
	\| f_J \|_{ M_{\vec p}^{t,r} }
 \lesssim \left\| \left( \sum_{k=-\infty}^{\infty}  | \F^{-1} \left(   \varphi_k \F f \right) |^2 \right)^{1/2}   \right\| _{ M_{\vec p}^{t,r}   }  .
\end{equation*}
(ii) The similar arguments in (i) work in the predual space $\mathcal{H}_{\vec p \, ^\prime}^{t',r'}$ and we omit it here.
\end{proof}

Next we show the Littlewood-Paley characterization for the elements of $\mathscr S ' $. The following lemma corresponds to Theorem \ref{LP char mixed BM} (iii).
\begin{lemma} \label{lem LP char S'}
		Let $ 1 < \vec p <\infty$. Let   $1 < n / ( \sum_{i=1}^n  1/p_{i})   < t <r<\infty $ or $ 1< n / ( \sum_{i=1}^n  1/p_{i})  \le t< r =\infty  $. 
		Let  $\varphi$  be the same as in Definition \ref{homo unity}.
	Assume that $f \in \mathscr S ' $ satisfies 
		\begin{equation*}
		\left\|  \left(  \sum_{j\in \mathbb Z} | \F^{-1} (\varphi_j \F f) |^2  \right) ^{1/2}    \right\|_{ M_{\vec p}^{t,r}} <\infty.
	\end{equation*}
Then the limit $F = \lim_{J \to \infty} \sum_{j =-J} ^J  \F^{-1} (\varphi_j \F f)  $ exists in the weak-$*$ topology of $M_{\vec p}^{t,r}$ and we have
\begin{equation} \label{eq char BM LP}
	\|F\|_{ M_{\vec p}^{t,r}} \approx \left\|  \left(  \sum_{j\in \mathbb Z} | \F^{-1} (\varphi_j \F f) |^2  \right) ^{1/2}    \right\|_{ M_{\vec p}^{t,r}} .
\end{equation}
\end{lemma}

\begin{proof}
	Let $f_J = \sum_{j =-J} ^J  \F^{-1} (\varphi_j \F f)$ for $J \in \mathbb N$. We obtain $f_J \in  M_{\vec p}^{t,r}$ for fixed $J$ in the same way as in Lemma \ref{lem f le LP}. Recall that $  (\mathcal{H}_{\vec p \, ^\prime}^{t',r'} )^* =  M_{\vec p}^{t,r}$ and $  \mathcal{H}_{\vec p \, ^\prime}^{t',r'} $ is separable. Hence, by the Banach-Alaoglu theorem (for example, see \cite[Theorem 89]{SDH20}), we can find a subsequence $ \{f_{J(K)}\}_{K \in \mathbb N} \subset \{f_J\}_{J \in \mathbb N} $  which converges in the weak-$*$ topology of $M_{\vec p}^{t,r}$ to a function $F \in M_{\vec p}^{t,r}$. That is, 
	for all $g \in \mathcal{H}_{\vec p \, ^\prime}^{t',r'}$,
	\begin{equation*}
		\lim_{K \to \infty} \int_\rn f_{J(K)}  (x) g(x) \d x = \int_\rn F (x) g(x) \d x .
	\end{equation*}
Next we show $f_J$ converges to $F$ in the weak-$*$ topology of $M_{\vec p}^{t,r}$. Since $F \in M_{\vec p}^{t,r}$, by Lemma \ref{lem f le LP}, for  $\epsilon >0$, there exists $J \in\mathbb N$ such that
\begin{equation*}
	\left|   \int_\rn \left(F (x) -\sum_{j =-J} ^J  \F^{-1} (\varphi_j \F F)   \right)  g(x) \d x \right| <\epsilon.
\end{equation*}
Fix $J \in \mathbb N$  as above. Then
\begin{equation*}
	\sum_{j =-J} ^J  \F^{-1} (\varphi_j \F F) = \sum_{j =-J} ^J  \F^{-1} (\varphi_j) *  F = \sum_{j =-J} ^J \langle F,  \F^{-1} (\varphi_j) (x- \cdot) \rangle .
\end{equation*}
Since $f_{J(K)} \to F$ in the weak-$*$ topology of $M_{\vec p}^{t,r}$, there exists a number $K=K_J$ depending on $J$ such that 
\begin{equation*}
		\left|   
		\left\langle  F-  f_{J(K)} , \sum_{j =-J} ^J \F^{-1} (\varphi_j)(x- \cdot)   \right\rangle \right|
	 <\epsilon.
\end{equation*} 
So we calculate	
\begin{equation*}
	   	\left\langle   f_{J(K)} , \sum_{j =-J} ^J \F^{-1} (\varphi_j)(x- \cdot)   \right\rangle  .
\end{equation*}
By the definition of $f_{J(K)} $ and the translation property of the Fourier transform, we have
\begin{align*}
		\left\langle  f_{J(K)} , \sum_{j =-J} ^J \F^{-1} (\varphi_j)(x- \cdot)   \right\rangle  & = 	\left\langle  
		\sum_{k = -J(K)} ^{J(K)}  \F^{-1} (\varphi_k \F f),  \sum_{j =-J} ^J  \F ( e^{2\pi x \cdot} \varphi_j )
	   \right\rangle  \\
		& = \left\langle  
		\F f,  \sum_{j =-J} ^J \sum_{k = -J(K)} ^{J(K)}  e^{2\pi x \cdot} \varphi_k \varphi_j 
		\right\rangle .
\end{align*}
If necessary, taking $K =K_J$ sufficiently large, we have
\begin{equation*}
		\left\langle  f_{J(K)} , \sum_{j =-J} ^J \F^{-1} (\varphi_j)(x- \cdot)   \right\rangle = \left\langle  
		\F f,  \sum_{j =-J} ^J   e^{2\pi x \cdot} (\varphi_{j-1} + \varphi_j + \varphi_{j+1} ) \varphi_j 
		\right\rangle .
\end{equation*}
Since $ \varphi_{j-1} + \varphi_j + \varphi_{j+1} =1 $ on the supp $\varphi_j$, we obtain
\begin{equation*}
		\left\langle  f_{J(K)} , \sum_{j =-J} ^J \F^{-1} (\varphi_j)(x- \cdot)   \right\rangle = \left\langle  
	\F f,  \sum_{j =-J} ^J   e^{2\pi x \cdot}  \varphi_j 
	\right\rangle = f_J.
\end{equation*}
Hence
\begin{equation*}
	 \sum_{j =-J} ^J   \F ^{-1}( \varphi_j \F F) =
	\lim_{K \to \infty} \sum_{j =-J} ^J   \F ^{-1} ( \varphi_j \F f_{J(K)})
=	
\sum_{j =-J} ^J  \F^{-1} (\varphi_j \F f)	= f_J .
\end{equation*}
Thus we obtain
\begin{equation*}
	\int_\rn F(x) g(x) \d x = \lim_{J \to \infty }\int_\rn \sum_{j=-J}^J f_J (x) g(x) \d x .
\end{equation*}
Since $g\in \mathcal{H}_{\vec p \, ^\prime}^{t',r'}$ is arbitrary, it follows that $f_J$ converges to $F$ in the weak-$*$ topology of $M_{\vec p}^{t,r} $.

Finally, we prove (\ref{eq char BM LP}). Since $f_J$ converges to $F$ in the weak-$*$ topology of $M_{\vec p}^{t,r} $, by  (i) of Lemma \ref{lem f le LP}, we obtain
\begin{equation*}
	\|F \|_{M_{\vec p}^{t,r}  } \le \liminf_{J\to \infty} 	\|F_J \|_{M_{\vec p}^{t,r}  } \lesssim \left\|  \left(  \sum_{j\in \mathbb Z} | \F^{-1} (\varphi_j \F f) |^2  \right) ^{1/2}    \right\|_{ M_{\vec p}^{t,r}}.
\end{equation*}
The opposite inequality is  obtained by using  (i) of Lemma \ref{ell 2 le BM and block} and the equation 
\begin{equation*}
	\F ^{-1}( \varphi_j \F F) = 	\lim_{K \to \infty}    \F ^{-1} ( \varphi_j \F f_{J(K)})
	=		  \F^{-1} (\varphi_j \F f) ,
\end{equation*}
since $F \in M_{\vec p}^{t,r}$.
\end{proof}

The next lemma concerns the uniqueness of $F$ in Lemma \ref{lem LP char S'} when $f \in M_{\vec p}^{t,r} $.
\begin{lemma}
		Let $ 1 < \vec p <\infty$. Let   $1 < n / ( \sum_{i=1}^n  1/p_{i})   < t <r<\infty $ or $ 1< n / ( \sum_{i=1}^n  1/p_{i})  \le t< r =\infty  $. 
			Let  $\varphi$  be the same as in Definition \ref{homo unity}.
		If $f \in M_{\vec p}^{t,r} $, then 
		\begin{equation} \label{eq unique when f in BM}
			f = \lim_{J \to \infty} \sum_{j =-J} ^J  \F^{-1} (\varphi_j \F f)
		\end{equation}
		in the  weak-$*$ topology of $M_{\vec p}^{t,r} $.
\end{lemma}
\begin{proof}
	In in view of Lemmas   \ref{ell 2 le BM and block} and \ref{lem LP char S'}, we have
	\begin{equation*}
		h = \lim_{J \to \infty} \sum_{j =-J} ^J  \F^{-1} (\varphi_j \F f) \in M_{\vec p}^{t,r} ,
	\end{equation*}
where the convergence takes place in the  weak-$*$ topology of $M_{\vec p}^{t,r} $. Since $\F (f-h)$ is supported in the origin, $f-h$ must be a polynomial. Since $f,h$ both belong to $M_{\vec p}^{t,r}$, we must have $ f-h =0 $. Consequently, we have (\ref{eq unique when f in BM}).
\end{proof} 

Finally, we prove Theorem \ref{LP char mixed BM} (iv).
\begin{lemma}
	Let $ 1 < \vec p <\infty$. Let   $1 < n / ( \sum_{i=1}^n  1/p_{i})   < t <r<\infty $ or $ 1< n / ( \sum_{i=1}^n  1/p_{i})  \le t< r =\infty  $. 
	Let $\varphi$  be the same as in Definition \ref{homo unity}.
		Assume that $g \in \mathscr S ' $ satisfies 
	\begin{equation} \label{condi ell 2 H}
		\left\|  \left(  \sum_{j\in \mathbb Z} | \F^{-1} (\varphi_j \F g) |^2  \right) ^{1/2}    \right\|_{ \mathcal{H}_{\vec p \, ^\prime}^{t',r'} } <\infty.
	\end{equation}
	Then the limit $G = \lim_{J \to \infty} \sum_{j =-J} ^J  \F^{-1} (\varphi_j \F g)  $ exists in the weak topology of $\mathcal{H}_{\vec p \, ^\prime}^{t',r'}$ and we have
	\begin{equation} \label{eq char H block LP}
		\|G  \|_{ \mathcal{H}_{\vec p \, ^\prime}^{t',r'}} \approx \left\|  \left(  \sum_{j\in \mathbb Z} | \F^{-1} (\varphi_j \F g) |^2  \right) ^{1/2}    \right\|_{ \mathcal{H}_{\vec p \, ^\prime}^{t',r'} } .
	\end{equation}
\end{lemma}

\begin{proof}
	We only prove the case $1 < n / ( \sum_{i=1}^n  1/p_{i})   < t <r<\infty $ since $ 1< n / ( \sum_{i=1}^n  1/p_{i})  \le t< r =\infty  $ is showed in \cite[Lemma 3.8]{N24}. Define 
	$g_J = \sum_{j=-J}^J \F ^{-1} (\varphi_j \F g )  $. Since $ \varphi \in \mathscr S  $, we have
	$  | \F ^{-1} (\varphi_j \F g_J ) | \lesssim \M ( \F ^{-1} (\varphi_j \F g ) ) $
	where the implicit constant in independent of $J$.  By virtue of Theorem \ref{HL mix block r le infty}  and the condition (\ref{condi ell 2 H}), we obtain 
	\begin{equation*}
		\| \F ^{-1} (\varphi_j \F g_J ) \|_{  \mathcal{H}_{\vec p \, ^\prime}^{t',r'} } \lesssim \|  \M ( \F ^{-1} (\varphi_j \F g ) ) \|_{  \mathcal{H}_{\vec p \, ^\prime}^{t',r'} } \lesssim \|   \F ^{-1} (\varphi_j \F g ) \|_{  \mathcal{H}_{\vec p \, ^\prime}^{t',r'} } <\infty.
	\end{equation*}
Thus, the lattice property of the space $ \mathcal{H}_{\vec p \, ^\prime}^{t',r'}  $ leads to the fact $ \F ^{-1} (\varphi_j \F g_J )  \in  \mathcal{H}_{\vec p \, ^\prime}^{t',r'}$ for all $j$. Note that $ \F ^{-1} (\varphi_j \F g_J )  =0$ if $j >J+1$ or $j<-J -1$,  and $0 \notin \operatorname{supp} \F (g_J)$. Hence
we obtain
\begin{equation*}
	g_J = \sum_{j= -J -1}^{J+1} \F^{-1} ( \varphi_j \F g_J ) \in \mathcal{H}_{\vec p \, ^\prime}^{t',r'} .
\end{equation*}
 Recall that $  (M_{\vec p}^{t,r}  ) ^*  = \mathcal{H}_{\vec p \, ^\prime}^{t',r'} $ (Theorem \ref{dual mixed BM}) and $M_{\vec p}^{t,r}$ is separable (Theorem \ref{separable mixed BM}).  Hence by the Banach-Alaoglu theorem, we can find a subsequence $ \{ g_{ J(K)}\}_{ K=-\infty}^\infty $ which converges in the weak-$*$ topology of $ \mathcal{H}_{\vec p \, ^\prime}^{t',r'}$ to a function $G \in \mathcal{H}_{\vec p \, ^\prime}^{t',r'}$. Namely, we obtain 
\begin{equation} \label{weak* H block}
	\lim_{K \to \infty} \int_\rn g_{J(K)} (x) f(x) \d x= \int_\rn G (x)f (x) \d x
\end{equation}
for all $f \in M_{\vec p}^{t,r} $.

Next we prove that $g_J$ converges  to $G$ in the topology of $ \mathcal{H}_{\vec p \, ^\prime}^{t',r'}$. Thanks to the fact $G \in \mathcal{H}_{\vec p \, ^\prime}^{t',r'}$  and Lemma \ref{LP block}, we have
\begin{equation*}
	\lim_{J \to \infty} \left\| G - \sum_{j=-J}^J \F^{-1} (\varphi_j \F G) \right\|_{ \mathcal{H}_{\vec p \, ^\prime}^{t',r'}} =0.
\end{equation*}
By virtue of (\ref{weak* H block}) and the same argument of Lemma \ref{lem LP char S'}, we obtain
\begin{align} \label{F G = F g}
	\sum_{j=-J}^J \F^{-1} (\varphi_j \F G) = \lim_{K \to \infty} \left\langle g_{J (K)},\sum_{j=-J}^J \F^{-1} (\varphi_j) (x- \cdot)  \right\rangle =\sum_{j=-J}^J  \F^{-1} (\varphi_j \F g)
\end{align}
for sufficiently large $J \in \mathbb N$. Thus we get
\begin{equation*}
	\lim_{J \to \infty} \left\| G - \sum_{j=-J}^J \F^{-1} (\varphi_j \F g) \right\|_{ \mathcal{H}_{\vec p \, ^\prime}^{t',r'}} = \lim_{J \to \infty} \left\| G - \sum_{j=-J}^J \F^{-1} (\varphi_j \F G) \right\|_{ \mathcal{H}_{\vec p \, ^\prime}^{t',r'}}  = 0.
\end{equation*}
Finally, we prove the norm estimate (\ref{eq char H block LP}). Since $g_J$ converges to $G$ in the weak  topology of $\mathcal{H}_{\vec p \, ^\prime}^{t',r'}$, using (ii) of Lemma \ref{lem f le LP}, we have
\begin{equation*}
	\|G\|_{\mathcal{H}_{\vec p \, ^\prime}^{t',r'} } \le \liminf_{J \to \infty} \|g_J\|_{\mathcal{H}_{\vec p \, ^\prime}^{t',r'} } \lesssim  \left\|  \left(  \sum_{j\in \mathbb Z} | \F^{-1} (\varphi_j \F g) |^2  \right) ^{1/2}    \right\|_{ \mathcal{H}_{\vec p \, ^\prime}^{t',r'} }.
\end{equation*}
Using (ii) of Lemma \ref{ell 2 le BM and block} and (\ref{F G = F g}), we obtain
the opposite inequality  since  $G \in \mathcal{H}_{\vec p \, ^\prime}^{t',r'}$.
\end{proof} 

\section{Characterization by the heat semigroup} \label{heat char}

 Let $\Delta$  be the Laplace operator. Then the  heat semigroup $\{e^{\alpha \Delta}  \}_{\alpha >0}$ is defined as 
  \begin{equation*}
 	e^{\alpha \Delta} f(x) := \F^{-1} ( e^{-\alpha |\cdot|^2} \F f ) (x)= \frac{1}{ (4\pi \alpha)^n } \int_\rn \exp \left( - \frac{|x-y|^2}{4\alpha } \right) f(y)  \d y
 \end{equation*}
for suitable functions $f$.
\begin{theorem} \label{heat semigroup char}
		Let $ 1 < \vec p <\infty$. Let   $1 < n / ( \sum_{i=1}^n  1/p_{i})   < t <r<\infty $ or $ 1< n / ( \sum_{i=1}^n  1/p_{i})  \le t< r =\infty$. 
		Let $\varphi$ be the same as in Definition \ref{homo unity}.	
	Then $f = \sum_{j\in \mathbb Z} \F^{-1} (\varphi_j\F f)$ holds in $ M_{\vec p}^{t,r} $	if and only if 
	\begin{equation*}
		\lim_{\alpha \to 0} \left( e^{\alpha \Delta} f - e^{\alpha^{-1} \Delta} f \right) = f
	\end{equation*}
holds in $ M_{\vec p}^{t,r} $.
\end{theorem}
To prove Theorem \ref{heat semigroup char}, we divide it into three steps.

\begin{lemma}
		Let $ 1 < \vec p <\infty$. Let   $1 < n / ( \sum_{i=1}^n  1/p_{i})   < t <r<\infty $ or $ 1< n / ( \sum_{i=1}^n  1/p_{i})  \le t< r =\infty$. 
		If $f = \sum_{j\in \mathbb Z} \F^{-1} (\varphi_j\F f) $ in  $ M_{\vec p}^{t,r} $,  then $ f= \lim_{\alpha \to 0}  e^{\alpha \Delta} f$ in  $ M_{\vec p}^{t,r} $.
\end{lemma}
\begin{proof}
	By the assumption, for all $\epsilon>0$, there exists $J \in \mathbb N$  such that
	\begin{equation*}
		\left\| \sum_{j >|J|} \F^{-1} (\varphi_j\F f)   \right\|_{ M_{\vec p}^{t,r} }  \le \epsilon .
	\end{equation*}
Fix $J \in \mathbb N$ as above. Then
\begin{align*}
	\| e^{\alpha \Delta} f -f \|_{  M_{\vec p}^{t,r} } & \le \left\| e^{\alpha \Delta} \left( \sum_{|j| \le J} \F^{-1} (\varphi_j \F f) \right)  - \sum_{j \le |J|} \F^{-1} (\varphi_j\F f)   \right\|_{  M_{\vec p}^{t,r} }  \\
	& + \left\| e^{\alpha \Delta} \left( \sum_{|j| > J} \F^{-1} (\varphi_j \F f) \right)  - \sum_{j > |J|} \F^{-1} (\varphi_j\F f)   \right\|_{  M_{\vec p}^{t,r} } \\
	& : = I +II .
\end{align*}
For $II$, 
note that the heat semigroup is bounded on $M_{\vec p}^{t,r}  $. Indeed, by Corollary \ref{convolution} and the fact that $\int_\rn e^{-|x|^2} \d x = \pi^{n/2}$ (for example, sees \cite[A.1]{G14}),
\begin{align*}
	\left\| e^{\alpha \Delta} f \right\|_{ M_{\vec p}^{t,r} } & = \left\|  \frac{ 1}{ (4\pi \alpha)^{n/2}   }\int_\rn e^{ ( - \frac{|\cdot - y|^2}{4\alpha} ) }  f(y) \d y \right\|_{ M_{\vec p}^{t,r} } \\
	& \le C_{n, \vec p, r}  \frac{ 1}{ (4\pi \alpha)^{n/2}   }\int_\rn e^{ ( - \frac{y^2}{4\alpha} ) } \d y  \left\| f\right\|_{ M_{\vec p}^{t,r} } \\
	& = C_{n, \vec p, r}  \left\| f\right\|_{ M_{\vec p}^{t,r} } .
\end{align*}
Hence we have
\begin{align*}
	II \lesssim \left\|  \sum_{j > |J|} \F^{-1} (\varphi_j\F f)   \right\|_{  M_{\vec p}^{t,r} } \le \epsilon .
\end{align*}

For $I$, let $j \in \mathbb Z \cap [-J,J]$. Note that since $f \in M_{\vec p}^{t,r} \hookrightarrow M_{\vec p}^{t,\infty} \hookrightarrow \mathscr S ' $ (\cite[page 2219]{N24}), $ \F^{-1} ( \varphi_j \F f )  \in C^\infty $. Thus, from \cite[Exercise 108]{SDH20}
\begin{align*}
	e^{\alpha \Delta}  \F^{-1} ( \varphi_j \F f ) - \F^{-1} ( \varphi_j \F f ) = \int_0^\alpha e^{s \Delta } \Delta ( \F ^{-1} (\varphi_j \F  f) ) \d s .
\end{align*}
Then by the Minkowski inequality and the support condition of $\varphi$, we have
\begin{align*}
	\left\|  e^{\alpha \Delta}  \F^{-1} ( \varphi_j \F f ) - \F^{-1} ( \varphi_j \F f )  \right\| _{ M_{\vec p}^{t,r} } & = \left\|  \int_0^\alpha e^{s \Delta } \Delta ( \F ^{-1} (\varphi_j \F  f) ) \d s  \right\| _{ M_{\vec p}^{t,r} } \\
	& \lesssim  \int_0^\alpha  \left\| \F ^{-1} ( e^{ -s |\cdot|^2 } |\cdot|^2  \varphi_j \F  f)    \right\| _{ M_{\vec p}^{t,r} } \d s \\
	& \approx \int_0^\alpha e^{ - 4^j s} 4^j  \left\| \F ^{-1} (  \varphi_j \F  f)    \right\| _{ M_{\vec p}^{t,r} } \d s .
\end{align*}
Thanks to Lemma \ref{ell 2 le BM and block} and $e^{ - 4^j s} \le 1 $, we obtain
\begin{equation*}
	\left\|  e^{\alpha \Delta}  \F^{-1} ( \varphi_j \F f ) - \F^{-1} ( \varphi_j \F f )  \right\| _{ M_{\vec p}^{t,r} } \le C_J \alpha  \| f\|_{ M_{\vec p}^{t,r}   } \to 0
\end{equation*}
as $\alpha \to 0$. Thus we obtain the desired result.
\end{proof}
Next we show that $ e^{\alpha \Delta} f$ vanished when $\alpha $ tends to infinity. 

\begin{lemma}
		Let $ 1 < \vec p <\infty$. Let   $1 < n / ( \sum_{i=1}^n  1/p_{i})   < t <r<\infty $ or $ 1< n / ( \sum_{i=1}^n  1/p_{i})  \le t< r =\infty$. 
	If $f = \sum_{j\in \mathbb Z} \F^{-1} (\varphi_j\F f) $ in  $ M_{\vec p}^{t,r} $,  then $  \lim_{\alpha \to \infty}  e^{\alpha \Delta} f =0 $ in  $ M_{\vec p}^{t,r} $.
\end{lemma}
\begin{proof}
	By the assumption, for all $\epsilon>0$, there exists $J \in \mathbb N$ such that
		\begin{equation*}
		\left\| \sum_{j >|J|} \F^{-1} (\varphi_j\F f)   \right\|_{ M_{\vec p}^{t,r} }  \le \epsilon .
	\end{equation*}
	Fix $J \in \mathbb N$ as above. Then
	\begin{align*}
		\| e^{\alpha \Delta} f  \|_{  M_{\vec p}^{t,r} } & \le \left\| e^{\alpha \Delta} \left( \sum_{|j| \le J} \F^{-1} (\varphi_j \F f) \right)    \right\|_{  M_{\vec p}^{t,r} }  
		 + \left\| e^{\alpha \Delta} \left(  \sum_{j > |J|} \F^{-1} (\varphi_j\F f)  \right) \right\|_{  M_{\vec p}^{t,r} } \\
		& : = I +II .
	\end{align*}
	For $II$, by the boundedness of  the heat semigroup on $M_{\vec p}^{t,r}  $, we have
	\begin{align*}
		II \lesssim \left\|  \sum_{j > |J|} \F^{-1} (\varphi_j\F f)   \right\|_{  M_{\vec p}^{t,r} } \le \epsilon .
	\end{align*}	
	For $I$, let $j \in \mathbb Z \cap [-J,J]$. Then by the support condition of $\varphi$, we have
		\begin{align*}
		\left\|  e^{\alpha \Delta}  \F^{-1} ( \varphi_j \F f )  \right\| _{ M_{\vec p}^{t,r} } & \lesssim 
		e^{ -4^j \alpha } 
		 \left\|   \F ^{-1} (\varphi_j \F  f)  \right\| _{ M_{\vec p}^{t,r} } .
	\end{align*}
		Thanks to Lemma \ref{ell 2 le BM and block}, we obtain
		\begin{equation*}
			\left\|  e^{\alpha \Delta}  \F^{-1} ( \varphi_j \F f )  \right\| _{ M_{\vec p}^{t,r} } \le C_J \sup_{ |j|\le J } 	e^{ -4^j \alpha } 
			\left\|   \F ^{-1} (\varphi_j \F  f)  \right\| _{ M_{\vec p}^{t,r} }  \to 0
		\end{equation*}
	as $\alpha \to  \infty$. Thus we obtain the desired result.
\end{proof}
At last, we prove the necessity of Theorem \ref{heat semigroup char}.
\begin{lemma}
		Let $ 1 < \vec p <\infty$. Let   $1 < n / ( \sum_{i=1}^n  1/p_{i})   < t <r<\infty $ or $ 1< n / ( \sum_{i=1}^n  1/p_{i})  \le t< r =\infty$. 
	Suppose that 
	\begin{equation*}
		\lim_{\alpha \to 0} e^{\alpha \Delta}  f - e^{\alpha^{-1} \Delta}  f = f
	\end{equation*}
in $M_{\vec p}^{t,r} $. Then 
\begin{equation*}
	f = \sum_{j \in \mathbb Z}  \F ^{-1} (\varphi_j \F  f) 
\end{equation*}
in $M_{\vec p}^{t,r} $.
\end{lemma}
\begin{proof}
	By the assumption, for all $\epsilon >0$, there exists $\delta >0$ such that for all $\alpha \in (0, \delta) $, 
	\begin{equation} \label{e alpha Delta - f epsilon}
		\left\|  e^{\alpha \Delta}  f - e^{\alpha^{-1} \Delta}  f -f  \right\|_{ M_{\vec p}^{t,r} }  \le \epsilon .
	\end{equation}
Fix $\alpha$ as above and let $J \in \mathbb N$ arbitrary large. Then
\begin{align*}
	\left\| f - \sum_{ |j| \le J}  \F ^{-1} (\varphi_j \F  f)  \right\|_{ M_{\vec p}^{t,r} } & \le 	\left\| f - ( e^{\alpha \Delta}  f - e^{\alpha^{-1} \Delta}  f )  \right\|_{ M_{\vec p}^{t,r} } \\
	& +	\left\| ( e^{\alpha \Delta}   - e^{\alpha^{-1} \Delta}  )  \left( f- \sum_{ |j| \le J}  \F ^{-1} (\varphi_j \F  f ) \right) \right\|_{ M_{\vec p}^{t,r} } \\
	& +	\left\|  \sum_{ |j| \le J}  \F ^{-1}  \left(  \varphi_j \F \left( ( e^{\alpha \Delta}   - e^{\alpha^{-1} \Delta}  ) f -f \right)    \right)  \right\|_{ M_{\vec p}^{t,r} } \\
	& =: I +II + III.
\end{align*}
By (\ref{e alpha Delta - f epsilon}), $ I \le \epsilon $.

For $III$, since $( e^{\alpha \Delta}   - e^{\alpha^{-1} \Delta}  ) f -f \in  M_{\vec p}^{t,r} $, by (\ref{f_J lesssim LP}), Lemma \ref{ell 2 le BM and block}, and (\ref{e alpha Delta - f epsilon}), we have
\begin{align*}
	III & \lesssim 	\left\| \left(  \sum_{ j\in \mathbb Z} \left|  \F ^{-1}  \left(  \varphi_j \F \left( ( e^{\alpha \Delta}   - e^{\alpha^{-1} \Delta}  ) f -f \right)    \right) \right|^2 \right)^{1/2} \right\|_{ M_{\vec p}^{t,r} }  \\
	& \lesssim \left\|  ( e^{\alpha \Delta}   - e^{\alpha^{-1} \Delta}  ) f -f  \right\|_{ M_{\vec p}^{t,r} } \le \epsilon .
\end{align*}

For $II$, consider the kernel
\begin{equation*}
	m_\alpha (\xi ) =\frac{e^{ -\alpha |\xi|^2   } - e^{ - \alpha^{-1}  |\xi|^2 }  }{ |\xi|^2 }  (1+ |\xi|^4) .
\end{equation*}
Since $ e^{ -\alpha |\xi|^2   } - e^{ - \alpha^{-1}  |\xi|^2 }  \in \mathscr S$, $ \partial^\beta m_{\alpha} (\xi) \lesssim |\xi|^{- |\beta|} $  for all multi-indices $\beta \in \mathbb N_0^n$. Hence,
\begin{align*}
	II & =  	\left\|  \F^{-1}  
	\left(  m_\alpha \times \frac{|\cdot|^2 }{ (1+ |\cdot|^4)} \F \left( f- \sum_{ |j| \le J}  \F ^{-1} (\varphi_j \F  f)  \right)  \right)
	 \right\|_{ M_{\vec p}^{t,r} } \\
	 &\lesssim \left\|  \F^{-1}  
	 \left(   \frac{|\cdot|^2 }{ (1+ |\cdot|^4)} \F \left(  \sum_{ |j| > J}  \F ^{-1} (\varphi_j \F  f)  \right)  \right)
	 \right\|_{ M_{\vec p}^{t,r} } \\
	 & \approx \left\|  \F^{-1}  
	 \left(    \sum_{ |j| > J}  \frac{|\cdot|^2 }{ (1+ |\cdot|^4)}   \varphi_j \F  f  \right)
	 \right\|_{ M_{\vec p}^{t,r} } .
\end{align*}
By the support of $\varphi_j$, the triangle inequality and Lemma \ref{ell 2 le BM and block}, we obtain
\begin{align*}
	II \lesssim  \sum_{ |j| > J}  \frac{2^{2j} }{ 1+ 2^{4j} } \left\|    \F ^{-1} (\varphi_j \F  f)  
	\right\|_{ M_{\vec p}^{t,r} }  \lesssim  \sum_{ |j| > J}  \frac{2^{2j} }{ 1+ 2^{4j} }  \|f\|_{M_{\vec p}^{t,r} } \lesssim 2^{-2J}  \|f\|_{M_{\vec p}^{t,r} } 
\end{align*}
Since $J$ is arbitrary, $II \lesssim \epsilon$. Combining the estimates of $I-III$, we obtain the desired result.
\end{proof}

\section{Wavelet characterization} \label{wavelet char}
Denote by $\mathcal P_d := \mathcal P_d (\rn) $  the set of all polynomial functions with degree less than or equal to $d$. Thus, $\mathcal P  = \cup_{d =0}^\infty \mathcal P_d $. The set $ \mathcal P_d  ^{\perp}$ denotes the set of measurable functions $f$ such that $ \langle \cdot \rangle ^d f \in L^1 $ and $\int_\rn x^\alpha f (x) \d x =0 $ for all $\alpha \in \mathbb N_0^n, |\alpha | \le d  $, where $  \langle x \rangle = (1 + |x| ^2) ^{1/2} $. We say that function $f$ satisfy the moment condition of order $d$ if $f$ satisfies the above conditions. In this case, one also writes $f \perp  \mathcal P_d $.

Choose compactly supported $C^m$-functions for large enough $m \in \mathbb N$, $\psi^\ell $, $\ell = 1,2, \ldots, 2^n-1$ such that the following conditions hold:

(i) The system 
\begin{equation*}
	\{ \psi^{\ell}_{j,k} : j \in \mathbb Z , k \in \mathbb Z^n, \ell = 1,2, \ldots, 2^n-1 \}
\end{equation*}
is an orthonormal basis of $L^2 $. Here, given a function $f$ defined on $\rn$, define
\begin{equation*}
	f_{j,k} : = 2^{ jn/2 } f (2^{j}  \cdot - k)
\end{equation*}
for $ j \in \mathbb Z , k \in \mathbb Z^n$.

(ii) Fix a large integer $L \in \mathbb N_0 $. We have
\begin{equation} \label{psi ell wavelet}
	\psi^\ell \in  \mathcal P_L  ^{\perp}, \ell  = 1,2, \ldots, 2^n-1 .
\end{equation}
In addition, they are real-valued and supp $ \psi^\ell = [0,2N-1]^n $ for some $N \in \mathbb N$. See \cite{LPM91}, for example.

We also set $ \chi_{j,k} := 2^{ jn/2} \chi_{ Q_{j,k} }$ for the dyadic cube $Q_{j,k} \in \D$. Then using the $L^2$-inner product $\langle \cdot, \cdot \rangle $, for $f \in L_{\operatorname{loc}}^1$, we define the square function $Sf$ by
\begin{equation*}
	Sf : = S ^{\psi ^{\ell} } f := \left( \sum_{\ell =1}^{2^n -1 }  \sum_{ (j,k) \in \mathbb Z^{n+1}} |\langle f, \psi_{j,k} ^\ell \rangle  \chi_{j,k}  |^2  \right)^{1/2} .
\end{equation*}

\begin{theorem} \label{wavelet char mixed BM}
		Let $ 1 < \vec p <\infty$. Let   $1 < n / ( \sum_{i=1}^n  1/p_{i})   < t <r<\infty $ or $ 1< n / ( \sum_{i=1}^n  1/p_{i})  \le t< r =\infty$. 		
		Assume that $L = n$ and $m = n+1$ as in (\ref{psi ell wavelet}). Then for $f \in M_{\vec p}^{t,r} $, we have
		\begin{equation*}
			\| f\|_{ M_{\vec p}^{t,r}  } \approx \| S f \|_{ M_{\vec p}^{t,r}  }  .
		\end{equation*}
\end{theorem}

We make some preparation for Theorem \ref{wavelet char mixed BM}.

\begin{lemma}[p. 595, \cite{G142}]  \label{Smoothness and Vanishing Moments}
	Let $\mu ,\nu \in \mathbb R$, $M,N >0$ and $L \in \mathbb N_0$ satisfy $\nu \ge  \mu$  and $N >M +L +n$. Suppose that $ \phi_{ (\mu) } \in C ^L $ satisfies 
	\begin{equation*}
		| \partial ^\alpha  \phi_{ (\mu) }  (x) | \le A_\alpha  \frac{2^{\mu  (n+L) } }{ (1+ 2^\mu | x - x_\mu |)^M } 
	\end{equation*}
for all $\alpha \in \mathbb N_0^n ,  | \alpha | =L$. Furthermore, suppose that $ \phi _{ (\nu) }$ is a measurable function satisfying 
\begin{equation*}
	\int_\rn  \phi _{ (\nu) } (x)  (x - x_{\nu} ) ^\beta \d x = 0 
\end{equation*}
for all $ | \beta | \le L-1$, and $  |  \phi _{ (\nu) } (x) | \le B 2^{\nu n}  ( 1+ 2^\nu  |x-x_\nu | ) ^{-N} $,
where the former condition is supposed to be vacuous when $L=0$. Then it holds
\begin{equation*}
	\left|  \int_\rn \phi_{ (\mu) }  (x)  \phi _{ (\nu) } (x) \d x \right| \le C_{ A_\alpha , B,L,M,N } 2^{ \mu n - ( \nu - \mu ) L }  ( 1+ 2^\mu |x_\mu - x_\nu | ) ^{-M}
\end{equation*}
where 
\begin{equation*}
	C_{ A_\alpha , B,L,M,N } = B \left(  \sum_{ |\alpha | =L  }  \frac{ A_\alpha }{ \alpha ! }  \right)  \frac{ \omega_n (N-M-L) }{ N-M-L-n},
\end{equation*}
and $\omega_n $ is the volume of the unit ball in $\rn$.
\end{lemma}

Let $ \psi \in \mathscr S $ satisfy $ \chi_{B_(0,2)} \le \psi \le \chi_{B(0,4)}$. Define $\varphi := \psi -\psi (2\cdot)$. For $j \in \mathbb Z$, let $\varphi_j = \varphi (2^{-j} \cdot)$. Let $\psi \in   \mathcal P_0  ^{\perp} $. Then we have $\varphi \in  \mathcal P_0  ^{\perp} $. Moreover, for $ f \in M_{\vec p}^{t,r} $, the coupling $ \langle f ,  \psi_{j,k} ^\ell \rangle $ is well defined.

\begin{theorem} \label{wavelet one direction}
	Let $L=0$. 	Let $ 1 < \vec p <\infty$. Let   $1 < n / ( \sum_{i=1}^n  1/p_{i})   < t <r<\infty $ or $ 1< n / ( \sum_{i=1}^n  1/p_{i})  \le t< r =\infty$. Then 
	\begin{equation*}
		 \| S f \|_{ M_{\vec p}^{t,r}  } \lesssim \| f\|_{ M_{\vec p}^{t,r}  }.
	\end{equation*}
\end{theorem}
\begin{proof}
	For a measurable function $f$ and $a>0$, define $f_a := a^{-n} f (a^{-1} \cdot)$. Define $\phi := \F^{-1} \varphi \in \mathscr S $. Let $\eta \in \mathscr S $ such that $\F \eta = \varphi (2^{-1} \cdot )  +  \varphi ( \cdot )  + \varphi (2 \cdot ) $. Since $\varphi$ vanishes around the origin, we see that $\eta \in \mathcal P_0  ^{\perp}  $ and 
	\begin{equation*}
		\sum_{ j \in \mathbb Z } \F ( \eta_{2^{-j} } ) (\xi) \F (\phi_{ 2^{-j} } ) (\xi) =1
	\end{equation*}
for all $\xi \neq 0$. Thus, by the Calder\'on reproducing formula (for example, see \cite[Corollary 1.1.7]{G142}), we have
\begin{equation*}
	f = \sum_{ m = -\infty}^{\infty} \eta_{2^{-m} } * \phi_{2^{-m} } * f .
\end{equation*}
Let $j \in \mathbb Z$ be fixed. Then 
\begin{align*}
 |\langle f, \psi_{j,k} ^\ell \rangle  \chi_{j,k}  | & \le \sum_{ m \in \mathbb Z}  | \langle \eta_{2^{-m} } * \phi_{2^{-m} } * f, \psi_{j,k} ^\ell \rangle |  \chi_{j,k} \\
 & = \sum_{ m \in \mathbb Z}  | \langle  \phi_{2^{-m} } * f, \tilde \eta_{2^{-m} } * \psi_{j,k} ^\ell \rangle |  \chi_{j,k} .
\end{align*}
From \cite[proof of Theorem 4.3]{N24}, we have
\begin{equation*}
 \left( \sum_{k \in \mathbb Z^n}  |\langle f, \psi_{j,k} ^\ell \rangle  \chi_{j,k}  |^2\right)	^{1/2} \lesssim \sum_{ m \in \mathbb Z} 2^{ - |j-m|} \mathcal M ( | \phi_{2^{-m} } * f|  ) (x).
\end{equation*}
Taking the $\ell^2$-norm over $j\in \mathbb Z$ and $\ell = 1,2,\ldots, 2^n -1$ and by the Cauchy-Schwartz inequality, we get
\begin{align*}
	\left( \sum_{\ell =1}^{2^n -1 }  \sum_{ (j,k) \in \mathbb Z^{n+1}} |\langle f, \psi_{j,k} ^\ell \rangle  \chi_{j,k}  |^2  \right)^{1/2} 
& \lesssim 	\left( \sum_{ j\in \mathbb Z } \left( \sum_{ m \in \mathbb Z} 2^{ - |j-m|} \mathcal M ( | \phi_{2^{-m} } * f|  ) (x)  \right) ^2    \right)^{1/2} \\
&\lesssim \left( \sum_{ j\in \mathbb Z }  \sum_{ m \in \mathbb Z} 2^{ - 2|j-m|} \mathcal M ( | \phi_{2^{-m} } * f|  ) (x)   ^2    \right)^{1/2} \\
& \lesssim \left( \sum_{ m \in \mathbb Z}  \mathcal M ( | \phi_{2^{-m} } * f|  ) (x)   ^2    \right)^{1/2} .
\end{align*}
Then by the boundedness of Hardy-Littlewood on mixed Bourgain-Morrey spaces and Theorem \ref{LP char mixed BM}, we obtain the desired result.
\end{proof}

Before proving the another direction of Theorem \ref{wavelet char mixed BM}, we recall two function spaces. Let $\varphi \in C_c^\infty $ such that $ \chi_{ B (0,4) \backslash B (0,2) } \le \varphi \le \chi_{ B (0,8) \backslash B (0,3/2) }  $. Let $1\le p,q  \le \infty $ and $s\in \mathbb R$. Then the homogeneous Besov spaces $ \dot B _{p,q}^s:= \dot B _{p,q}^s (\rn) $  is the set of all functions $f \in \mathscr S ' \backslash \mathcal P $ such that 
\begin{equation*}
	\| f\|_{\dot B _{p,q}^s  } : = \left(\sum_{j \in \mathbb Z} 2^{ js  q}  \| \F^{-1} ( \varphi (2^{-j} \cdot ) \F f ) \|_{L^p} ^q \right) ^{1/q} < \infty .
\end{equation*}
Moreover, let $1\le p,q  \le \infty $ and $s\in \mathbb R$. If we replace $ L^p$-norm by the  Morrey $M_{p,t}^\infty  $-norm in the above, we obtain the homogeneous Besov-Morrey spaces $  \mathcal N_{p,t, q }^s (\rn) $. 
That is 
\begin{equation*}
	\| f\|_{\mathcal N_{p,t, q }^s (\rn) } : =  \left(\sum_{j \in \mathbb Z} 2^{ js  q}  \| \F^{-1} ( \varphi (2^{-j} \cdot ) \F f ) \|_{M_{p,t}^\infty} ^q \right) ^{1/q} < \infty .
\end{equation*}
We refer the reader to \cite{Ma03,NS22,R13,S10} for more results about the homogeneous Besov-Morrey spaces. From the \cite[page 2224]{N24} and the basic embeddings of mixed Bourgain-Morrey spaces, we get
\begin{equation*}
	M_{\vec p}^{t,r} \hookrightarrow M_{\vec p}^{t,\infty} \hookrightarrow M_{\min \vec p}^{t,\infty} \hookrightarrow \mathcal N_{\min \vec p,t, \infty }^0 (\rn) \hookrightarrow \dot B _{\infty,\infty}^{-n/t} .
\end{equation*} 
The last embedding comes from the Sobolev embedding for Besov-Morrey spaces; sees \cite[Theorem 2.4]{Ma03}. Hence we can regard the element of $M_{\vec p}^{t,r}  $ as one of $ \dot B _{\infty,\infty}^{-n/t}  $. By \cite[Theorem 7.20]{FJW91}, we have the wavelet decomposition for $\dot B _{\infty,\infty}^{-n/t}  $. Thus we can expand the all element of $M_{\vec p}^{t,r}  $ by using the wavelet. 

Now we move on the another direction of  Theorem \ref{wavelet char mixed BM}. 
\begin{proof} [Proof of Theorem \ref{wavelet char mixed BM}]

By Theorem \ref{LP char mixed BM}, let us prove
\begin{equation*}
	\left\| \left(\sum_{m\in\mathbb Z} | \phi_{2^{-m}  } * f|^2  \right) ^{1/2} \right\|_{M_{\vec p}^{t,r}  }
 \lesssim 	\| S f\|_{ M_{\vec p}^{t,r}  } 
\end{equation*}
for all $ f \in M_{\vec p}^{t,r}$. Here $\phi = \F^{-1} \varphi $.
Let $m \in \mathbb Z$ be fixed. Thanks to the above remark, for all $ f \in M_{\vec p}^{t,r}$, we can use the wavelet decomposition and estimate each term of the decomposition:
\begin{equation*}
	\phi_{2^{-m}  } * f = \sum_{\ell=1}^{2^n -1} \sum_{ j \in \mathbb Z } \sum_{k \in \mathbb Z^n} \langle f, \psi_{j,k} ^\ell \rangle \phi_{2^{-m}  } * \psi_{j,k} ^\ell.
\end{equation*}
Note that $\F \phi = \varphi =0$ at the origin, $\phi \in \mathcal P_n  ^\perp$. Hence, by Lemma \ref{Smoothness and Vanishing Moments}, we have
\begin{align*}
		2^{jn /2 } |\phi_{2^{-m} } * \psi_{j,k} ^\ell (y)|   \lesssim 2 ^{ \min \{j,m \} n - |j-m| (n+1) } ( 1+ 2^{\min \{j,m \}  } |y - 2^{-j} k | )^{-N}
\end{align*}
where $N \in (n,n+1)$. Thus 
\begin{align*}
	&  \sum_{ j \in \mathbb Z } \sum_{k \in \mathbb Z^n} \langle f, \psi_{j,k} ^\ell \rangle \phi_{2^{-m}  } * \psi_{j,k} ^\ell (y) \\
	& \lesssim \sum_{ j \in \mathbb Z } \sum_{k \in \mathbb Z^n} 2^{-jn/2} 2 ^{ \min \{j,m \} n - |j-m| (n+1) } ( 1+ 2^{\min \{j,m \}  } |y - 2^{-j} k | )^{-N} | \langle f, \psi_{j,k} ^\ell \rangle| \\
	&\le  \sum_{ j \in \mathbb Z } \sum_{k \in \mathbb Z^n} 2 ^{ j n /2 - |j-m| (n+1) } ( 1+ 2^{\min \{j,m \}  } |y - 2^{-j} k | )^{-N} | \langle f, \psi_{j,k} ^\ell \rangle| .
\end{align*}
Let $c_{Q_{j,k}}$ be the center of the cube $Q_{j,k} $. 
 A geometric observation shows 
\begin{align*} 
1+ 	2^{\min \{j,m \}  } |x-y |
	& \le 1 + 2^{ \min \{j,m \} } |y - 2^{-j} k| +  2^{ \min \{0, m-j \} } \sqrt{n}\\
	& \lesssim  1+ 2^{\min \{j,m \}  } |y - 2^{-j} k |
\end{align*}
for all $y \in \rn$. 
Then  we  have
\begin{align*}
	2^{ \min\{ m -j, 0 \} } (1+2^j |y-c_{Q_{j,k}} | ) \le 1 + 2^{ \min\{j,m \}} | y-c_{Q_{j,k}} | \lesssim  1 + 2^{ \min\{j,m \}} | y- 2^{-j} k | .
\end{align*}
Hence 
\begin{align*}
	&  \sum_{ j \in \mathbb Z } \sum_{k \in \mathbb Z^n} \langle f, \psi_{j,k} ^\ell \rangle \phi_{2^{-m}  } * \psi_{j,k} ^\ell (y) \\
	& \lesssim  \sum_{ j \in \mathbb Z } \sum_{k \in \mathbb Z^n} 2 ^{ j n /2 - |j-m| (n+1)  -\min\{m-j , 0\}N  } ( 1+ 2^j  |y - c_{Q_{j,k}} | )^{-N} | \langle f, \psi_{j,k} ^\ell \rangle|  \\
	& \approx  \sum_{ j \in \mathbb Z } \sum_{k \in \mathbb Z^n} 2 ^{ j n /2 - |j-m| (n+1)  -\min\{m-j , 0\}N  }  | \langle f, \psi_{j,k} ^\ell \rangle|   \mathcal M (\chi_{Q_{j,k} }   ) (y) ^{1/w},
\end{align*}
where $0 <w<1$.  Since 
\begin{equation*}
	- |j-m| (n+1)  -\min\{m-j , 0\}N \le - |j-m| ( n+1-N), 
\end{equation*} using the Cauchy-Schwartz inequality for $j$,
we have
\begin{align*}
&\sum_{m\in \mathbb Z} \left( 	\sum_{ j \in \mathbb Z } \sum_{k \in \mathbb Z^n}  | \langle f, \psi_{j,k} ^\ell \rangle \phi_{2^{-m}  } * \psi_{j,k} ^\ell (y)  | \right)^2 \\
& \lesssim \sum_{m\in \mathbb Z} \left( 	\sum_{ j \in \mathbb Z } \sum_{k \in \mathbb Z^n} 2 ^{ j n /2 - |j-m| ( n+1-N)  }  | \langle f, \psi_{j,k} ^\ell \rangle|   \mathcal M (\chi_{Q_{j,k} }   ) (y) ^{1/w} \right)^2  \\
& \le  \sum_{m\in \mathbb Z} 
\sum_{ j \in \mathbb Z } 2 ^{ j n  - |j-m| ( n+1-N)  } \left(  \sum_{k \in \mathbb Z^n}   | \langle f, \psi_{j,k} ^\ell \rangle|   \mathcal M (\chi_{Q_{j,k} }   ) (y) ^{1/w}\right)^2  \sum_{ j \in \mathbb Z } 2 ^{  - |j-m| ( n+1-N)  } \\
& \lesssim  \sum_{m\in \mathbb Z} 
\sum_{ j \in \mathbb Z } 2 ^{ j n  - |j-m| ( n+1-N)  } \left(  \sum_{k \in \mathbb Z^n}   | \langle f, \psi_{j,k} ^\ell \rangle|   \mathcal M (\chi_{Q_{j,k} }   ) (y) ^{1/w}\right)^2
\end{align*}
Taking the $ M_{\vec p}^{t,r}$-norm and using Theorem \ref{mixed vector HL mixed BM}, we have
\begin{align*}
	& \left\| \left(  \sum_{m\in \mathbb Z} \left( 	\sum_{ j \in \mathbb Z } \sum_{k \in \mathbb Z^n} | \langle f, \psi_{j,k} ^\ell \rangle \phi_{2^{-m}  } * \psi_{j,k} ^\ell |  \right)^2  \right) ^{1/2}\right\|_{ M_{\vec p}^{t,r} } \\
	& \lesssim \left\| \left(  \sum_{m\in \mathbb Z} 
	\sum_{ j \in \mathbb Z } 2 ^{ j n  - |j-m| ( n+1-N)  } \left(  \sum_{k \in \mathbb Z^n}   | \langle f, \psi_{j,k} ^\ell \rangle|   \mathcal M (\chi_{Q_{j,k} }   )  ^{1/w}\right)^2  \right) ^{1/2}\right\|_{ M_{\vec p}^{t,r} } \\
	 & \lesssim \left\| \left(  \sum_{m\in \mathbb Z} 
	 \sum_{ j \in \mathbb Z } 2 ^{ j n  - |j-m| ( n+1-N)  } \left(  \sum_{k \in \mathbb Z^n}   | \langle f, \psi_{j,k} ^\ell \rangle|    \chi_{Q_{j,k} }   \right)^2 \right) ^{1/2}\right\|_{ M_{\vec p}^{t,r} } \\
	 &= \left\| \left(  \sum_{m\in \mathbb Z} 
	 \sum_{ j \in \mathbb Z } 2 ^{ j n  - |j-m| ( n+1-N)  }  \sum_{k \in \mathbb Z^n}   | \langle f, \psi_{j,k} ^\ell \rangle|^2    \chi_{Q_{j,k} }   \right) ^{1/2}\right\|_{ M_{\vec p}^{t,r} } \\
	 & \lesssim \left\| \left(  
 \sum_{ j \in \mathbb Z } \sum_{k \in \mathbb Z^n} 	 \left( 2 ^{ j n /2  }  | \langle f, \psi_{j,k} ^\ell \rangle|  \chi_{Q_{j,k} }   
	 \right)^2   \right) ^{1/2}\right\|_{ M_{\vec p}^{t,r} } .
\end{align*}
Together Theorem \ref{wavelet one direction}, we prove Theorem \ref{wavelet char mixed BM}.
\end{proof}

\section{Chain rules } \label{sec chain rules}
In \cite{BX25}, the first and the third author of this paper introduce the weighted homogeneous Bourgain-Morrey Besov spaces and Triebel-Lizorkin spaces associated with the operator $L$ where $L$ is a nonnegative self-adjoint operator on $L^2$ satisfying a Gaussian upper bound on its heat kernel. 
If we consider the special case the weight $\omega \equiv 1$ and the operator $L= - \Delta =- \sum_{i=1}^n \partial_i ^2 $, the Laplacian on $\rn$, replacing the $L^p$-norm by the $L^{\vec p}$-norm, then we obtain the following mixed Bourgain-Morrey  Triebel-Lizorkin spaces.

\begin{definition}
	Let $\varphi$  be a partition of unity as in Definition \ref{homo unity}. Let $ 0 < \vec p \le \infty $. 
	Let  $   0 < n / ( \sum_{i=1}^n  1/p_{i})      < t <r <\infty $ or $ 0< n / ( \sum_{i=1}^n  1/p_{i})    \le t < r=\infty  $. Let $s\in \mathbb R$ and $q \in (0,\infty ]$. Then the  homogeneous Bourgain-Morrey-Triebel-Lizorkin space $\dot F_{\vec p , t,r} ^{s,q}$ is set of all $f\in \mathscr S '$  such that 	
	\begin{equation*}
		\| f\|_{\dot F_{\vec p , t,r} ^{s,q} } :=	\left\|  \left(  \sum_{j\in \mathbb Z} 2^{jsq}| \F^{-1} (\varphi_j \F  (f) )|^q  \right) ^{1/q}    \right\|_{ M_{\vec p}^{t,r}} <\infty .
	\end{equation*}
The  homogeneous Bourgain-Morrey-Besov space $\dot B_{\vec p , t,r} ^{s,q}$ is set of all $f\in \mathscr S '$  such that 	
\begin{equation*}
	\| f\|_{\dot B_{\vec p , t,r} ^{s,q} } :=\left(  \sum_{j\in \mathbb Z} 2^{jsq}	\left\|   \F^{-1} (\varphi_j \F  (f) )   \right\|_{ M_{\vec p}^{t,r}}   ^q  \right) ^{1/q}  <\infty  .
\end{equation*}
\end{definition}

We first recall the  following result coming from \cite{Pe75}. For $j\in \mathbb Z$ and $a>0$, define the Peetre maximal operator by 
\begin{equation*}
     (   \F ^{-1}( \varphi_j \F f) ) ^\ast_{a } (x)	: = \sup_{y\in \rn } \frac{ |\F ^{-1}( \varphi_j \F f) (x-y )| }{ (1+ 2^{j}|y| )^a }.
\end{equation*}
Then 
\begin{equation} \label{peetre inequality}
	(   \F ^{-1}( \varphi_j \F f) ) ^\ast_{a } (x) \lesssim \M _{n /a  } ( \F ^{-1}( \varphi_j \F f) ) (x).
\end{equation}

\begin{theorem}\label{peetre char} 
	Let $ 0 < \vec p \le \infty $. 
	Let  $   0 < n / ( \sum_{i=1}^n  1/p_{i})     < t <r <\infty $ or $ 0< n / ( \sum_{i=1}^n  1/p_{i})     \le t < r=\infty  $. Let $s\in \mathbb R$ and $q \in (0,\infty ]$.
	If $a > n / \min\{q,  \vec p \} $, then 
	\begin{align*}
			\| f\|_{\dot F_{\vec p , t,r} ^{s,q} } \approx \left\|  \left(  \sum_{j\in \mathbb Z} 2^{jsq}|  (   \F ^{-1}( \varphi_j \F f) ) ^\ast_{a }  |^q  \right) ^{1/q}    \right\|_{ M_{\vec p}^{t,r}}  .
	\end{align*} 
If $a > n /  \min \vec p$, then 
\begin{align*}
	\| f\|_{\dot B_{\vec p , t,r} ^{s,q} } \approx \left(  \sum_{j\in \mathbb Z} 2^{jsq}	\left\|   (   \F ^{-1}( \varphi_j \F f) ) ^\ast_{a }   \right\|_{ M_{\vec p}^{t,r}}   ^q  \right) ^{1/q}  .
\end{align*} 
\end{theorem}
\begin{proof}
	We only prove the Triebel-Lizorkin case since Besov case is similar.
	By inequality (\ref{peetre inequality}) and Theorem \ref{HL M}, we obtain
	\begin{align*}
		 &\left\|  \left(  \sum_{j\in \mathbb Z} 2^{jsq}|  (   \F ^{-1}( \varphi_j \F f) ) ^\ast_{a }  |^q  \right) ^{1/q}    \right\|_{ M_{\vec p}^{t,r}} \\
		 &  \lesssim \left\|  \left(  \sum_{j\in \mathbb Z} 2^{jsq}|  \M _{n /a  } ( \F ^{-1}( \varphi_j \F f) )  |^q  \right) ^{1/q}    \right\|_{ M_{\vec p}^{t,r}} \\
		 	 &  \lesssim 	\| f\|_{\dot F_{\vec p , t,r} ^{s,q} } 
	\end{align*}
The another inequality is trivial. Hence we finish the proof.
\end{proof}

	For  $s\in \mathbb R$,  we define
\begin{equation*}
	 H_2^s =\left\{  
	f\in \mathscr S' : \| f\| _{  H_2^s }  :=\left\|  (1+ |\cdot |^2)^{s/2}  \mathcal F f (\cdot)    \right\|     _{L^2}
	<\infty	\right\}.
\end{equation*}

	Using Theorem \ref{peetre char}, we have the following Fourier multiplier Theorem. 

	\begin{theorem}\label{multiplier}
			Let $ 0 < \vec p \le \infty $. 
		Let  $   0 < n / ( \sum_{i=1}^n  1/p_{i})     < t <r <\infty $ or $ 0< n / ( \sum_{i=1}^n  1/p_{i})     \le t < r=\infty  $. Let $s\in \mathbb R$ and $q \in (0,\infty ]$.
		Let $\{ \vec f_k\} _{k=0}^\infty $ be a sequence of function such that for each $k \in \mathbb N_0$, supp $\mathcal F \vec f_k \subset  \{\xi : |\xi |  \le 2^{k+1}\}$.
		\newline
		{\rm (i)} If  $a > n / \min\{q, \min \vec p \} $, then for any $\epsilon>0$
		\begin{equation*}
		\left\|  \left(  \sum_{j\in \mathbb Z} 2^{jsq}| \F^{-1} ( m_j \varphi_j \F  (f) )|^q  \right) ^{1/q}    \right\|_{ M_{\vec p}^{t,r}}  	\lesssim \|f\|_{\dot F_{\vec p , t,r} ^{s,q}  }  \sup_{j \in \mathbb Z}  \| (m_j(2^j \cdot)) \|_{H_2^ {a+n/2+\epsilon} }  .
		\end{equation*}
		\newline
		{\rm (ii)} If  $a > n/ \min \vec p$, then  for any $\epsilon>0$
		\begin{equation*}
			\left(  \sum_{j\in \mathbb Z} 2^{jsq}	\left\|   \F^{-1} (m_j \varphi_j \F  (f) )   \right\|_{ M_{\vec p}^{t,r}}   ^q  \right) ^{1/q}	\lesssim \|f\|_{\dot B_{\vec p , t,r} ^{s,q}  }  \sup_{j \in \mathbb Z}  \| (m_j(2^j \cdot)) \|_{H_2^ {a+n/2+\epsilon} } .
		\end{equation*}
	\end{theorem}
	\begin{proof}
		We just prove (i) since (ii) is similar. Note that
		\begin{align*}
			\left| \F^{-1} (m_j  \varphi_j \F f)  ( x)  \right|  & \le \int_\rn \left|    \mathcal F ^{-1} (  m _j)  (y)  \F^{-1} (  \varphi_j \F f)  (x- y)  \right|  \d y \\
			&\le  \sup_{y\in \rn } \frac{ |\F ^{-1}( \varphi_j \F f) (x-y )| }{ (1+ 2^{j}|y| )^a } \int_\rn \left|     \mathcal F ^{-1} (  m _j)  (y)  (1+ 2^{j}|y| )^a \right|  \d y.
		\end{align*}
		Using $ \mathcal F ^{-1} (  m _j ( 2^j \cdot)  )  (y) =  2^{-jn}  \mathcal F ^{-1}  (m _j) (2^{-j} y)  $, we obtain
		\begin{align*}
			&	\int_\rn \left|     \mathcal F ^{-1} (  m _j)  (y)  (1+ 2^{j}|y| )^a \right|  \d y  \\
			& = \int_\rn \left|       \mathcal F ^{-1} (  m _j  (2^j \cdot) )  ( y )  (1+ | y| ) ^{a+n/2+\epsilon}  (1+  |y|) ^{-n/2 -\epsilon}  \right|  \d y \\
			&\lesssim \left(  \int_\rn \left|       \mathcal F ^{-1} (  m _j (2^j \cdot) )  ( y )  (1+ | y| ) ^{a+n/2+ \epsilon} \right| ^2 \d y \right)^{1/2} \\
			&	\lesssim \| \mathcal F^{-1} \mathcal F^{-1} (m_j (2^j \cdot)) \|_{H_2^ {a+n/2+\epsilon} } \\
			& = \| (m_j(2^j \cdot)) \|_{H_2^ {a+n/2+\epsilon} } ,
		\end{align*}
		where in the last step, we use the fact that 
		\begin{equation*}
			\mathcal F^{-1} \mathcal F^{-1}  f (x)  =  f (-x).
		\end{equation*} 
		Since $a >n / \min\{q, \min \vec p \}$,  
		we obtain 
		\begin{align*}
		&	\left\|  \left(  \sum_{j\in \mathbb Z} 2^{jsq}| \F^{-1} ( m_j \varphi_j \F  (f) )|^q  \right) ^{1/q}    \right\|_{ M_{\vec p}^{t,r}} \\
		& \lesssim \left\|  \left(  \sum_{j\in \mathbb Z} 2^{jsq}|  (   \F ^{-1}( \varphi_j \F f) ) ^\ast_{a }  |^q  \right) ^{1/q}    \right\|_{ M_{\vec p}^{t,r}}  \sup_{j \in \mathbb Z}  \| (m_j(2^j \cdot)) \|_{H_2^ {a+n/2+\epsilon} }  \\
			&	\lesssim 	\left\|  \left(  \sum_{j\in \mathbb Z} 2^{jsq}| \F^{-1} ( \varphi_j \F  (f) )|^q  \right) ^{1/q}    \right\|_{ M_{\vec p}^{t,r}}  \sup_{j \in \mathbb Z}  \| (m_j(2^j \cdot)) \|_{H_2^ {a+n/2+\epsilon} } .
		\end{align*}
	Thus the proof is finished.
	\end{proof}

Next we prove the spaces are independent of $\varphi$.
\begin{theorem} \label{equivalent norms}
		Let $ 0 < \vec p \le \infty $. 
	Let  $   0 < n / ( \sum_{i=1}^n  1/p_{i})     < t <r <\infty $ or $ 0< n / ( \sum_{i=1}^n  1/p_{i})     \le t < r=\infty  $. Let $s\in \mathbb R$ and $q \in (0,\infty ]$.
	Let $\varphi, \psi$ be two partition of unity on $\rn$ as in Definition \ref{homo unity}. Then $\dot B_{\vec p , t,r} ^{s,q} $ and $ \dot F_{\vec p , t,r} ^{s,q} $ are independent of the choice of $\varphi , \psi$ in the sense
	that different choices yield equivalent quasi-norms.
\end{theorem}
\begin{proof}
Let $\varphi, \psi$ be two partition of unity on $\rn$ as in Definition \ref{homo unity}. 
	Then we have
	\begin{equation*}
		\varphi_j = \varphi_j \sum_{k=-1  }^{1} \psi_{j+k}.
	\end{equation*}
Consequently, 
\begin{equation*}
	 \F^{-1} ( \varphi_j \F  (f) ) =   \sum_{k=-1  }^{1}  \F^{-1} ( \varphi_j \psi_{j+k} \F  (f) ) .
\end{equation*}
By Theorem \ref{multiplier},
we have
\begin{align*}
	\|f\|_{\dot B_{\vec p , t,r} ^{s,q}  }^{\varphi} & =	\left(  \sum_{j\in \mathbb Z} 2^{jsq}	\left\|   \F^{-1} (\varphi_j \F  (f) )   \right\|_{ M_{\vec p}^{t,r}}   ^q  \right) ^{1/q} \\
	& = 	\left(  \sum_{j\in \mathbb Z} 2^{jsq}	\left\|   \sum_{k=-1  }^{1}  \F^{-1} ( \varphi_j \psi_{j+k} \F  (f) )  \right\|_{ M_{\vec p}^{t,r}}   ^q  \right) ^{1/q} \\
	& \lesssim \sup_{j\in \mathbb Z} \| \varphi_j(2^j \cdot) \|_{H_2^ {a+n/2+\epsilon} }  	\left(  \sum_{j\in \mathbb Z} 2^{jsq}	\left\|   \sum_{k=-1  }^{1}  \F^{-1} ( \psi_{j+k} \F  (f) )  \right\|_{ M_{\vec p}^{t,r}}   ^q  \right) ^{1/q} \\
		& =  \| \varphi \|_{H_2^ {a+n/2+\epsilon} }  	\left(  \sum_{j\in \mathbb Z} 2^{jsq}	\left\|   \sum_{k=-1  }^{1}  \F^{-1} ( \psi_{j+k} \F  (f) )  \right\|_{ M_{\vec p}^{t,r}}   ^q  \right) ^{1/q} \\
		& \lesssim \| f \|_{\dot B_{\vec p , t,r} ^{s,q}  }^{\psi} .
\end{align*}
By interchanging $\varphi$ and $\psi$, we obtain the  equivalence.
\end{proof}

The homogeneous fractional Laplacian for $f \in \mathcal S' / \mathcal P$  and $s >0$ is denoted by $D^s f \in \mathcal S' / \mathcal P $  and defined by $ \langle D^s f, \varphi \rangle  : = \langle  f, \F^{-1} (|\cdot|^s \F \varphi  ) \rangle  $ for $\varphi \in \mathscr S_0$. 

For $j \in \mathbb Z$, define $\psi_j :=  2^{-j \alpha } |\xi|^\alpha \varphi_j (\xi )  $. Then $\{\psi_j \}_{j\in \mathbb Z}$ is a partition of unity on $\rn $ modulo a positive constant $ (  \sum_{j\in \mathbb Z} \psi_j )^{-1} $.
Then by Theorem \ref{equivalent norms}, we obtain

\begin{align} \label{shift}
	\nonumber
	\| D^\alpha f \|_{ \dot F_{\vec p , t,r} ^{s,q} }& = 	\left\|  \left(  \sum_{j\in \mathbb Z} 2^{jsq}| \F^{-1} (\varphi_j | \cdot|^\alpha \F  (f) ) ) |^q  \right) ^{1/q}    \right\|_{ M_{\vec p}^{t,r}}  \\
	\nonumber
	& = 	\left\|  \left(  \sum_{j\in \mathbb Z} 2^{j ( s +\alpha ) q}| \F^{-1} (\psi_j  \F  (f) ) ) |^q  \right) ^{1/q}    \right\|_{ M_{\vec p}^{t,r}}  \\
	& \approx \|  f \|_{ \dot F_{\vec p , t,r} ^{s+\alpha,q} }.
\end{align}

Christ and Weinstein \cite{CW91} use of a fractional chain rule to study generalized Korteweg-de Vries (gKdV) equations. A fractional chain rule plays an important role in establishing well-posedness of wave equations; see \cite{K95,HJLW20,S97}.

In \cite{D25}, Douglas 	established a fractional chain rule in the  weighted Triebel-Lizorkin spaces.
The following result is the fractional chain rule in the Bourgain-Morrey Triebel-Lizorkin spaces.
 
\begin{theorem} \label{fractional chain rule}
	Let $ 0<s<1$ and $0<q<\infty$.	
	Let $ 1 < \vec p <\infty$. Let   $1 < n / ( \sum_{i=1}^n  1/p_{i})   < t <r<\infty $ or $ 1< n / ( \sum_{i=1}^n  1/p_{i})  \le t< r =\infty$. 		
		Let $1/\vec p = 1/ \vec p_1 + 1 /\vec p_2 $, $1/t = 1/t_1 + 1/t_2$ and $1/r = 1/r_1 + 1/r_2$.		
	For $j =1 ,2$,	assume that  $1 < n / ( \sum_{i=1}^n  1/p_{j,i})   < t_j <r_j<\infty $ or $ 1< n / ( \sum_{i=1}^n  1/p_{j,i})  \le t_j < r_j =\infty$ where $p_{j,i} $ is the $i$-th item of $\vec p_j$. 	
		Let $\varphi$  be a partition of unity as in Definition \ref{homo unity} and 
		\begin{equation*}
			0< n \left( \max\left\{ \frac{1}{\min \vec p_2 } , \frac{1}{q} \right\}  - \frac{1}{ (\min \vec p_1)  ' } \right) <s.
		\end{equation*}
Let $F: \mathbb C \to \mathbb C$ and $G : \mathbb C \to [0,\infty)$ such that $F(0) =0$. Suppose that for $x,y \in \mathbb C $ that 
\begin{equation*}
	| F(x) - F (y) | \le (G(x) +G(y) ) |x-y|.
\end{equation*}
Then there exists a constant $C>0$ such that 
\begin{equation*}
	\| F (u) \|_{ \dot F_{\vec p , t,r} ^{s,q} } \le C \| G (u) \|_{M_{\vec p_1} ^{t_1, r_1} }  \| u \|_{\dot F_{\vec p_2 , t_2,r_2} ^{s,q}  } .
\end{equation*}
\end{theorem}

\begin{remark}
	To the best knowledge of authors, the result  of Theorem \ref{fractional chain rule} is  new even for $\vec p =p$.
\end{remark}

The following lemma comes from the proof of \cite[Theorem 1.1]{D25}.

\begin{lemma}
		Let $\varphi$  be a partition of unity as in Definition \ref{homo unity}. Let $\Psi = \F^{-1} \varphi$ and $\Psi_j :=2^{jn} \Psi (2^j \cdot) $ for $j \in \mathbb Z$.
	Let $F: \mathbb C \to \mathbb C$ and $G : \mathbb C \to [0,\infty)$ such that $F(0) =0$. Suppose that for $x,y \in \mathbb C $ that 
	\begin{equation*}
		| F(x) - F (y) | \le (G(x) +G(y) ) |x-y|.
	\end{equation*}
	Then
	\begin{align*}
		| \F^{-1} (\varphi_j \F  (F(u) ) ) | 
		& \lesssim G(u) (x) \int_\rn |u(x) - u(y)| |\Psi_j (x-y)| \d y \\
		& \quad +   \int_\rn G(u) (y)|u(x) - u(y)| |\Psi_j (x-y)| \d y \\
		& = : I_1 +I_2.
	\end{align*}
Furthermore,
\begin{align*}
I_1	
		& \lesssim  G (u) (x) \sum_{k<j} 2^{k-j} \M_{\zeta_1} ( \F^{-1} (\varphi_k \F  u ) )  (x) \\
	& +  G (u) (x)\sum_{ k \ge j} \M_{\zeta_1} ( \F^{-1} (\varphi_k \F  u ) )  (x) \\
	& +  G (u) (x) \sum_{ k \ge j}  2^{(k-j) n (1/ \zeta_2 - 1/ \alpha ' ) } \M_{\zeta_2} ( \F^{-1} (\varphi_k \F  u ) )  (x) \\
	& =: I_{1,1} +  I_{1,2} + I_{1,3}, 
\end{align*}
and 
\begin{align*}
I_2
	& \lesssim \M (G (u) ) (x) \sum_{k<j} 2^{k-j} \M_{\zeta_1} ( \F^{-1} (\varphi_k \F  u ) )  (x) \\
	& + \M (G (u) ) (x) \sum_{ k \ge j} \M_{\zeta_1} ( \F^{-1} (\varphi_k \F  u ) )  (x) \\
	& + \M_{\alpha } (G (u) ) (x) \sum_{ k \ge j}  2^{(k-j) n (1/ \zeta_2 - 1/ \alpha ' ) } \M_{\zeta_2} ( \F^{-1} (\varphi_k \F  u ) )  (x)  \\
	& =: I_{2,1} +  I_{2,2} + I_{2,3},
\end{align*}
where $\zeta_1 \in (0, \infty)$ and $\alpha \in [1,\infty]$,  $\zeta_2 \in (0, \alpha '] $.
\end{lemma}

\begin{lemma} [Lemma 3.1, \cite{D25}] \label{Hardy seq inequ}
	Let $\{ a_k \}_{k\in \mathbb Z}$ be a sequence of non-negative real numbers, $0<q<\infty$ and $a \in \{ 0, -\infty\}$. Then if $0 \le s<1 $, 
	\begin{equation*}
		\left( \sum_{j>a} \left( 2^{js}  \sum_{a<k <j} 2^{k-j}  a_k  \right)^{q}   \right) ^{1/q}  \lesssim \left( \sum_{j>a}  ( 2^{js}   a_j )^{q}   \right) ^{1/q};
	\end{equation*}
if $0\le \theta < s$, 
\begin{equation*}
	\left( \sum_{j>a} \left( 2^{js}  \sum_{j \le k } 2^{(k-j) \theta }  a_k  \right)^{q}   \right) ^{1/q}  \lesssim \left( \sum_{j>a}  ( 2^{js}   a_j )^{q}   \right) ^{1/q}.
\end{equation*}
\end{lemma}

\begin{proof}[Proof of Theorem \ref{fractional chain rule}]
	Let $1/\vec p = 1/ \vec p_1 + 1 /\vec p_2 $, $1/t = 1/t_1 + 1/t_2$ and $1/r = 1/r_1 + 1/r_2$. We only estimate $I_2$ since the estimate of $I_1$ is similar.
	
	Let $1<\vec p_1 <\infty $. Let $\zeta_1 \in (0,\min\{ q, \vec p_2 \}  ) $.
	
	Estimate of $I_{2,1}$.
	Applying the Triebel-Lizorkin Bourgain-Morrey norm,  the H\"older inequality and Lemma \ref{Hardy seq inequ}, we obtain
	\begin{align*}
 & 	\left\|  \left(  \sum_{j\in \mathbb Z} 2^{jsq}|\M (G (u) )  \sum_{k<j} 2^{k-j} \M_{\zeta_1} ( \F^{-1} (\varphi_k \F  u ) ) |^q  \right) ^{1/q}    \right\|_{ M_{\vec p}^{t,r}} \\
	& \le \| \M (G (u) ) \|_{M_{\vec p_1}^{t_1,r_1}   } 	\left\|  \left(  \sum_{j\in \mathbb Z} 2^{jsq}|   \sum_{k<j} 2^{k-j} \M_{\zeta_1} ( \F^{-1} (\varphi_k \F  u ) ) |^q  \right) ^{1/q}    \right\|_{ M_{\vec p_2}^{t_2,r_2}} \\
	& \lesssim \| G (u ) \|_{M_{\vec p_1}^{t_1,r_1}   } 	\left\|  \left(  \sum_{j\in \mathbb Z} 2^{jsq}|    \M_{\zeta_1} ( \F^{-1} (\varphi_j \F  u ) ) |^q  \right) ^{1/q}    \right\|_{ M_{\vec p_2}^{t_2,r_2}} \\
		& \lesssim \|G (u)  \|_{M_{\vec p_1}^{t_1,r_1}   } 	\left\|  \left(  \sum_{j\in \mathbb Z} 2^{jsq}|     ( \F^{-1} (\varphi_j \F  u ) ) |^q  \right) ^{1/q}    \right\|_{ M_{\vec p_2}^{t_2,r_2}} .
	\end{align*}
Estimate of $I_{2,2}$. By Lemma \ref{Hardy seq inequ}, we have
	\begin{align*}
	& 	\left\|  \left(  \sum_{j\in \mathbb Z} 2^{jsq}| \M (G (u) ) \sum_{ k \ge j} \M_{\zeta_1} ( \F^{-1} (\varphi_k \F  u ) ) |^q  \right) ^{1/q}    \right\|_{ M_{\vec p}^{t,r}} \\
	& \le \| \M (G (u) ) \|_{M_{\vec p_1}^{t_1,r_1}   } 	\left\|  \left(  \sum_{j\in \mathbb Z} 2^{jsq}|   \sum_{ k \ge j} \M_{\zeta_1} ( \F^{-1} (\varphi_k \F  u ) ) |^q  \right) ^{1/q}    \right\|_{ M_{\vec p_2}^{t_2,r_2}} \\
	& \lesssim \| G(u) \|_{M_{\vec p_1}^{t_1,r_1}   } 	\left\|  \left(  \sum_{j\in \mathbb Z} 2^{jsq}|    \M_{\zeta_1} ( \F^{-1} (\varphi_j \F  u ) ) |^q  \right) ^{1/q}    \right\|_{ M_{\vec p_2}^{t_2,r_2}} \\
	& \lesssim \| G(u) \|_{M_{\vec p_1}^{t_1,r_1}   } 	\left\|  \left(  \sum_{j\in \mathbb Z} 2^{jsq}|     ( \F^{-1} (\varphi_j \F  u ) ) |^q  \right) ^{1/q}    \right\|_{ M_{\vec p_2}^{t_2,r_2}} .
\end{align*}

Estimate of $I_{2,3}$. Note that we suppose $\displaystyle  0< n \left( \max\left\{ \frac{1}{\min \vec p_2 } , \frac{1}{q} \right\}  - \frac{1}{ (\min \vec p_1)  ' } \right) <s $.
Let $\alpha  = \min \vec p_1 -\epsilon_1 >1.$ Let $\zeta_2 = \min\{ q, \min \vec p_2 \} -\epsilon_2>0 $. Note that $ 0< \zeta_2 < \alpha ' $.
Hence $0\le n (1/ \zeta_2 - 1/ \alpha ' )  <s <1$ if choosing suitable numbers $\epsilon_1 , \epsilon_2 >0$ above. Then by Lemma \ref{Hardy seq inequ}, we obtain
	\begin{align*}
	& 	\left\|  \left(  \sum_{j\in \mathbb Z} 2^{jsq}|  \M_\alpha (G (u) )  \sum_{ k \ge j}  2^{(k-j) n (1/ \zeta_2 - 1/ \alpha ' ) } \M_{\zeta_2} ( \F^{-1} (\varphi_k \F  u ) )   |^q  \right) ^{1/q}    \right\|_{ M_{\vec p}^{t,r}} \\
	& \le \|\M_\alpha (G (u) )  \|_{M_{\vec p_1}^{t_1,r_1}   } 	\left\|  \left(  \sum_{j\in \mathbb Z} 2^{jsq}| \sum_{ k \ge j}  2^{(k-j) n (1/ \zeta_2 - 1/ \alpha ' ) } \M_{\zeta_2} ( \F^{-1} (\varphi_k \F  u ) )  |^q  \right) ^{1/q}    \right\|_{ M_{\vec p_2}^{t_2,r_2}} \\
	& \lesssim \| G(u) \|_{M_{\vec p_1}^{t_1,r_1}   } 	\left\|  \left(  \sum_{j\in \mathbb Z} 2^{jsq}|     \M_{\zeta_2} ( \F^{-1} (\varphi_j \F  u ) )  |^q  \right) ^{1/q}    \right\|_{ M_{\vec p_2}^{t_2,r_2}} \\
	& \lesssim \| G(u) \|_{M_{\vec p_1}^{t_1,r_1}   } 	\left\|  \left(  \sum_{j\in \mathbb Z} 2^{jsq}|     ( \F^{-1} (\varphi_j \F  u ) ) |^q  \right) ^{1/q}    \right\|_{ M_{\vec p_2}^{t_2,r_2}} .
\end{align*}
Thus we finish the proof.
\end{proof}

By Theorems \ref{fractional chain rule}, \ref{LP char mixed BM} and equation \ref{shift}, we have the following result.
\begin{corollary}
		Let $ 0<s<1$.	
	Let $ 1 < \vec p <\infty$. Let   $1 < n / ( \sum_{i=1}^n  1/p_{i})   < t <r<\infty $ or $ 1< n / ( \sum_{i=1}^n  1/p_{i})  \le t< r =\infty$. 		
	Let $1/\vec p = 1/ \vec p_1 + 1 /\vec p_2 $, $1/t = 1/t_1 + 1/t_2$ and $1/r = 1/r_1 + 1/r_2$.		
	For $j =1 ,2$,	assume that  $1 < n / ( \sum_{i=1}^n  1/p_{j,i})   < t_j <r_j<\infty $ or $ 1< n / ( \sum_{i=1}^n  1/p_{j,i})  \le t_j < r_j =\infty$ where $p_{j,i} $ is the $i$-th item of $\vec p_j$. 
	Let $\varphi$  be a partition of unity as in Definition \ref{homo unity} and 
\begin{equation*}
	0< n \left( \max\left\{ \frac{1}{\min \vec p_2 } , \frac{1}{2} \right\}  - \frac{1}{ (\min \vec p_1)  ' } \right) <s.
\end{equation*}
	Let $F: \mathbb C \to \mathbb C$ and $G : \mathbb C \to [0,\infty)$ such that $F(0) =0$. Suppose that for $x,y \in \mathbb C $ that 
	\begin{equation*}
		| F(x) - F (y) | \le (G(x) +G(y) ) |x-y|.
	\end{equation*}
	Then there exists a constant $C>0$ such that 
	\begin{equation*}
		\|  D ^s F (u) \|_{ M_{\vec p} ^{t,r} } \le C \| G (u) \|_{M_{\vec p_1} ^{t_1, r_1} }  \| D^s u \|_{M_{\vec p_2} ^{t_2,r_2}  } .
	\end{equation*}
\end{corollary}

\section*{Ethical Approval}

No applicable for both human and/ or animal studies.

\section*{Competing interests}

The authors have no competing interests to declare that are relevant to the content of this article.

\section*{Authors' contributions}

T. Bai and J. Xu wrote the draft. All authors reviewed the manuscript.


\section*{Availability of data and materials} No data and materials was used for the research described in the article.

\newcommand{\etalchar}[1]{$^{#1}$}
\providecommand{\bysame}{\leavevmode\hbox to3em{\hrulefill}\thinspace}
\providecommand{\MR}{\relax\ifhmode\unskip\space\fi MR }
\providecommand{\MRhref}[2]{%
	\href{http://www.ams.org/mathscinet-getitem?mr=#1}{#2}
}
\providecommand{\href}[2]{#2}


\begin{thebibliography}{HvNVW16}
	
	\bibitem[ABHN11]{ABHN11}
	Wolfgang Arendt, Charles J.~K. Batty, Matthias Hieber, and Frank Neubrander,
	\emph{Vector-valued {Laplace} transforms and {Cauchy} problems}, 2nd ed. ed.,
	Monogr. Math., Basel, vol.~96, Basel: Birkh{\"a}user, 2011 (English).
	
	\bibitem[AIV19]{AIV19}
	Nenad Antoni{\'c}, Ivan Ivec, and Ivana Vojnovi{\'c}, \emph{Continuity of
		pseudodifferential operators on mixed-norm {Lebesgue} spaces}, Monatsh. Math.
	\textbf{190} (2019), no.~4, 657--674 (English).
	
	\bibitem[Bag75]{B75}
	Richard~J. Bagby, \emph{An extended inequality for the maximal function}, Proc.
	Am. Math. Soc. \textbf{48} (1975), 419--422 (English).
	
	\bibitem[BGX25]{BGX25}
	Tengfei Bai, Pengfei Guo, and Jingshi Xu, \emph{The preduals of {Banach} space
		valued {Bourgain}-{Morrey} spaces}, Ann. Funct. Anal. \textbf{16} (2025),
	no.~4, 40 (English), Id/No 64.
	
	\bibitem[BIN78]{BIN78}
	Oleg~V. Besov, Valentin~P. Il'in, and Sergei~M. Nikol'skii, \emph{Integral
		representations of functions and imbedding theorems. {Vol}. {I}. {Ed}. by
		{Mitchell} {H}. {Taibleson}. {Translation} from the {Russian}}, Scripta
	{Series} in {Mathematics}. {Washington}, {D}.{C}.: {V}. {H}. {Winston} \&
	{Sons}. {New} {York} etc.: {John} {Wiley} \& {Sons}. {VIII}, 345 p. \$ 26.50;
	{{\textsterling}} 14.00 (1978)., 1978.
	
	\bibitem[BIN79]{BIN79}
	\bysame, \emph{Integral representations of functions and imbedding theorems.
		{Vol}. {II}. {Ed}. by {Mitchell} {H}. {Taiblesson}. {Translation} from the
		{Russian}}, Scripta {Series} in {Mathematics}. {Washington}, {D}.{C}.:
	{V}.{H}. {Winston} \& {Sons}. {New} {York} etc.: {John} {Wiley} \& {Sons}
	{VIII}, 311 p. \$ 26.50; {{\textsterling}} 13.50 (1979)., 1979.
	
	\bibitem[BIN96]{BIN96}
	O.~V. Besov, V.~P. Il'in, and S.~M. Nikol'ski{\u{\i}}, \emph{Integral
		representations of functions, and embedding theorems}, 2nd ed., rev. and
	compl. ed., Moskva: Nauka. Fizmatlit, 1996 (Russian).
	
	\bibitem[BKPS06]{BKPS06}
	Sorina Barza, Anna Kami{\'n}ska, Lars-Erik Persson, and Javier Soria,
	\emph{Mixed norm and multidimensional {Lorentz} spaces}, Positivity
	\textbf{10} (2006), no.~3, 539--554 (English).
	
	\bibitem[Blo81]{Bl81}
	A.~P. Blozinski, \emph{Multivariate rearrangements and {Banach} function spaces
		with mixed norms}, Trans. Am. Math. Soc. \textbf{263} (1981), 149--167
	(English).
	
	\bibitem[Bou91]{Bou91}
	Jean Bourgain, \emph{On the restriction and multiplier problems in {{\(R^
				3\)}}}, Geometric aspects of functional analysis. Proceedings of the Israel
	seminar (GAFA) 1989-90, Berlin etc.: Springer-Verlag, 1991, pp.~179--191
	(English).
	
	\bibitem[BP61]{BP61}
	A.~Benedek and R.~Panzone, \emph{The spaces {{\({L}^ p\)}}, with mixed norm},
	Duke Math. J. \textbf{28} (1961), 301--324 (English).
	
	\bibitem[BX25a]{BX252}
	Tengfei Bai and Jingshi Xu, \emph{Precompactness in matrix weighted
		{Bourgain}-{Morrey} spaces}, Filomat \textbf{39} (2025), no.~18, 6261--6280
	(English).
	
	\bibitem[BX25b]{BX25}
	\bysame, \emph{Weighted {Bourgain}-{Morrey}-{Besov}-{Triebel}-{Lizorkin} spaces
		associated with operators}, Math. Nachr. \textbf{298} (2025), no.~3, 886--924
	(English).
	
	\bibitem[CG20]{CG20}
	Galatia Cleanthous and Athanasios~G. Georgiadis, \emph{Mixed-norm
		{{\(\alpha\)}}-modulation spaces}, Trans. Am. Math. Soc. \textbf{373} (2020),
	no.~5, 3323--3356 (English).
	
	\bibitem[CGN17]{CGN17}
	G.~Cleanthous, A.~G. Georgiadis, and M.~Nielsen, \emph{Anisotropic mixed-norm
		{Hardy} spaces}, J. Geom. Anal. \textbf{27} (2017), no.~4, 2758--2787
	(English).
	
	\bibitem[CGN19]{CGN19}
	Galatia Cleanthous, Athanasios~G. Georgiadis, and Morten Nielsen,
	\emph{Molecular decomposition of anisotropic homogeneous mixed-norm spaces
		with applications to the boundedness of operators}, Appl. Comput. Harmon.
	Anal. \textbf{47} (2019), no.~2, 447--480 (English).
	
	\bibitem[Cia13]{C13}
	Philippe~G. Ciarlet, \emph{Linear and nonlinear functional analysis with
		applications. {With} 401 problems and 52 figures}, Other Titles Appl. Math.,
	vol. 130, Philadelphia, PA: Society for Industrial {and} Applied Mathematics
	(SIAM)., 2013 (English).
	
	\bibitem[CS22]{CS22}
	Ting Chen and Wenchang Sun, \emph{Hardy-littlewood-sobolev inequality on
		mixed-norm {Lebesgue} spaces}, J. Geom. Anal. \textbf{32} (2022), no.~3, 43
	(English), Id/No 101.
	
	\bibitem[CW91]{CW91}
	F.~M. Christ and M.~I. Weinstein, \emph{Dispersion of small amplitude solutions
		of the generalized {Korteweg}-de {Vries} equation}, J. Funct. Anal.
	\textbf{100} (1991), no.~1, 87--109 (English).
	
	\bibitem[dCFMN21]{CFMN21}
	Ricardo del Campo, Antonio Fern{\'a}ndez, Fernando Mayoral, and Francisco
	Naranjo, \emph{Orlicz spaces associated to a quasi-{Banach} function space:
		applications to vector measures and interpolation}, Collect. Math.
	\textbf{72} (2021), no.~3, 481--499 (English).
	
	\bibitem[DGZ24]{DGZ24}
	Wei Ding, Min Gu, and YuePing Zhu, \emph{The boundedness of operators on
		weighted multi-parameter mixed {Hardy} spaces}, Math. Nachr. \textbf{297}
	(2024), no.~7, 2730--2749 (English).
	
	\bibitem[Dou25]{D25}
	Sean Douglas, \emph{Chain rule for weighted {Triebel}-{Lizorkin} spaces}, J.
	Funct. Anal. \textbf{289} (2025), no.~12, 19 (English), Id/No 111165.
	
	\bibitem[DZ23]{DZ23}
	Wei Ding and Fangli Zou, \emph{Atomic decomposition of weighted multi-parameter
		mixed {Hardy} spaces}, Bull. Malays. Math. Sci. Soc. (2) \textbf{46} (2023),
	no.~6, 20 (English), Id/No 182.
	
	\bibitem[Fer77]{Fe77}
	Dicesar~Lass Fernandez, \emph{Lorentz spaces, with mixed norms}, J. Funct.
	Anal. \textbf{25} (1977), 128--146 (English).
	
	\bibitem[FJW91]{FJW91}
	Michael Frazier, Bj{\"o}rn Jawerth, and Guido Weiss, \emph{Littlewood-paley
		theory and the study of function spaces}, Reg. Conf. Ser. Math., vol.~79,
	Providence, RI: American Mathematical Society, 1991 (English).
	
	\bibitem[GN16]{GN16}
	A.~G. Georgiadis and M.~Nielsen, \emph{Pseudodifferential operators on
		mixed-norm {Besov} and {Triebel}-{Lizorkin} spaces}, Math. Nachr.
	\textbf{289} (2016), no.~16, 2019--2036 (English).
	
	\bibitem[Gra14a]{G14}
	Loukas Grafakos, \emph{Classical {Fourier} analysis}, 3rd ed. ed., Grad. Texts
	Math., vol. 249, New York, NY: Springer, 2014 (English).
	
	\bibitem[Gra14b]{G142}
	\bysame, \emph{Modern {Fourier} analysis}, 3rd ed. ed., Grad. Texts Math., vol.
	250, New York, NY: Springer, 2014 (English).
	
	\bibitem[HCY21]{HCY21}
	Long Huang, Der-Chen Chang, and Dachun Yang, \emph{Fourier transform of
		anisotropic mixed-norm {Hardy} spaces}, Front. Math. China \textbf{16}
	(2021), no.~1, 119--139 (English).
	
	\bibitem[HJLW20]{HJLW20}
	Kunio Hidano, Jin-Cheng Jiang, Sanghyuk Lee, and Chengbo Wang, \emph{Weighted
		fractional chain rule and nonlinear wave equations with minimal regularity},
	Rev. Mat. Iberoam. \textbf{36} (2020), no.~2, 341--356 (English).
	
	\bibitem[HLY23]{HLY23}
	Pingxu Hu, Yinqin Li, and Dachun Yang, \emph{Bourgain-{Morrey} spaces meet
		structure of {Triebel}-{Lizorkin} spaces}, Math. Z. \textbf{304} (2023),
	no.~1, 49 (English), Id/No 19.
	
	\bibitem[HLYY20]{HLYY20}
	Long Huang, Jun Liu, Dachun Yang, and Wen Yuan, \emph{Real-variable
		characterizations of new anisotropic mixed-norm {Hardy} spaces}, Commun. Pure
	Appl. Anal. \textbf{19} (2020), no.~6, 3033--3082 (English).
	
	\bibitem[HNSH23]{HNSH23}
	Naoya Hatano, Toru Nogayama, Yoshihiro Sawano, and Denny~Ivanal Hakim,
	\emph{Bourgain-{Morrey} spaces and their applications to boundedness of
		operators}, J. Funct. Anal. \textbf{284} (2023), no.~1, 52 (English), Id/No
	109720.
	
	\bibitem[Ho16]{Ho16}
	Kwok-Pun Ho, \emph{Strong maximal operator on mixed-norm spaces}, Ann. Univ.
	Ferrara, Sez. VII, Sci. Mat. \textbf{62} (2016), no.~2, 275--291 (English).
	
	\bibitem[Ho18]{Ho18}
	\bysame, \emph{Mixed norm {Lebesgue} spaces with variable exponents and
		applications}, Riv. Mat. Univ. Parma (N.S.) \textbf{9} (2018), no.~1, 21--44
	(English).
	
	\bibitem[Ho21]{Ho21}
	\bysame, \emph{Extrapolation to mixed norm spaces and applications}, Acta
	Comment. Univ. Tartu. Math. \textbf{25} (2021), no.~2, 281--296 (English).
	
	\bibitem[Hua23]{Hu23}
	Long Huang, \emph{The {Mihlin} multiplier theorem on anisotropic mixed-norm
		{Hardy} spaces}, Bull. Malays. Math. Sci. Soc. (2) \textbf{46} (2023), no.~4,
	11 (English), Id/No 129.
	
	\bibitem[HvNVW16]{HNVW16}
	Tuomas Hyt{\"o}nen, Jan van Neerven, Mark Veraar, and Lutz Weis, \emph{Analysis
		in {Banach} spaces. {Volume} {I}. {Martingales} and {Littlewood}-{Paley}
		theory}, Ergeb. Math. Grenzgeb., 3. Folge, vol.~63, Cham: Springer, 2016
	(English).
	
	\bibitem[HY21]{HY21}
	Long Huang and Dachun Yang, \emph{On function spaces with mixed norms: a
		survey}, J. Math. Study \textbf{54} (2021), no.~3, 262--336 (English).
	
	\bibitem[HYY21]{HYY21}
	Long Huang, Dachun Yang, and Wen Yuan, \emph{Anisotropic mixed-norm
		{Campanato}-type spaces with applications to duals of anisotropic mixed-norm
		{Hardy} spaces}, Banach J. Math. Anal. \textbf{15} (2021), no.~4, 36
	(English), Id/No 62.
	
	\bibitem[Iga86]{Ig86}
	Satoru Igari, \emph{Interpolation of operators in {Lebesgue} spaces with mixed
		norm and its applications to {Fourier} analysis}, T{\^o}hoku Math. J. (2)
	\textbf{38} (1986), no.~3, 469--490 (English).
	
	\bibitem[IST15]{IST15}
	Takashi Izumi, Yoshihiro Sawano, and Hitoshi Tanaka, \emph{Littlewood-paley
		theory for {Morrey} spaces and their preduals}, Rev. Mat. Complut.
	\textbf{28} (2015), no.~2, 411--447 (English).
	
	\bibitem[JS08]{JS08}
	Jon Johnsen and Winfried Sickel, \emph{On the trace problem for
		{Lizorkin}--{Triebel} spaces with mixed norms}, Math. Nachr. \textbf{281}
	(2008), no.~5, 669--696 (English).
	
	\bibitem[Kat95]{K95}
	Tosio Kato, \emph{On nonlinear {Schr{\"o}dinger} equations. {II}: {{\({H}^
				{S}\)}}-solutions and unconditional well-posedness}, J. Anal. Math.
	\textbf{67} (1995), 281--306 (English).
	
	\bibitem[Kol08]{Ko08}
	Jerry~J. Koliha, \emph{Metrics, norms and integrals. {An} introduction to
		contemporary analysis}, Hackensack, NJ: World Scientific, 2008 (English).
	
	\bibitem[Li22]{Li22}
	Nan Li, \emph{Summability in anisotropic mixed-norm {Hardy} spaces}, Electron.
	Res. Arch. \textbf{30} (2022), no.~9, 3362--3376 (English).
	
	\bibitem[LN19]{LN19}
	Andrei~K. Lerner and Fedor Nazarov, \emph{Intuitive dyadic calculus: the
		basics}, Expo. Math. \textbf{37} (2019), no.~3, 225--265 (English).
	
	\bibitem[LRM91]{LPM91}
	Pierre-Gilles Lemari{\'e}-Rieusset and G{\'e}rard Malgouyres, \emph{Support des
		fonctions de base dans une analyse multi-r{\'e}solution. ({On} the support of
		the scaling function in a multi-resolution analysis)}, C. R. Acad. Sci.,
	Paris, S{\'e}r. I \textbf{313} (1991), no.~6, 377--380 (French).
	
	\bibitem[Mas16]{M16}
	Satoshi Masaki, \emph{Two minimization problems on non-scattering solutions to
		mass-subcritical nonlinear {Schr{\"o}dinger} equation}, Preprint,
	{arXiv}:1605.09234 [math.{AP}] (2016), 2016.
	
	\bibitem[Maz03]{Ma03}
	Anna~L. Mazzucato, \emph{Besov-morrey spaces: {Function} space theory and
		applications to nonlinear {PDE}}, Trans. Am. Math. Soc. \textbf{355} (2003),
	no.~4, 1297--1364 (English).
	
	\bibitem[Mil81]{Mi81}
	Mario Milman, \emph{On interpolation of {{\(2^ n \)}}{Banach} spaces and
		{Lorentz} spaces with mixed norms}, J. Funct. Anal. \textbf{41} (1981), 1--7
	(English).
	
	\bibitem[Mor38]{Mo38}
	Charles B.~jun. Morrey, \emph{On the solutions of quasi-linear elliptic partial
		differential equations}, Trans. Am. Math. Soc. \textbf{43} (1938), 126--166
	(English).
	
	\bibitem[Nog19a]{N192}
	Toru Nogayama, \emph{Boundedness of commutators of fractional integral
		operators on mixed {Morrey} spaces}, Integral Transforms Spec. Funct.
	\textbf{30} (2019), no.~10, 790--816 (English).
	
	\bibitem[Nog19b]{N19}
	\bysame, \emph{Mixed {Morrey} spaces}, Positivity \textbf{23} (2019), no.~4,
	961--1000 (English).
	
	\bibitem[Nog24]{N24}
	\bysame, \emph{Littlewood-paley and wavelet characterization for mixed {Morrey}
		spaces}, Math. Nachr. \textbf{297} (2024), no.~6, 2198--2233 (English).
	
	\bibitem[NS22]{NS22}
	Toru Nogayama and Yoshihiro Sawano, \emph{Local and global solvability for
		{Keller}-{Segel} system in {Besov}-{Morrey} spaces}, J. Math. Anal. Appl.
	\textbf{516} (2022), no.~1, 25 (English), Id/No 126508.
	
	\bibitem[Pee75]{Pe75}
	Jaak Peetre, \emph{On spaces of {Triebel}-{Lizorkin} type}, Ark. Mat.
	\textbf{13} (1975), 123--130 (English).
	
	\bibitem[Ros13]{R13}
	Marcel Rosenthal, \emph{Local means, wavelet bases and wavelet isomorphisms in
		{Besov}-{Morrey} and {Triebel}-{Lizorkin}-{Morrey} spaces}, Math. Nachr.
	\textbf{286} (2013), no.~1, 59--87 (English).
	
	\bibitem[RS72]{RS72}
	Michael Reed and Barry Simon, \emph{Methods of modern mathematical physics. 1:
		{Functional} analysis}, New {York}-{London}: {Academic} {Press}, {Inc}. xvii,
	325 p. \$ 12.50 (1972)., 1972.
	
	\bibitem[Saw10]{S10}
	Yoshihiro Sawano, \emph{Br{\'e}zis--gallou{\"e}t--wainger type inequality for
		{Besov}--{Morrey} spaces}, Stud. Math. \textbf{196} (2010), no.~1, 91--101
	(English).
	
	\bibitem[SDFH20a]{SDH20}
	Yoshihiro Sawano, Giuseppe Di~Fazio, and Denny~Ivanal Hakim, \emph{Morrey
		spaces. {Introduction} and applications to integral operators and {PDE}'s.
		{Volume} {I}}, Monogr. Res. Notes Math., Boca Raton, FL: CRC Press, 2020
	(English).
	
	\bibitem[SDFH20b]{SDH202}
	\bysame, \emph{Morrey spaces. {Introduction} and applications to integral
		operators and {PDE}'s. {Volume} {II}}, Monogr. Res. Notes Math., Boca Raton,
	FL: CRC Press, 2020 (English).
	
	\bibitem[ST09]{ST09}
	Yoshihiro Sawano and Hitoshi Tanaka, \emph{Predual spaces of {Morrey} spaces
		with non-doubling measures}, Tokyo J. Math. \textbf{32} (2009), no.~2,
	471--486 (English).
	
	\bibitem[Sta97]{S97}
	Gigliola Staffilani, \emph{On the generalized {Korteweg}-de {Vries} type
		equations}, Differ. Integral Equ. \textbf{10} (1997), no.~4, 777--796
	(English).
	
	\bibitem[SW21]{SW21}
	Krist{\'o}f Szarvas and Ferenc Weisz, \emph{Mixed martingale {Hardy} spaces},
	J. Geom. Anal. \textbf{31} (2021), no.~4, 3863--3888 (English).
	
	\bibitem[TW15]{TW15}
	Rodolfo~H. Torres and Erika~L. Ward, \emph{Leibniz's rule, sampling and
		wavelets on mixed {Lebesgue} spaces}, J. Fourier Anal. Appl. \textbf{21}
	(2015), no.~5, 1053--1076 (English).
	
	\bibitem[WYY22]{WYY22}
	Suqing Wu, Dachun Yang, and Wen Yuan, \emph{Interpolations of mixed-norm
		function spaces}, Bull. Malays. Math. Sci. Soc. (2) \textbf{45} (2022),
	no.~1, 153--175 (English).
	
	\bibitem[ZST{\etalchar{+}}23]{ZSTYY23}
	Yirui Zhao, Yoshihiro Sawano, Jin Tao, Dachun Yang, and Wen Yuan,
	\emph{Bourgain-{Morrey} spaces mixed with structure of {Besov} spaces}, Proc.
	Steklov Inst. Math. \textbf{323} (2023), 244--295 (English).
	
	\bibitem[ZX20]{ZX20}
	Juan Zhang and Qingying Xue, \emph{Multilinear strong maximal operators on
		weighted mixed norm spaces}, Publ. Math. Debr. \textbf{96} (2020), no.~3-4,
	347--361 (English).
	
	\bibitem[ZYY24]{ZYY242}
	Chenfeng Zhu, Dachun Yang, and Wen Yuan,
	\emph{Bourgain-{Brezis}-{Mironescu}-type characterization of inhomogeneous
		ball {Banach} {Sobolev} spaces on extension domains}, J. Geom. Anal.
	\textbf{34} (2024), no.~10, 70 (English), Id/No 295.
	
	\bibitem[ZYZ24]{ZYZ24}
	Yijin Zhang, Dachun Yang, and Yirui Zhao, \emph{Grand
		{Besov}-{Bourgain}-{Morrey} spaces and their applications to boundedness of
		operators}, Anal. Math. Phys. \textbf{14} (2024), no.~4, 58 (English), Id/No
	79.
	
\end{thebibliography}
\end{document}